\documentclass[11pt,reqno]{amsart}
\usepackage{amsmath,amssymb,amsthm,amsfonts,amsopn,todonotes}
\usepackage[noadjust]{cite}
\usepackage{dsfont}
\usepackage{enumitem}
\usepackage[utf8]{inputenc}
\usepackage{mathrsfs}

\setlist[itemize]{leftmargin=0.8cm}

\usepackage{color,url,setspace,wasysym}
\usepackage{lscape}
\usepackage{geometry}
\usepackage[unicode,pdfencoding=unicode,pdfusetitle,
 bookmarks=true,bookmarksnumbered=false,bookmarksopen=false,
 breaklinks=false,pdfborder={0 0 1},backref=page,colorlinks=false]
 {hyperref}

\newcommand{\BBm}{\mathcal B_m}

\usepackage[ddmmyyyy,hhmmss]{datetime}
\newdateformat{daymonthyeardate}{%
  \THEDAY.\THEMONTH.\THEYEAR}
\usepackage{fancyhdr}
\usepackage{lastpage}
\usepackage{etoolbox}
\def\felixleftmark{}
\def\felixrightmark{}
\newcommand\updatemark{\markleft{\felixleftmark~\ifdefempty{\felixrightmark}{}{--- \felixrightmark}}}
\let\origsection\section
\renewcommand{\section}[1]{\origsection{#1}\sectionmark{#1}}
\renewcommand\sectionmark[1]{\def\temp{#1}\def\felixrightmark{}\edef\x{\noexpand\def\noexpand\felixleftmark{Section \thesection: \expandonce{\temp}}}\x\updatemark}
\let\origsubsection\subsection
\renewcommand{\subsection}[1]{\origsubsection{#1}\subsectionmark{#1}}

\theoremstyle{plain}
\newtheorem{theorem}{Theorem}[section]
\newtheorem{proposition}[theorem]{Proposition}
\newtheorem{lemma}[theorem]{Lemma}
\newtheorem{corollary}[theorem]{Corollary}
\theoremstyle{definition}
\newtheorem{definition}[theorem]{Definition}

\theoremstyle{remark}
\newtheorem{remark}[theorem]{Remark}
\newtheorem{example}[theorem]{Example}
\theoremstyle{remark}
\newtheorem*{rem*}{Remark}
\theoremstyle{plain}
\newtheorem*{thm*}{Theorem}

\numberwithin{equation}{section}

\newcommand{\bd}{\mathbf}

\newcommand{\AAi}{\mathcal A}
\newcommand{\AAm}{\mathcal A_m}

\newcommand{\AAvA}{{\mathcal A_{v,\BanachOne}}}

\newcommand{\BBi}{\mathcal B}

\newcommand{\UU}{\mathcal U}

\newcommand{\SumSpaceSymbol}{\mathscr{H}}
\newcommand{\SumSpace}{\lebesgue^1 + \lebesgue^\infty + \lebesgue^{1,\infty} + \lebesgue^{\infty,1}}
\newcommand{\IntersectionNorm}[1]{\max \big\{ \|#1\|_{\lebesgue^1} , \|#1\|_{\lebesgue^\infty} , \|#1\|_{\lebesgue^{1,\infty}} , \|#1\|_{\lebesgue^{\infty,1}} \big\}}
\newcommand{\IntersectionSpace}{\lebesgue^1 \cap \lebesgue^\infty \cap \lebesgue^{1,\infty} \cap \lebesgue^{\infty,1}}
\newcommand{\IntersectionSpaceSymbol}{\mathscr{G}}

\newcommand{\R}{\mathbb{R}}

\newcommand{\Z}{\mathbb{Z}}

\newcommand{\N}{\mathbb{N}}

\newcommand{\QQ}{\mathbb{Q}}
\newcommand{\CC}{\mathbb{C}}

\newcommand{\MaxKernel}[1]{M_{#1}}
\newcommand{\Reservoir}{\mathscr{R}}
\newcommand{\Bounded}{\mathscr{B}}
\newcommand{\BanachOne}{\mathbf{A}}
\newcommand{\BanachTwo}{\mathbf{B}}
\newcommand{\BanachThree}{\mathbf{C}}
\newcommand{\XTuple}{(X,\CalF,\mu)}
\newcommand{\XIndexTuple}[1]{(X_{#1}, \CalF_{#1}, \mu_{#1})}
\newcommand{\YTuple}{(Y,\CalG,\nu)}
\newcommand{\YIndexTuple}[1]{(Y_{#1}, \CalG_{#1}, \nu_{#1})}
\newcommand{\ProductTupleX}{(X_1 \times X_2, \CalF_1 \otimes \CalF_2, \mu_1 \otimes \mu_2)}
\newcommand{\ProductTupleXCompressed}{(X_1 \! \times \! X_2, \CalF_1 \! \otimes \! \CalF_2, \mu_1 \! \otimes \! \mu_2)}
\newcommand{\ProductTupleY}{(Y_1 \times Y_2, \CalG_1 \otimes \CalG_2, \nu_1 \otimes \nu_2)}
\newcommand{\ProductTupleYCompressed}{(Y_1 \! \times \! Y_2, \CalG_1 \! \otimes \! \CalG_2, \nu_1 \! \otimes \! \nu_2)}
\newcommand{\GoodFunctions}{\mathscr{G}}
\newcommand{\KernelProduct}{\odot}

\newcommand{\vertiii}[1]{{\left\vert \kern-0.25ex  \left\vert \kern-0.25ex  \left\vert #1\right\vert \kern-0.25ex  \right\vert \kern-0.25ex  \right\vert}}

\def\esssup{\mathop{\operatorname{ess~sup}}}

\def\supp{\mathop{\operatorname{supp}}}

\makeatletter
\newcommand{\LeftEqNo}{\let\veqno\@@leqno}

\DeclareFontFamily{U}{mathx}{\hyphenchar\font45}
\DeclareFontShape{U}{mathx}{m}{n}{
      <5> <6> <7> <8> <9> <10>
      <10.95> <12> <14.4> <17.28> <20.74> <24.88>
      mathx10
      }{}
\DeclareSymbolFont{mathx}{U}{mathx}{m}{n}
\DeclareFontSubstitution{U}{mathx}{m}{n}
\DeclareMathAccent{\widecheck}{0}{mathx}{"71}
\DeclareMathAccent{\wideparen}{0}{mathx}{"75}

\newcommand{\Hil}{\mathcal{H}}

\newcommand{\oscUG}{{\mathrm{osc}}_{\UU,\Gamma}}
\newcommand{\oscUtG}{{\mathrm{osc}}_{\widetilde{\UU},\Gamma}}

\newcommand{\Indicator}{{\mathds{1}}}

\newcommand{\lebesgue}{\bd{L}}

\newcommand{\CalA}{\mathcal{A}}
\newcommand{\CalC}{\mathcal{C}}
\newcommand{\CalM}{\mathcal{M}}

\newcommand{\CalU}{\mathcal{U}}

\newcommand{\CalB}{\mathcal{B}}
\newcommand{\CalG}{\mathcal{G}}
\newcommand{\CalF}{\mathcal{F}}
\newcommand{\CalH}{\mathcal{H}}

\newcommand{\Co}{\operatorname{Co}}

\newcommand{\identity}{\mathrm{id}}

\newcommand{\eps}{\varepsilon}

\usepackage{scalerel}
\renewcommand\upsilon{{\scaleobj{0.65}{\Upsilon}}}

\let\emptyset\varnothing

\begin{document}

\author{Nicki Holighaus\texorpdfstring{$^\dag$}{}}
\address{$^\dag$ Acoustics Research Institute,
Austrian Academy of Sciences,
Wohllebengasse 12--14,
A-1040 Vienna,
Austria}
\thanks{This work was supported by the Austrian Science Fund (FWF): I\,3067--N30.
NH is grateful for the hospitality and support of the Katholische Universität Eichstätt-Ingolstadt during his visit.
FV would like to thank the Acoustics Research Institute for the hospitality during several visits, which were partially
supported by the Austrian Science Fund (FWF): 31225--N32.}
\email{nicki.holighaus@oeaw.ac.at}
\author{Felix Voigtlaender\texorpdfstring{$^{*\S}$ $^\ddag$}{}}
\address{$^*$ Lehrstuhl Wissenschaftliches Rechnen,
Katholische Universität Eichstätt-Ingolstadt,
Ostenstraße 26,
85072 Eichstätt,
Germany}
\address{$^{\S}$ Faculty of Mathematics,
University of Vienna,
Oskar-Morgenstern-Platz 1
A-1090 Vienna,
Austria}
\email{felix@voigtlaender.xyz}

\title[Schur-type Banach modules acting on mixed-norm Lebesgue spaces]
      {\texorpdfstring{Schur-type Banach modules of integral kernels\\{}
                       acting on mixed-norm Lebesgue spaces}
                      {Schur-type Banach modules of integral kernels
                       acting on mixed-norm Lebesgue spaces}}





\begin{abstract}
  Schur's test for integral operators states that if a kernel $K : X \times Y \to \CC$
  satisfies $\int_Y |K(x,y)| \, d \nu(y) \leq C$
  and $\int_X |K(x,y)| \, d \mu(x) \leq C$, then the associated integral
  operator is bounded from $\lebesgue^p (\nu)$ into $\lebesgue^p(\mu)$,
  simultaneously for all $p \in [1,\infty]$.
  We derive a variant of this result which ensures that the integral operator
  acts boundedly on the (weighted) \emph{mixed-norm} Lebesgue spaces $\lebesgue_w^{p,q}$,
  simultaneously for all $p,q \in [1,\infty]$.
  For \emph{non-negative} integral kernels our criterion is sharp; that is,
  the integral operator satisfies our criterion \emph{if and only if} it acts boundedly
  on all of the mixed-norm Lebesgue spaces.

  Motivated by this new form of Schur's test, we introduce solid Banach modules $\CalB_m(X,Y)$
  of integral kernels with the property that all kernels in $\CalB_m(X,Y)$ map the
  mixed-norm Lebesgue spaces $\lebesgue_w^{p,q}(\nu)$ boundedly into $\lebesgue_v^{p,q}(\mu)$,
  for arbitrary $p,q \in [1,\infty]$, provided that the weights $v,w$ are $m$-moderate.
  Conversely, we show that if $\bd A$ and $\bd B$ are non-trivial solid Banach spaces
  for which all kernels $K \in \CalB_m(X,Y)$
  define bounded maps from $\bd A$ into $\bd B$, then $\bd A$ and $\bd B$ are related
  to mixed-norm Lebesgue-spaces, in the sense that
  \(
    \left(\lebesgue^1
    \cap \lebesgue^\infty
    \cap \lebesgue^{1,\infty}
    \cap \lebesgue^{\infty,1}\right)_v
    \hookrightarrow \bd B
  \)
  and
  \(
    \bd A \hookrightarrow \left(
                            \lebesgue^1
                            + \lebesgue^\infty
                            + \lebesgue^{1,\infty}
                            + \lebesgue^{\infty,1}
                          \right)_{1/w}
  \)
  for certain weights $v,w$ depending on the weight $m$ used in the definition of $\BBm$.

  The kernel algebra $\CalB_m(X,X)$ is particularly suited for applications in (generalized) coorbit theory.
  Usually, a host of technical conditions need to be verified to guarantee that the coorbit space
  $\Co_\Psi (\bd A)$ associated to a continuous frame $\Psi$ and a solid Banach space $\bd A$
  are well-defined and that the discretization machinery of coorbit theory is applicable.
  As a simplification, we show that it is enough to check that certain integral kernels
  associated to the frame $\Psi$ belong to $\CalB_m(X,X)$;
  this ensures that the spaces $\Co_\Psi (\lebesgue_\kappa^{p,q})$ are well-defined
  for all $p,q \in [1,\infty]$ and all weights $\kappa$ compatible with $m$.
  Further, if some of these integral kernels have sufficiently small norm,
  then the discretization theory is also applicable.
\end{abstract}

\maketitle

\renewcommand{\thefootnote}{\fnsymbol{footnote}}
\footnotetext[3]{Corresponding author}
\renewcommand*{\thefootnote}{\arabic{footnote}}

\markright{}

\section{Introduction}
\label{sec:Intro}

For integral kernels that do not exhibit cancellations---in particular for non-negative kernels---%
Schur's test is one of the most important criteria to verify that the associated integral operator
acts boundedly on the Lebesgue space $\lebesgue^p$.
More precisely, given $\sigma$-finite measure spaces $(X,\CalF,\mu)$ and $(Y,\CalG,\nu)$
and a measurable integral kernel $K : X \times Y \to \CC$, the \emph{integral operator}
$\Phi_K$ associated to the kernel $K$ is defined by
\[
  [\Phi_K \, f] (x)
  := \int_{Y}
       K(x,y) \, f(y)
     \, d \nu(y)
  \text{ for measurable } f : Y \to \CC
  \text{ and } x \in X,
  \text{ if the integral exists};
\]
one then wants to study conditions on $K$ which guarantee that $\Phi_K$ acts boundedly
on a given function space.

\subsection{The different versions of \texorpdfstring{Schur's}{Schurʼs} test}

Schur's test provides such a criterion for operators acting on the Lebesgue spaces $\lebesgue^p$.
In fact, there are two somewhat different results in the literature that are commonly referred to
as ``Schur's test''.
The first variant yields boundedness of $\Phi_K$ on $\lebesgue^p$ for a \emph{specific} choice of
$p \in (1,\infty)$:

\begin{thm*}[Schur's test, specific version]
  Under the above assumptions, let $p,q \in (1,\infty)$ be conjugate exponents.
  Assume that there exist $A,B > 0$ and measurable functions
  $g : Y \to (0,\infty)$ and $h : X \to (0,\infty)$ satisfying
  \[
    \int_Y
      |K(x,z)| \cdot [g(z)]^q
    \, d \nu (z)
    \leq A \cdot [h(x)]^q
    \qquad \text{and} \qquad
    \int_X
      |K(z,y)| \cdot [h(z)]^p
    \, d \mu(z)
    \leq B \cdot [g(y)]^p
  \]
  for $\mu$-almost all $x \in X$ and $\nu$-almost all $y \in Y$.
  Then $\Phi_K : \lebesgue^p(\nu) \to \lebesgue^p(\mu)$ is bounded with
  $\| \Phi_K \|_{\lebesgue^p \to \lebesgue^p} \leq A^{1/q} \, B^{1/p}$.
\end{thm*}

The above formulation of Schur's test appeared for the first time in
\mbox{\cite[Pages~239--240]{AronszajnSpacesOfPotentialsConnectedWithLP}};
it is a highly generalized version of a result by Schur \cite{Schur1911},
which considered matrix operators acting on $\ell^2$.
Gagliardo showed for non-negative kernels $K$ that the sufficient condition given above
is ``almost'' necessary; see \cite{GagliardoIntegralTransformstionWithPositiveKernel} for the details.

While the above theorem is very flexible and in particular allows to prove boundedness
of operators which act boundedly on $\lebesgue^p$ only for \emph{some but not all} exponents $p$,
the second version of Schur's test---which yields boundedness on $\lebesgue^p$ \emph{simultaneously}
for all $p \in [1,\infty]$---is more frequently used in applications related to time-frequency
analysis \cite[Lemma~6.2.1]{gr01} and coorbit theory
\cite{FeichtingerCoorbit0,FeichtingerCoorbit1,FeichtingerCoorbit2,GeneralizedCoorbit1,RauhutCoorbitQuasiBanach,kempka2015general}.
This second version reads as follows:

\begin{thm*}[Schur's test, uniform version]
  Assume that the integral kernel $K : X \times Y \to \CC$ is measurable and such that
  \begin{equation}\label{eq:schurtest0}
    \int_X |K(z,y)| \, d \mu(z) \leq C
    \,\,\, \text{and} \,\,\,
    \int_Y |K(x,z)| \, d \nu(z) \leq C
    \,\,\, \text{for almost all } x \in X \text{ and } y \in Y.
  \end{equation}
  Then $\Phi_K : \lebesgue^p(\nu) \to \lebesgue^p(\mu)$ is bounded for all $p \in [1,\infty]$,
  with $\| \Phi_K \|_{\lebesgue^p \to \lebesgue^p} \leq C$.
\end{thm*}

This second version of Schur's test is again a generalization of an estimate in Schur's
original work \cite{Schur1911}; it is a folklore result and can be found for instance in
\cite[Theorem~6.18]{FollandRA}.
It also seems to be folklore that the above form of Schur's test is \emph{sharp};
that is, \eqref{eq:schurtest0} holds \emph{if and only if}
$\Phi_K : \lebesgue^p(\nu) \to \lebesgue^p(\mu)$ is bounded for all $p \in [1,\infty]$.
Since we could not locate a reference for this fact in the setting of general measure spaces,
we provide a proof in Appendix~\ref{sec:SharpnessComplexValued}.

\medskip{}

An important application of this second form of Schur's test occurs in
\emph{generalized coorbit theory} \cite{GeneralizedCoorbit1},
where a \emph{Schur-type Banach algebra of integral kernels} is considered.
More precisely, in the setting where $(X,\CalF,\mu) = (Y,\CalG,\nu)$, define
\[
  \AAi(X)
  := \bigl\{
       K : X \times X \to \CC
       \quad\colon\quad
       K \text{ measurable and } \| K \|_{\AAi} < \infty
     \bigr\},
\]
where
\[
  \| K \|_{\AAi} := \max
                    \left\{
                      \esssup_{y \in X} \int_X |K(x,y)| \, d \mu(x), \quad
                      \esssup_{x \in X} \int_X |K(x,y)| \, d \mu(y)
                    \right\}
  \in [0,\infty] .
\]

Clearly, $\AAi(X)$ contains all kernels that satisfy Schur's test, so that
$\AAi(X) \hookrightarrow \Bounded(\lebesgue^p(\mu))$.
Moreover, $\AAi(X)$ is a Banach algebra with multiplication given by
\[
  L \odot K(x,z)
  := \int_Y
       L(x,y) \, K(y,z)
     d\mu(y),
  \text{ for (almost) all } x,z \in X.
\]

\subsection{Our contribution}
\label{sub:IntroContribution}

We are concerned with extending the ``uniform version'' of Schur's test for integral operators
acting on the Lebesgue spaces $\lebesgue^p$ to a version for operators acting
on the \emph{mixed-norm} Lebesgue spaces $\lebesgue^{p,q}$,
which were originally introduced in \cite{MixedLpSpaces}.

Precisely, given a measure space $\XTuple = (X_1 \times X_2, \CalF_1 \otimes \CalF_2, \mu_1 \otimes \mu_2)$
which is the product of two $\sigma$-finite measure spaces $(X_i,\CalF_i,\mu_i)$,
the mixed Lebesgue-norm with exponents $p,q \in [1,\infty]$ is given by
\begin{equation}
  \|f\|_{\lebesgue^{p,q}(\mu)}
  := \Big\|
       x_2 \mapsto \big\| f(\bullet, x_2) \big\|_{\lebesgue^p (\mu_1)}
     \Big\|_{\lebesgue^q (\mu_2)}
  \in [0,\infty]
  \quad \text{for} \quad
  f : X \to \CC \text{ measurable} .
  \label{eq:MixedLebesgueNormDefinition}
\end{equation}
As usual, $\lebesgue^{p,q}(\mu)$ is the set of all (equivalence classes of) measurable functions
for which this norm is finite.
One can show that $\lebesgue^{p,q}(\mu)$ is indeed a Banach space; see \cite{MixedLpSpaces}.

\medskip{}

We make the following contributions regarding Schur's test for mixed-norm Lebesgue spaces:
\begin{itemize}
  \item We derive a variant of the condition \eqref{eq:schurtest0} which guarantees that
        $\Phi_K : \lebesgue^{p,q}(\nu) \to \lebesgue^{p,q}(\mu)$ is bounded simultaneously for all
        $p,q \in [1,\infty]$.

  \item This sufficient condition turns out to be reasonably sharp,
        meaning that for \emph{non-negative} integral kernels $K$,
        our sufficient condition is in fact necessary.
        For complex-valued integral kernels, however, this is no longer true in general.

  \item In the same way that the Banach algebra $\AAi(X)$ relates to the Schur-type condition
        \eqref{eq:schurtest0}, we introduce a novel family $\BBi(X,Y)$ of Banach spaces
        of integral kernels related to our generalized Schur-type condition
        and study its properties.

  \item In particular, we study \emph{necessary} conditions that a function space $\BanachOne$
        has to satisfy in order for $\BBi(X,X)$ to act boundedly on it.
        Our main result in this direction shows that such a space $\BanachOne$
        necessarily satisfies
        \[
          \lebesgue^1 \cap \lebesgue^\infty \cap \lebesgue^{1,\infty} \cap \lebesgue^{\infty,1}
          \hookrightarrow \BanachOne
          \hookrightarrow \lebesgue^1 + \lebesgue^\infty + \lebesgue^{1,\infty} + \lebesgue^{\infty,1}.
        \]

  \item We indicate how our results can be used to obtain streamlined conditions for the
        applicability of generalized coorbit theory using the kernel spaces $\BBi(X,Y)$.
\end{itemize}

In the above list, we only restricted ourselves to the unweighted spaces $\BBi(X,Y)$ for simplicity.
In fact, each of the listed questions is also considered for the weighted spaces $\BBm(X,Y)$,
acting boundedly on weighted mixed-norm Lebesgue spaces $\lebesgue^{p,q}_w$.

\subsection{Related work}

Mixed-norm Lebesgue spaces---originally introduced in \cite{MixedLpSpaces}---appear naturally
whenever one considers functions of more than one variable where the different variables
are related to fundamentally different notions or physical quantities.
Such is the case in time-frequency analysis~\cite{gr01}, where the variables represent time
(or space) and frequency, respectively.
In particular, the \emph{modulation space norms}
\cite{gr01,feichtinger1983modulation,feichtinger1989atomic,grochenig1999modulation}
are defined by putting a (weighted) mixed Lebesgue norm on the short-time Fourier transform
of the considered functions.

Generalizing from this, Feichtinger and Gröchenig developed \emph{coorbit theory}
\cite{FeichtingerCoorbit0,FeichtingerCoorbit1,FeichtingerCoorbit2,GroechenigDescribingFunctions},
a general framework for defining function spaces by putting certain function space norms
on suitable integral transforms of the functions under consideration,
and for discretizing such function spaces.
The prototypical examples of coorbit spaces are the modulation spaces and the Besov spaces
\cite{TriebelTheoryOfFunctionSpaces,TriebelTheoryOfFunctionSpaces3}, which can be described
by putting a weighted mixed Lebesgue norm on either the short-time Fourier transform or the
continuous wavelet transform of the given functions.
Other examples include Triebel-Lizorkin spaces; see for instance
\cite{UllrichContinuousCharacterizationsOfBesovTLSpaces,UllrichYangNewCharacterizationsOfBesovTLSpaces}.
The coorbit description of Besov spaces was generalized in
\cite{FuehrVoigtlaenderCoorbitSpacesAsDecompositionSpaces} to function spaces
associated to more general wavelet-type transforms, including the \emph{anisotropic} Besov spaces
studied in \cite{BownikAnisotropicBesovSpaces}.

In the original setup of Feichtinger and Gröchenig, the considered integral transforms
were required to stem from an irreducible, integrable group representation.
This assumptions has been significantly relaxed by the combined work of several authors
\cite{GeneralizedCoorbit1,GeneralizedCoorbit2,kempka2015general}, leading to the theory
of \emph{general coorbit spaces}, for which the integral transform is merely required
to be induced by a continuous (Parseval) frame;
see Sections~\ref{sec:IntroCoorbitTheory} and \ref{sub:CoorbitReview}.
However, while mixed-norm Lebesgue spaces have seen substantial interest and applications
for group-based coorbit theory, they have been mostly neglected in \emph{general} coorbit theory.
In Sections~\ref{sec:IntroCoorbitTheory} and \ref{sec:CoorbitTheory},
we demonstrate how our results can be applied to fill this gap.

Extensions of Schur's test to mixed-norm Lebesgue spaces have been studied in
\cite{TaylorPhDThesis} and \cite{samarah2005schur}.
In \cite{TaylorPhDThesis}, Taylor considers generalizations of the ``specific version''
of Schur's test to the setting of $n$-variable mixed-norm Lebesgue spaces $\lebesgue^{P}$
with $P = (p_1,\dots,p_n) \in (1,\infty)^n$.
In addition to generalizing the sufficient conditions to this setting, most of \cite{TaylorPhDThesis}
is concerned with showing that these sufficient conditions are also (almost) necessary,
similar to the results of Gagliardo \cite{GagliardoIntegralTransformstionWithPositiveKernel}
for the usual Lebesgue spaces.
Motivated by applications in time-frequency analysis, the work \cite{samarah2005schur}
by Samarah et al.~studies sufficient criteria regarding the integral kernel $K$
which ensure that $\Phi_K : \lebesgue^{p,q} \to \lebesgue^{p,q}$ is bounded%
\footnote{It should be observed that while the theorem statements in \cite{samarah2005schur}
seem to consider boundedness of ${\Phi_K : \lebesgue^{p,q} \to \lebesgue^{p',q'}}$, an analysis
of the proofs shows that actually $\Phi_K : \lebesgue^{p,q} \to \lebesgue^{p,q}$ is meant.}
for $p,q \in (1,\infty)$.
Again, these sufficient conditions are in the spirit of the ``specific version'' of Schur's test.
Furthermore, it should be noted that most of the results in \cite{samarah2005schur} assume
one of the measure spaces under consideration to have \emph{finite} measure;
the only result which does not do so is \cite[Proposition~4]{samarah2005schur},
in which it is assumed instead that $1 < p \leq q < \infty$, leaving open the case where $q < p$.

For the Lebesgue spaces $\lebesgue^p$, the ``uniform'' version of Schur's test
(for $p \in (1,\infty)$) is a straightforward consequence of the ``specific'' version,
derived by simply choosing $h \equiv g \equiv 1$.
However, it is unclear how such a derivation can be obtained from the results
in \cite{TaylorPhDThesis,samarah2005schur} in the case of the \emph{mixed-norm} Lebesgue spaces.
Besides formalizing the ``uniform'' version of Schur's test in that setting,
our work confirms its necessity and studies the properties of the closely related kernel spaces $\BBm(X,Y)$,
thereby complementing the results in \cite{TaylorPhDThesis,samarah2005schur}
by the addition of a new toolset for the study of integral operators on mixed-norm Lebesgue spaces.

\subsection{Structure of the paper}
\label{sub:Structure}

In the next section we state the main results of this paper.
Namely, we formulate the ``uniform version'' of Schur's test for mixed-norm Lebesgue spaces
and discuss its relevance as sufficient and necessary criterion for integral operators
mapping $\lebesgue^{p,q}(\nu)$ into $\lebesgue^{p,q}(\mu)$.
We also introduce the Schur-type Banach modules $\BBi(X,Y)$ and
their weighted variants, and we discuss the properties of these spaces.
In Section~\ref{sec:IntroCoorbitTheory}, we present
an overview of general coorbit theory with respect to mixed-norm Lebesgue spaces.
All proofs are deferred to the later sections:
Section~\ref{sec:MixedNormSchur} covers all proofs
related directly to the proposed extension of Schur's test,
while Sections~\ref{sec:KernelModulePropertiesProof} and \ref{sec:StructureOfBBmCompatibleSpaces}
are concerned with proving the properties of $\BBi(X,Y)$ and the kernels contained therein.
In Section~\ref{sec:EmbeddingIntoWeightedLInfty} we consider additional mapping properties of
the integral operators derived from kernels in $\BBi(X,Y)$.
Finally, Section~\ref{sec:CoorbitTheory} closes our treatment
with a more detailed account of general coorbit theory using mixed-norm Lebesgue spaces.
Several technical proofs are deferred to the Appendices.
In particular, we show in Appendix~\ref{sec:DualCharacterizationOfSumSpace} that
$\lebesgue^1 + \lebesgue^\infty + \lebesgue^{1,\infty} + \lebesgue^{\infty,1}$
is the associate space of the space
$\lebesgue^1 \cap \lebesgue^\infty \cap \lebesgue^{1,\infty} \cap \lebesgue^{\infty,1}$,
which might be of independent interest.

\section{\texorpdfstring{Schur's}{Schurʼs} test for mixed-norm Lebesgue spaces}
\label{sec:IntroMixedNormSchur}

As for the classical Schur test, we aim for readily verifiable conditions concerning the kernel
$K : X \times Y \to \CC$ which guarantee that the associated integral operator $\Phi_K$ defines
a bounded linear map $\Phi_K : \lebesgue^{p,q}(\nu) \to \lebesgue^{p,q}(\mu)$,
simultaneously for all $p,q \in [1,\infty]$.

To conveniently state our result, given a measurable function $K : X \times Y \to \CC$,
we define
\begin{equation}
  \begin{split}
    C_1 (K)
    & := \esssup_{x \in X}
           \int_Y
             |K(x,y)|
           \, d \nu (y) \in [0,\infty] , \\
    C_2 (K)
    & := \esssup_{y \in Y}
           \int_X
             |K(x,y)|
           \, d \mu(x) \in [0,\infty] , \\
    C_3 (K)
    & := \esssup_{x_2 \in X_2}
         \left[
           \int_{Y_2}
             \bigg(
               \esssup_{y_1 \in Y_1}
                 \int_{X_1}
                   \big| K \big( (x_1, x_2), (y_1,y_2) \big) \big|
                 \, d \mu_1 (x_1)
             \bigg)
           \, d \nu_2 (y_2)
         \right] \in [0,\infty] , \\
    C_4 (K)
    & := \esssup_{y_2 \in Y_2}
         \left[
           \int_{X_2}
             \bigg(
               \esssup_{x_1 \in X_1}
                 \int_{Y_1}
                   \big| K \big( (x_1, x_2), (y_1,y_2) \big) \big|
                 \, d \nu_1 (y_1)
             \bigg)
           \, d \mu_2 (x_2)
         \right] \in [0,\infty] .
  \end{split}
  \label{eq:MixedNormSchurConstants}
\end{equation}
It should be observed that $C_3(K), C_4(K) \in [0,\infty]$ are well-defined%
---that is, all appearing integrands are indeed measurable---as follows by combining Tonelli's
theorem with Lemma~\ref{lem:CountableLInfinityCharacterization}.

Using the quantities $C_1(K),\dots,C_4(K)$, we can now conveniently state our version of Schur's test
for mixed Lebesgue spaces:

\begin{theorem}\label{thm:SchurTestSufficientUnweighted}
  Let $\XTuple \!=\! \ProductTupleXCompressed$
  and $\YTuple \!=\! \ProductTupleYCompressed$
  and assume that $\mu_1,\mu_2,\nu_1,\nu_2$ are all $\sigma$-finite.

  If $K : X \times Y \to \CC$ is measurable and satisfies $C_i (K) < \infty$ for all
  $i \in \{1,2,3,4\}$, where the $C_i (K)$ are as in Equation~\eqref{eq:MixedNormSchurConstants},
  then the integral operator $\Phi_K$ is well-defined and bounded as an operator
  $\Phi_K : \lebesgue^{p,q}(\nu) \to \lebesgue^{p,q}(\mu)$ for all $p,q \in [1,\infty]$,
  with absolute convergence a.e.~of the defining integral.
  More precisely:
  \begin{itemize}
    \item If $p \leq q$ and $C_i (K) < \infty$ for $i \in \{1,2,3\}$, then
          \(
            \vertiii{\Phi_K}_{\lebesgue^{p,q} \to \lebesgue^{p,q}}
            \leq {\displaystyle
                  \max_{i \in \{ 1,2,3 \}}
                    C_i (K)
                 }
            <    \infty
            .
          \)
          \vspace{0.1cm}

    \item If $p > q$ and $C_i (K) < \infty$ for $i \in \{1,2,4\}$, then
          \(
            \vertiii{\Phi_K}_{\lebesgue^{p,q} \to \lebesgue^{p,q}}
            \leq {\displaystyle
                  \max_{i \in \{ 1,2,4 \}}
                    C_i (K)
                 }
            <    \infty
            .
          \)
  \end{itemize}
\end{theorem}

\begin{proof}
  The proof is given in Section~\ref{sub:MixedNormSchurSufficient}.
\end{proof}

Even though the developed criterion is very convenient, one might wonder how sharp it is.
At first sight, it might be possible that there are kernels $K$ for which
$\Phi_K : \lebesgue^{p,q}(\nu) \to \lebesgue^{p,q} (\mu)$ is bounded for all $p,q \in [1,\infty]$,
but for which the generalized Schur test does not prove this boundedness.
At least for kernels without cancellations---that is, non-negative kernels---this does not happen.

\begin{theorem}\label{thm:SchurNecessity}
  Let $\XTuple \!=\! \ProductTupleXCompressed$ and $\YTuple \!=\! \ProductTupleYCompressed$,
  and assume that $\mu_1,\mu_2,\nu_1,\nu_2$ are all $\sigma$-finite.

  Let $K : X \times Y \to [0,\infty]$ be measurable, and let the constants $C_i(K)$
  be as defined in Equation~\eqref{eq:MixedNormSchurConstants}.
  Then the following hold:
  \begin{enumerate}
    \item If $\Phi_K : \lebesgue^1(\nu) \to \lebesgue^1(\mu)$ is well-defined and bounded,
          then $C_2 (K) \leq \vertiii{\Phi_K}_{\lebesgue^1 \to \lebesgue^1}$.
          \par\vspace{0.1cm}

    \item If $\Phi_K : \lebesgue^\infty(\nu) \to \lebesgue^\infty(\mu)$ is well-defined and bounded,
          then $C_1(K) \leq \vertiii{\Phi_K}_{\lebesgue^\infty \to \lebesgue^\infty}$.
          \par\vspace{0.1cm}

    \item If $\Phi_K : \lebesgue^{1,\infty}(\nu) \to \lebesgue^{1,\infty}(\mu)$
          is well-defined and bounded,
          then $C_3 (K) \leq \vertiii{\Phi_K}_{\lebesgue^{1,\infty} \to \lebesgue^{1,\infty}}$.
          \par\vspace{0.1cm}

    \item If $\Phi_K : \lebesgue^{\infty,1}(\nu) \to \lebesgue^{\infty,1}(\mu)$
          is well-defined and bounded,
          then $C_4 (K) \leq \vertiii{\Phi_K}_{\lebesgue^{\infty,1} \to \lebesgue^{\infty,1}}$.
  \end{enumerate}
\end{theorem}
\begin{proof}
  The proof is given in Section~\ref{sub:MixedNormSchurNecessary}.
\end{proof}

\begin{rem*}
  a) This shows in particular that if $K$ is non-negative,
  then $\Phi_K : \lebesgue^{p,q}(\nu) \to \lebesgue^{p,q}(\mu)$ is well-defined and bounded
  simultaneously for \emph{all} $p,q \in [1,\infty]$ if and only if this holds for all
  the ``boundary cases'' $p,q \in \{1, \infty\}$.

  b) The non-negativity assumption regarding $K$ cannot be dropped in general.
  Indeed, in Appendix~\ref{sec:SharpnessComplexValued} we provide an example of a complex-valued
  integral kernel for which $\Phi_K$ acts boundedly on the mixed Lebesgue space $\lebesgue^{p,q}$
  for all $p,q \in [1,\infty]$, but such that $C_3 (K) = \infty$.
\end{rem*}

The stated results readily generalize to the \emph{weighted} mixed-norm Lebesgue spaces
$\lebesgue^{p,q}_w$.
To describe this, let us first define these spaces.
In general, given any Banach space $(\mathbf{A}, \| \bullet \|_{\mathbf{A}})$
whose elements are (equivalence classes of) measurable functions $f : X \to \CC$
on a measure space $(X,\CalA,\mu)$, and given any measurable function $w : X \to (0,\infty)$
(called a \emph{weight}), we define the \emph{weighted space} $\mathbf{A}_w$ as
\begin{equation}
  \mathbf{A}_w
  := \big\{ f : X \to \CC \colon f \text{ measurable and } w \cdot f \in \mathbf{A} \big\}
  \quad \text{ with norm } \quad
  \| f \|_{\mathbf{A}_w} := \| w \cdot f \|_{\mathbf{A}} .
  \label{eq:WeightedSpaceDefinition}
\end{equation}

Now, given weights $v : X \to (0,\infty)$ and $w : Y \to (0,\infty)$, and an integral kernel
$K : X \times Y \to \CC$, define
\begin{equation}
  K_{v,w} :
  X \times Y \to \CC,
  (x,y) \mapsto \frac{v(x)}{w(y)} \cdot K(x,y) .
  \label{eq:WeightedKernelDefinition}
\end{equation}
It is then straightforward to verify that for a given measurable function $f : Y \to \CC$,
$\Phi_{K} f$ is well-defined if and only if $\Phi_{K_{v,w}} (w \cdot f)$ is well-defined,
and in this case we have
\[
  v \cdot (\Phi_K \, f)
  = \Phi_{K_{v,w}} (w \cdot f).
\]
Then, by applying the results from above for the unweighted spaces to the weighted kernel $K_{v,w}$,
we obtain the following generalized Schur-test concerning the boundedness of the integral operator
$\Phi_K$ acting on weighted mixed-norm Lebesgue spaces.

\begin{theorem}\label{thm:SchurTestWeighted}
  Let $\XTuple \!=\! \ProductTupleXCompressed$ and $\YTuple \!=\! \ProductTupleYCompressed$,
  and assume that $\mu_1,\mu_2,\nu_1,\nu_2$ are all $\sigma$-finite.
  Let $v : X \to (0,\infty)$ and $w : Y \to (0,\infty)$, as well as $K : X \times Y \to \CC$
  be measurable, and let the weighted kernel $K_{v,w}$ be as defined
  in Equation~\eqref{eq:WeightedKernelDefinition}.
  Then the following hold for $p,q \in [1,\infty]$:
  \begin{itemize}
    \item If $p \leq q$ and if $C_i (K_{v,w}) < \infty$ for $i \in \{1,2,3\}$,
          then $\Phi_K : \lebesgue^{p,q}_w (\nu) \to \lebesgue^{p,q}_v (\mu)$
          is well-defined and bounded, with
          \(
            \vertiii{\Phi_K}_{\lebesgue^{p,q}_w \to \lebesgue^{p,q}_v}
            \leq {\displaystyle
                  \max_{1 \leq i \leq 3}
                    C_i (K_{v,w})
                 }
          \).

    \item If $p > q$ and if $C_i (K_{v,w}) < \infty$ for $i \in \{ 1,2,4 \}$
          then $\Phi_K : \lebesgue^{p,q}_w (\nu) \to \lebesgue^{p,q}_v (\mu)$
          is well-defined and bounded, with
          \(
            \vertiii{\Phi_K}_{\lebesgue^{p,q}_w \to \lebesgue^{p,q}_v}
            \leq {\displaystyle
                  \max_{i \in \{ 1,2,4 \}}
                    C_i (K_{v,w})
                 }
          \).
  \end{itemize}
  Finally, if $K : X \times Y \to [0,\infty]$ is non-negative, then the following hold:
  \begin{itemize}
    \item If $\Phi_K : \lebesgue^{1}_w (\nu) \to \lebesgue^{1}_v (\mu)$
          is well-defined and bounded, then
          $C_2 (K_{v,w}) \leq \vertiii{\Phi_K}_{\lebesgue^1_w \to \lebesgue^1_v}$.

    \item If $\Phi_K : \lebesgue^{\infty}_w (\nu) \to \lebesgue^{\infty}_v (\mu)$
          is well-defined and bounded, then
          $C_1 (K_{v,w}) \leq \vertiii{\Phi_K}_{\lebesgue^\infty_w \to \lebesgue^\infty_v}$.

    \item If $\Phi_K : \lebesgue^{1,\infty}_w (\nu) \to \lebesgue^{1,\infty}_v (\mu)$
          is well-defined and bounded, then
          $C_3 (K_{v,w}) \leq \vertiii{\Phi_K}_{\lebesgue^{1,\infty}_w \to \lebesgue^{1,\infty}_v}$.

    \item If $\Phi_K : \lebesgue^{\infty,1}_w (\nu) \to \lebesgue^{\infty,1}_v (\mu)$
          is well-defined and bounded, then
          $C_4 (K_{v,w}) \leq \vertiii{\Phi_K}_{\lebesgue^{\infty,1}_w \to \lebesgue^{\infty,1}_v}$.
  \end{itemize}
\end{theorem}

\subsection{A novel Banach module of integral kernels}
\label{sub:IntroKernelAlgebras}

Motivated by the classical form of Schur's test, Fornasier and Rauhut considered in
\cite{GeneralizedCoorbit1} the algebra $\AAi$ of integral kernels, and also its weighted variant $\AAm$.
Precisely, given $\sigma$-finite measure spaces $\XTuple$ and $\YTuple$,
and an integral kernel $K : X \times Y \to \CC$, the $\AAi$-norm of $K$ is given by
\begin{equation}\label{eq:normA1}
  \|K\|_{\AAi}
  := \|K\|_{\AAi (X,Y)}
  := \max \left\{
            \esssup_{x \in X}
              \| K(x,\bullet) \|_{\lebesgue^1(\nu)}
            \quad
            \esssup_{y \in Y}
              \| K(\bullet, y) \|_{\lebesgue^1(\mu)}
          \right\}
  \in [0,\infty] .
\end{equation}
Given this norm, the associated space is defined as
\[
  \AAi
  := \AAi (X,Y)
  := \big\{
       K : X \times Y \to \CC
       \,\,\colon\,
       K \text{ measurable and } \| K \|_{\AAi} < \infty
     \big\} .
\]
In \cite{GeneralizedCoorbit1}, it was observed that if $(X,\CalF,\mu) = (Y,\CalG,\nu)$,
then $\AAi (X,X)$ is a Banach algebra which satisfies
$\AAi(X,X) \hookrightarrow \Bounded (\lebesgue^p, \lebesgue^p)$ for all $p \in [1,\infty]$,
where we denote by $\Bounded (\BanachOne, \BanachTwo)$ the space of all bounded linear operators
$T : \BanachOne \to \BanachTwo$, for given Banach spaces $\BanachOne, \BanachTwo$.

Here and in the following, we frequently identify a space of kernels,
e.g., $\AAi(X,X)$, with the space of integral operators induced by such kernels,
e.g., $\{ \Phi_K \colon K \!\in\! \AAi(X,X)\}$.
Hence, the statement $\AAi(X,X) \hookrightarrow \Bounded (\BanachOne, \BanachTwo)$
is to be interpreted as follows:
For all $K\in \AAi(X,X)$, the operator $\Phi_K \colon \BanachOne \rightarrow \BanachTwo$
is well-defined and satisfies $\vertiii{\Phi_K}_{\BanachOne \rightarrow \BanachTwo} \lesssim \|K\|_{\AAi}$,
where the implied constant is independent of the choice of $K \in \AAi(X,X)$.

\begin{rem*}
  The space $\AAi$ can also be expressed in terms of mixed-norm Lebesgue spaces.
  In fact, using the notation $K^T (y,x) := K(x,y)$ for the \emph{transposed kernel},
  it follows directly from the definition that
  \(
    \| K \|_{\AAi}
    = \max \big\{
             \| K^T \|_{\lebesgue^{1,\infty}(\nu \otimes \mu)},
             \| K \|_{\lebesgue^{1,\infty}(\mu \otimes \nu)}
           \big\} .
  \)
\end{rem*}

Given our generalized form of Schur's test, it is natural to introduce an associated
kernel algebra similar to the algebra $\AAi$, and to study its basic properties.
In fact, we will \emph{not} assume that $(X,\CalF,\mu) = (Y,\CalG,\nu)$, so that in general
we will not obtain a \emph{Banach algebra} of kernels, but rather a \emph{Banach module} of kernels.
The formal definition is as follows:

\begin{definition}\label{def:NewKernelModule}
  Let $(X_i,\CalF_i,\mu_i)$ and $(Y_i,\CalG_i,\nu_i)$ be $\sigma$-finite measure spaces
  for $i \in \{ 1,2 \}$, and let $\XTuple = \ProductTupleX$ and $\YTuple = \ProductTupleY$.
  Given a measurable kernel $K : X \times Y \to \CC$ (or $K : X \times Y \to [0,\infty]$),
  we define
  \begin{equation}
    K^{(x_2,y_2)} (x_1, y_1) := K \big( (x_1,x_2), (y_1,y_2) \big)
    \quad \text{ for } \quad (x_1,x_2) \in X \text{ and } (y_1,y_2) \in Y.
    \label{eq:PartialKernelDefinition}
  \end{equation}
  Using this notation, we define
  \[
    \| K \|_{\BBi}
    := \| K \|_{\BBi (X,Y)}
    := \Big\|
         (x_2,y_2) \mapsto \| K^{(x_2,y_2)} \|_{\AAi (X_1, Y_1)}
       \Big\|_{\AAi (X_2, Y_2)}
    \in [0,\infty] ,
  \]
  and
  \(
    \BBi (X,Y)
    := \{ K : X \times Y \to \CC \, \colon K \text{ measurable and } \| K \|_{\BBi} < \infty \} .
  \)

  Finally, given a weight $m : X \times Y \to (0,\infty)$, we define
  $\| K \|_{\BBm} := \| m \cdot K \|_{\BBi} \in [0,\infty]$ and
  \(
    \BBm (X,Y)
    := \{ K : X \times Y \to \CC \, \colon m \cdot K \in \BBi (X,Y) \}.
  \)
\end{definition}

\begin{rem*}
  It should be observed that the norm $\| K \|_{\BBi} \in [0,\infty]$ is indeed well-defined,
  since ${X_2 \times Y_2 \to [0,\infty], (x_2, y_2) \mapsto \| K^{(x_2,y_2)} \|_{\AAi(X_1,Y_1)}}$
  is measurable, as can be seen by recalling the definition of $\| \bullet \|_{\AAi}$
  and by combining Tonelli's theorem with Lemma~\ref{lem:CountableLInfinityCharacterization}.
\end{rem*}

We now collect important basic properties of the spaces $\BBm(X,Y)$.
The proofs of Propositions~\ref{prop:NewKernelModuleBasicProperties1}--\ref{prop:NewKernelModuleBasicProperties3}
are given in Section~\ref{sec:KernelModulePropertiesProof}.

\begin{proposition}\label{prop:NewKernelModuleBasicProperties1}
  Let $(X_i,\CalF_i,\mu_i)$ and $(Y_i,\CalG_i,\nu_i)$ be $\sigma$-finite measure spaces
  for $i \in \{ 1,2 \}$, and let $\XTuple = \ProductTupleX$ and $\YTuple = \ProductTupleY$.
  The following hold for any (measurable) weight $m : X \times Y \to (0,\infty)$:
  \begin{enumerate}
    \item \label{enu:NewKernelModuleSolid}
          $\BBm(X,Y)$ is solid.
          That is, if $K \in \BBm(X,Y)$ and if $L : X \times Y \to \CC$ is measurable with
          $|L| \leq |K|$ almost everywhere, then $L \in \BBm(X,Y)$ and
          $\| L \|_{\BBm} \leq \| K \|_{\BBm}$.
          \vspace{0.1cm}

    \item \label{enu:NewKernelModuleFatouProperty}
          $\| \bullet \|_{\BBm}$ satisfies the \emph{Fatou property}:
          If $K_n : X \times Y \to [0,\infty]$ with $K_n \leq K_{n+1}$
          and if $K := \lim_{n \to \infty} K_n$ (pointwise),
          then $\| K_n \|_{\BBm} \to \| K \|_{\BBm}$.
          \vspace{0.1cm}

    \item \label{enu:NewKernelModuleComplete}
          $\big( \BBm(X,Y), \| \bullet \|_{\BBm} \big)$ is a Banach space.
    \vspace{0.1cm}

    \item \label{enu:NewKernelModuleEmbedsInOld}
          Each $K \in \BBm (X,Y)$ satisfies $K \in \AAm(X,Y)$ and
          $\| K \|_{\AAm} \leq \| K \|_{\BBm}$.
    \end{enumerate}
    \end{proposition}

Proposition~\ref{prop:NewKernelModuleBasicProperties2} below concerns transpositions
and compositions of kernels.
In particular, it shows that the spaces $\BBm(X,Y)$ are compatible with the usual
\emph{product of kernels}, given by
\[
  (K \KernelProduct L) (x,z)
  := \int_Y
       K(x,y) \, L(y,z)
     \, d \nu (y)
  \quad \text{for all } (x,z) \in X \times Z \text{ for which the integral exists},
\]
for arbitrary measurable kernels $K : X \times Y \to \CC$ and $L : Y \times Z \to \CC$.

\begin{proposition}\label{prop:NewKernelModuleBasicProperties2}
 Let $(X_i,\CalF_i,\mu_i)$ and $(Y_i,\CalG_i,\nu_i)$ be $\sigma$-finite measure spaces
  for $i \in \{ 1,2 \}$, and let $\XTuple = \ProductTupleX$ and $\YTuple = \ProductTupleY$.
  The following hold for any (measurable) weight $m : X \times Y \to (0,\infty)$:
   \begin{enumerate}
    \item \label{enu:NewKernelModuleAdjoint}
          Define $m^T : Y \times X \to (0,\infty), (y,x) \mapsto m(x,y)$.
          If $K \in \BBm(X,Y)$, then the transposed kernel
          $K^T : Y \times X \to \CC, (y,x) \mapsto K(x,y)$ satisfies
          $K^T \in \BBi_{m^T}(Y,X)$ and furthermore
          \(
            \| K^T \|_{\BBi_{m^T}(Y,X)}
            = \| K \|_{\BBm(X,Y)} .
          \)
          \vspace{0.1cm}

    \item \label{enu:NewKernelModuleMultiplicationProperty}
          Let
          \(
            (Z,\mathcal{H},\varrho)
            = (Z_1 \times Z_2, \mathcal{H}_1 \otimes \mathcal{H}_2, \varrho_1 \otimes \varrho_2) ,
          \)
          where $(Z_1, \mathcal{H}_1, \varrho_1)$ and $(Z_2, \mathcal{H}_2, \varrho_2)$
          are $\sigma$-finite measure spaces.
          Let $\omega : X \times Y \to (0,\infty)$, $\sigma : Y \times Z \to (0,\infty)$
          and $\tau : X \times Z \to (0,\infty)$ be measurable and such that
          $\tau (x,z) \leq C \cdot \omega(x,y) \cdot \sigma(y,z)$ for all $x \in X, y \in Y$
          and $z \in Z$, and some $C > 0$.

          If $K \in \BBi_\omega (X,Y)$ and $L \in \BBi_\sigma (Y,Z)$, then
          \[
            K \KernelProduct L \in \BBi_\tau (X,Z)
            \quad \text{and} \quad
            \| K \KernelProduct L \|_{\BBi_\tau}
            \leq C \, \| K \|_{\BBi_\omega} \, \| L \|_{\BBi_\sigma} .
          \]
          Furthermore, the integral defining $(K \KernelProduct L)(x,z)$ converges absolutely
          for almost all $(x,z) \in X \times Z$.

          In particular, if the assumptions hold for $X = Y = Z$ and $\tau = \omega = \sigma$,
          then $\BBi_\tau(X,X)$ is an algebra with respect to $\KernelProduct$,
          and a \emph{Banach} algebra if $C = 1$.
 \end{enumerate}
\end{proposition}

Finally, the integral operators $\Phi_K$ induced by kernels $K\in \BBm(X,Y)$ act boundedly on
mixed-norm Lebesgue spaces.

\begin{proposition}\label{prop:NewKernelModuleBasicProperties3}
          Let $(X_i,\CalF_i,\mu_i)$ and $(Y_i,\CalG_i,\nu_i)$ be $\sigma$-finite measure spaces
          for $i \in \{ 1,2 \}$, and let $\XTuple = \ProductTupleX$ and $\YTuple = \ProductTupleY$.
          Let further $v : X \to (0,\infty)$, $w : Y \to (0,\infty)$,
          and $m: X\times Y \rightarrow (0,\infty)$ be measurable,
          and assume that there is $C > 0$ such that $\frac{v(x)}{w(y)} \leq C \cdot m(x,y)$
          for all $(x,y) \in X \times Y$.

          For any $K \in \BBm(X,Y)$ and $p,q \in [1,\infty]$,
          the integral operator $\Phi_K$ is well-defined and bounded as an operator
          $\Phi_K : \lebesgue^{p,q}_w (\nu) \to \lebesgue^{p,q}_v (\mu)$,
          with absolute convergence almost everywhere of the defining integral.
          Finally,
          \(
            \vertiii{\Phi_K}_{\lebesgue^{p,q}_w (\nu) \to \lebesgue^{p,q}_v (\mu)}
            \leq C \cdot \| K \|_{\BBm (X,Y)}.
          \)
\end{proposition}

\subsection{Necessary conditions for spaces compatible with the kernel modules \texorpdfstring{$\BBm$}{𝓑ₘ}}
\label{sub:RestrictionsOnSolidSpaces}

Proposition~\ref{prop:NewKernelModuleBasicProperties3} shows in particular that kernels belonging
to $\BBm$ induce integral operators that act boundedly on mixed-norm Lebesgue spaces.
In addition to these spaces, there certainly exist other function spaces
$\BanachOne, \BanachTwo$ such that every kernel $K \in \BBm$ gives rise
to a bounded integral operator $\Phi_K : \BanachOne \to \BanachTwo$.
To mention just one possibility, we note that if
$\BanachOne = \lebesgue^{1,\infty} (\nu) \cap \lebesgue^{\infty,1} (\nu)$
and $\BanachTwo = \lebesgue^{1,\infty} (\mu) \cap \lebesgue^{\infty,1} (\mu)$,
then each kernel $K \in \BBi (X,Y)$ defines a bounded integral operator
$\Phi_K : \BanachOne \to \BanachTwo$.
Here, we equip the space $\BanachOne$ with the norm
\(
  \| f \|_{\BanachOne}
  = \max
    \big\{
      \| f \|_{\lebesgue^{1,\infty} (\nu)} , \,
      \| f \|_{\lebesgue^{\infty,1}(\nu)}
    \big\}
\),
and similarly for $\BanachTwo$.
Of course, a host of other choices for $\BanachOne, \BanachTwo$ are possible as well.

Still, the following question appears natural:
Given function spaces $\BanachOne, \BanachTwo$ such that each kernel $K \in \BBm (X,Y)$
induces a bounded integral operator $\Phi_K : \BanachOne \to \BanachTwo$, what can be said
about the spaces $\BanachOne, \BanachTwo$?
Are they ``similar'' to (weighted) mixed-norm Lebesgue spaces in some sense?

We will give a (partial) answer to this question for the case that $\BanachOne, \BanachTwo$
are \emph{solid function spaces}.
Intuitively, the membership of a function $f$ in a solid space does not depend on
any regularity properties of $f$, but only on the magnitude of the function values $f(x)$.
Formally, we say that a normed vector space $(\BanachOne, \| \bullet \|_{\BanachOne})$
is a solid function space on a measure space $\XTuple$
if $\BanachOne$ consists of (equivalence classes of almost everywhere equal)
measurable functions $f : X \to \CC$, with the additional property that if
$f \in \BanachOne$ and if $g : X \to \CC$ is measurable with $|g| \leq |f|$ almost everywhere,
then $g \in \BanachOne$ and $\| g \|_\BanachOne \leq \| f \|_\BanachOne$.
In case that a solid function space $\BanachOne$ is complete, we say that it is
a \emph{solid Banach function space}.
For more details on such spaces, we refer to \cite{ZaanenIntegration}
and \cite{BennettSharpleyInterpolationOfOperators},
or to \cite[Section~2.2]{VoigtlaenderPhDThesis} for readers interested in Quasi-Banach spaces.

We will show in Theorem~\ref{thm:NecessaryConditionsForCompatibleSpaces} below
that Banach function spaces ``compatible'' with the kernel modules $\BBm(X,Y)$
are indeed somewhat similar to certain weighted mixed-norm Lebesgue spaces.
To formulate this conveniently, we first introduce two special spaces related
to mixed-norm Lebesgue spaces.

\begin{definition}\label{def:SumAndIntersectionSpaces}
  Let $\XTuple = \ProductTupleX$ where $\mu_1,\mu_2$ are $\sigma$-finite.
  We denote by
  \(
    \lebesgue^1 (\mu)
    + \lebesgue^\infty (\mu)
    + \lebesgue^{1,\infty} (\mu)
    + \lebesgue^{\infty,1} (\mu)
  \)
  the space of all (equivalence classes of almost everywhere equal)
  measurable functions $f: X \rightarrow \CC$, for which the norm
  \begin{equation}
    \begin{split}
      & \| f \|_{\lebesgue^1 + \lebesgue^\infty + \lebesgue^{1,\infty} + \lebesgue^{\infty,1}} \\
      & := \inf
           \Big\{
             \| f_1 \|_{\lebesgue^1}
             + \| f_2 \|_{\lebesgue^\infty}
             + \| f_3 \|_{\lebesgue^{1,\infty}}
             + \| f_4 \|_{\lebesgue^{\infty,1}}
             \,\,\colon
             \begin{array}{l}
               f_1,f_2,f_3,f_4 : X \to \CC \text{ measurable} \\
               \text{and } f = f_1 + f_2 + f_3 + f_4
             \end{array}
           \Big\}
    \end{split}
    \label{eq:SumNormDefinition}
  \end{equation}
  is finite.
  We further denote by
  \(
    \lebesgue^1 (\mu)
    \cap \lebesgue^\infty (\mu)
    \cap \lebesgue^{1,\infty} (\mu)
    \cap \lebesgue^{\infty,1} (\mu)
  \)
  the space of all (equivalence classes of almost everywhere equal)
  measurable functions $f: X \rightarrow \CC$, for which the norm
  \begin{equation}
    \| f \|_{\IntersectionSpace} := \IntersectionNorm{f}
    \label{eq:IntersectionNormDefinition}
  \end{equation}
  is finite.
\end{definition}

\begin{rem*}
  It is not quite obvious that the functional $\| \bullet \|_{\SumSpace}$ is indeed definite,
  and hence a norm. We verify definiteness as part of Theorem~\ref{thm:DualOfIntersection}.
\end{rem*}

\begin{theorem}\label{thm:NecessaryConditionsForCompatibleSpaces}
  Let $\XTuple \!=\! \ProductTupleX$ and $\YTuple \!=\! \ProductTupleY$,
  where $\mu_1,\mu_2, \nu_1,\nu_2$ are all $\sigma$-finite.
  Let $m : X \times Y \to (0,\infty)$ be measurable, and let
  $\BanachOne$ be a solid Banach function space on $Y$
  and $\BanachTwo$ a solid Banach function space on $X$.
  Finally, let $v : X \to (0,\infty)$ and $w : Y \to (0,\infty)$ be measurable and such that
  $m(x,y) \leq C \cdot v(x) \cdot w(y)$ for all $(x,y) \in X \times Y$ and some $C > 0$.
  Then the following hold:
  \begin{enumerate}
    \item \label{enu:KernelCoDomainEmbedding}
          If $\BanachOne \neq \{0\}$ and if each $K \in \BBm(X,Y)$ induces a well-defined operator
          $\Phi_K : \BanachOne \to \BanachTwo$, then
          \[
            \big(
              \lebesgue^1(\mu)
              \cap \lebesgue^\infty (\mu)
              \cap \lebesgue^{1,\infty}(\mu)
              \cap \lebesgue^{\infty,1}(\mu)
            \big)_v
            \hookrightarrow \BanachTwo .
          \]

    \item \label{enu:KernelDomainEmbedding}
          If $\mu(X) \neq 0$ and if for each non-negative kernel $K \in \BBm (X,Y)$
          and each non-negative $f \in \BanachOne$,
          the function $\Phi_K f : X \to [0,\infty]$ is almost-everywhere finite-valued, then
          \[
            \BanachOne
            \hookrightarrow \big(
                              \lebesgue^1 (\nu)
                              + \lebesgue^\infty (\nu)
                              + \lebesgue^{1,\infty} (\nu)
                              + \lebesgue^{\infty,1} (\nu)
                            \big)_{1/w} .
          \]
  \end{enumerate}
\end{theorem}

\begin{proof}
  The proof of Theorem~\ref{thm:NecessaryConditionsForCompatibleSpaces}
  is given in Section~\ref{sec:StructureOfBBmCompatibleSpaces}.
\end{proof}

\begin{remark}\label{rem:NecessaryConditionsSharpnessRemark}
  In many cases,
  \(
    \BanachThree = \SumSpaceSymbol
    := \big(
         \lebesgue^1(\nu)
         + \lebesgue^\infty (\nu)
         + \lebesgue^{1,\infty}(\nu)
         + \lebesgue^{\infty,1} (\nu)
       \big)_{1/w}
  \)
  turns out to be the \emph{minimal} space with the property that
  $\BanachOne \hookrightarrow \BanachThree$ for every solid Banach function space $\BanachOne$ on $Y$
  for which $\Phi_K f$ is almost-everywhere finite-valued for all $0 \leq K \in \BBm (X,Y)$
  and $0 \leq f \in \BanachOne$.

  Indeed, let us assume that $\BanachThree$ is a space with this property
  and that there is a weight ${u : X \to (0,\infty)}$ satisfying
  \begin{equation}
    \frac{u(x)}{1/w(y)} = u(x) \cdot w(y) \leq m(x, y)
    \qquad \text{for all} \qquad
    (x,y) \in X \times Y .
    \label{eq:NecessaryConditionsSharpnessAssumption}
  \end{equation}
  Then, Proposition~\ref{prop:NewKernelModuleBasicProperties3} shows for
  $K \in \BBm(X,Y)$ that $\Phi_K : \lebesgue^{p,q}_{1/w}(\nu) \to \lebesgue^{p,q}_u (\mu)$
  is well-defined and bounded, for arbitrary $p,q \in \{1,\infty\}$.
  In particular, this means that $\Phi_K f (x) < \infty$ almost everywhere
  for all $0 \leq K \in \BBm(X,Y)$ and $0 \leq f \in \lebesgue^{p,q}_{1/w} (\nu)$,
  with $p,q \in \{1,\infty\}$.
  By our assumption on $\BanachThree$,
  this means $\lebesgue^{p,q}_{1/w} (\nu) \hookrightarrow \BanachThree$
  for all $p,q \in \{1,\infty\}$, and hence $\SumSpaceSymbol \hookrightarrow \BanachThree$.

  The existence of $u$ with $u(x) \cdot w(y) \leq m(x,y)$
  is for example satisfied in the common case that $X = Y$ and
  \begin{equation}
    m(x,y) = m_w (x,y) := \max \Big\{ \frac{w(x)}{w(y)}, \frac{w(y)}{w(x)} \Big\}
    \quad \text{for a weight} \quad
    w : X \to (c,\infty) \text{ with } c > 0.
    \label{eq:SeparableMatrixWeight}
  \end{equation}
  Indeed, if we set $u : X \to (0,\infty), x \mapsto 1/w(x)$, then
  $u(x) \cdot w(y) \leq m(x,y) \leq \frac{1}{c^2} \cdot w(x) \cdot w(y)$.
\end{remark}

\subsection{\texorpdfstring{$\AAm$}{𝓐ₘ} as a special case of \texorpdfstring{$\BBm$}{𝓑ₘ}}%
\label{sub:OldAlgebraIsSpecialCase}

The classical spaces $\AAm (X_1, Y_1)$ arise as special cases of the newly introduced spaces $\BBm(X,Y)$.
Indeed, given $\sigma$-finite measure spaces $\XIndexTuple{1}$ and $\YIndexTuple{1}$, define
$X_2 := Y_2 := \{0\}$ and $\mu_2 := \nu_2 := \delta_0$,
and as usual $X := X_1 \times X_2$ and $Y := Y_1 \times Y_2$, as well as
$\mu := \mu_1 \otimes \mu_2$ and $\nu := \nu_1 \otimes \nu_2$.
Now, let us identify a function ${f : X_1 \to \CC}$
with $\widetilde{f} : X \to \CC, (x_1, x_2) \mapsto f(x_1)$,
and $g : Y_1 \to \CC$ with $\widetilde{g} : Y \to \CC, (y_1, y_2) \mapsto g(y_1)$.
A direct calculation shows that
$\| \widetilde{f} \, \|_{\lebesgue^{p,q}(\mu)} = \| f \|_{\lebesgue^p (\mu_1)}$
and similarly $\| \widetilde{g} \|_{\lebesgue^{p,q}(\nu)} = \| g \|_{\lebesgue^p (\nu_1)}$,
for arbitrary $p,q \in [1,\infty]$.
In other words, ${\lebesgue^{p,q}(\mu) = \lebesgue^{p}(\mu_1)}$
and ${\lebesgue^{p,q}(\nu) = \lebesgue^p (\nu_1)}$, up to canonical identifications.
Finally, an easy calculation shows that if we identify a given kernel ${K : X_1 \times Y_1 \to \CC}$
with ${\widetilde{K} : X \times Y \to \CC, \big( (x_1,x_2), (y_1,y_2) \big) \mapsto K(x_1, y_1)}$,
then ${\| \widetilde{K} \|_{\BBi(X,Y)} = \| K \|_{\AAi(X_1, Y_1)}}$,
and thus $\AAm(X_1,Y_1) = \BBm(X,Y)$, up to canonical identifications.

By using this identification of the spaces $\AAm(X_1, Y_1)$ as special $\BBm(X,Y)$ spaces,
all of our results imply a corresponding result for $\AAm(X_1, Y_1)$.
Most of the conclusions obtained in this way are already well-known, but some seem to be new.
As an example of such a result, we explicitly state the consequences of
Theorem~\ref{thm:NecessaryConditionsForCompatibleSpaces} for the spaces $\AAm$.

\begin{corollary}\label{cor:NecessaryConditionsForSpacesCompatibleWithAAm}
  Let $(X,\CalF,\mu)$ and $(Y,\CalG,\nu)$ be $\sigma$-finite measure spaces,
  let $\BanachOne$ be a solid Banach function space on $Y$,
  and $\BanachTwo$ a solid Banach function space on $X$.
  Finally, let $m : X \times Y \to (0,\infty)$, $v : X \to (0,\infty)$,
  and $w : Y \to (0,\infty)$ be measurable and such that $m(x,y) \leq C \cdot v(x) \cdot w(y)$
  for all $(x,y) \in X \times Y$ and some $C > 0$.
  Then the following hold:
  \begin{enumerate}
    \item If $\BanachOne \neq \{0\}$ and if each $K \in \AAm(X,Y)$ induces a well-defined operator
          $\Phi_K : \BanachOne \to \BanachTwo$, then
          \(
            \big( \lebesgue^1(\mu) \cap \lebesgue^\infty(\mu) \big)_{v} \hookrightarrow \BanachTwo .
          \)
          \vspace{0.2cm}

    \item If $\mu(X) \neq 0$ and if for each $0 \leq K \in \AAm(X,Y)$ and each $0 \leq f \in \BanachOne$,
          the function $\Phi_K f : X \to [0,\infty]$ is almost-everywhere finite-valued, then
          \(
            \BanachOne \hookrightarrow \big( \lebesgue^1(\nu) + \lebesgue^\infty (\nu) \big)_{1/w} .
          \)
  \end{enumerate}
\end{corollary}

\begin{rem*}
  Part~(2) of the corollary is strictly stronger than the inclusion
  $\BanachOne \hookrightarrow \mathcal{D} \big( \mathcal{U}, L^1, (L^\infty_{1/w})^\natural \big)$
  obtained in \cite[Lemma~8(a)]{GeneralizedCoorbit1}.
  Furthermore, the space $(\lebesgue^1 + \lebesgue^\infty)_{1/w}$ seems more natural than the space
  $\mathcal{D} \big( \mathcal{U}, L^1, (L^\infty_{1/w})^\natural \big)$,
  the definition of which is more involved (see \mbox{\cite[Pages 261 and 265]{GeneralizedCoorbit1}}).

  To see that the above result is indeed a strict improvement,
  first note that Remark~\ref{rem:NecessaryConditionsSharpnessRemark} shows
  \(
    (\lebesgue^1 + \lebesgue^\infty)_{1/w}
    \hookrightarrow \mathcal{D} \big( \mathcal{U}, L^1, (L^\infty_{1/w})^\natural \big)
    .
  \)
  Here, Condition~\eqref{eq:NecessaryConditionsSharpnessAssumption} is satisfied
  under the assumptions imposed in \cite{GeneralizedCoorbit1}, since these imply that $X = Y$
  and that $\frac{w(y)}{w(x)} \leq m(x,y) \leq w(x) \cdot w(y)$ for $w(x) := m(x,z)$ with $z \in X$ fixed;
  see \cite[Equations~(3.2), (3.3), and (3.7)]{GeneralizedCoorbit1}.

  Finally, if one chooses $(X,\mu) = (\R, \lambda)$ with the Lebesgue measure $\lambda$,
  and $m \equiv w \equiv 1$, as well as $\CalU = (U_k)_{k \in \Z} = \big( (k-1,k+1) \big)_{k \in \Z}$,
  then the function $f := \sum_{n=1}^\infty 2^n \cdot \Indicator_{[n,n+2^{-n}]}$ satisfies
  $f \notin \lebesgue^1 + \lebesgue^\infty$,
  but $f \in \mathcal{D} \big( \mathcal{U}, L^1, (L^\infty_{1/w})^\natural \big)$.
  To see this, note that $\| f \cdot \Indicator_{U_k} \|_{\lebesgue^1} \leq 3$ for all $k \in \Z$
  and hence $f \in \mathcal{D} \big( \mathcal{U}, L^1, (L^\infty_{1/w})^\natural \big)$,
  but that $\int_E f \, d \lambda = \infty$ for the finite measure set
  $E := \bigcup_{n \in \N} [n, n + 2^{-n}]$; this shows $f \notin \lebesgue^1 + \lebesgue^\infty$.
\end{rem*}

\subsection{Boundedness of \texorpdfstring{$\Phi_K : \BanachOne \to \lebesgue^\infty_{1/v}$}
                                          {ΦK from 𝐀 into a weighted 𝑳∞ space}}
\label{sub:IntroEmbeddingIntoWeightedLInfty}

In some applications it is not enough to merely know that an integral kernel $K$ induces
a bounded integral operator $\Phi_K : \BanachOne \to \BanachOne$.
Instead, it might be required that $\Phi_K : \BanachOne \to \BanachThree$ is well-defined and bounded,
for some space $\BanachThree$ that does not necessarily contain $\BanachOne$.
In particular for general coorbit theory---see Section~\ref{sec:IntroCoorbitTheory}---%
such a property is required for $\BanachThree = \lebesgue^\infty_{1/v}$,
for a suitable weight $v$.
Having to verify this additional condition can be a serious obstruction
for the application of coorbit theory.

Using the kernel modules $\BBm(X,X)$, it is possible to simplify proving that
$\Phi_K : \BanachOne \to \lebesgue^\infty_{1/v}$ is indeed bounded for a suitable choice of $v$.
All one needs to verify are the following two conditions:
\begin{enumerate}
  \item $\BBm(X,X) \hookrightarrow \Bounded(\BanachOne, \BanachOne)$,
        which holds for all (suitably) weighted
        mixed-norm Lebesgue spaces $\BanachOne = \lebesgue^{p,q}_w$; and

  \item a certain \emph{maximal kernel} $\MaxKernel{\CalU} K$ of the kernel $K$ belongs to $\BBm(X,X)$.
\end{enumerate}

The maximal kernel $\MaxKernel{\CalU} K$ is defined using a suitable covering $\CalU$
of $X = X_1 \times X_2$.
The following definition clarifies the conditions that this covering has to satisfy,
and formally introduces the maximal kernel $\MaxKernel{\CalU} K$.

\begin{definition}\label{def:ProductAdmissibleCovering}
  Let $\XTuple = \ProductTupleX$, where $\XIndexTuple{i}$ is a $\sigma$-finite measure space
  for $i \in \{ 1, 2 \}$.
  A family $\CalU = (U_j)_{j \in J}$ is said to be a \emph{product-admissible covering} of $X$,
  if it satisfies the following conditions:
  \begin{enumerate}
    \item the index set $J$ is countable;

    \item we have $X = \bigcup_{j \in J} U_j$;
          \vspace{0.05cm}

    \item each $U_j$ is of the form $U_j = V_j \times W_j$
          with $V_j \in \CalF_1$ and $W_j \in \CalF_2$ and we have $\mu(U_j) > 0$;

    \item there is a constant $C > 0$ such that
          the \emph{covering weight} $w_{\CalU}$ defined by
          \begin{equation}
            (w_{\CalU})_j
            := \min \big\{ 1, \mu_1 (V_j), \mu_2 (W_j), \mu(U_j) \big\}
            \qquad \text{for } j \in J
            \label{eq:CoveringWeightDefinition}
          \end{equation}
          satisfies $(w_{\CalU})_i \leq C \cdot (w_{\CalU})_j$ for all $i,j \in J$ for which
          $U_i \cap U_j \neq \emptyset$.
  \end{enumerate}

  Given such a product-admissible covering $\CalU = (U_j)_{j \in J}$
  and any kernel $K : X \times X \to \CC$, the associated \emph{maximal kernel} $\MaxKernel{\CalU} K$
  is defined as
  \begin{equation}
    \MaxKernel{\CalU} K
    : X \times X \to [0,\infty],
    (x,y) \mapsto \sup_{z \in \CalU(x)} |K(z,y)|
    \quad \text{where} \quad
    \CalU(x) := \bigcup_{j \in J \text{ with } x \in U_j} U_j .
    \label{eq:MaximalKernelDefinition}
  \end{equation}

  Finally, given a weight $u : X \to (0,\infty)$, we say that $u$ is \emph{$\CalU$-moderate},
  if there is a constant $C' > 0$ such that $u(x) \leq C' \cdot u(y)$ for all $j \in J$
  and all $x,y \in U_j$.
\end{definition}

Our precise sufficient criterion for ensuring that $\Phi_K : \BanachOne \to \lebesgue^\infty_{1/v}$
reads as follows:

\begin{theorem}\label{thm:BoundednessIntoWeightedLInfty}
  Let $\XTuple = \ProductTupleX$, where $\XIndexTuple{i}$ is a $\sigma$-finite measure space
  for $i \in \{ 1, 2 \}$.
  Let $m : X \times X \to (0,\infty)$ and $u : X \to (0,\infty)$ be measurable and such that
  $m(x,y) \leq C \cdot u(x) \cdot u(y)$ for all $x,y \in X$ and some constant $C > 0$.
  Let $\CalU = (U_j)_{j \in J}$ be a product-admissible covering of $X$ for which $u$
  is \emph{$\CalU$-moderate}.
  Then the following hold:
  \begin{enumerate}
    \item There is a measurable weight $w_{\CalU}^c : X \to (0,\infty)$
          (a ``continuous version'' of the discrete weight $w_{\CalU}$
          defined in \eqref{eq:CoveringWeightDefinition}) such that
          \begin{equation}
            \sup_{j \in J}
              \sup_{x \in U_j}
                \bigg[
                  \frac{(w_{\CalU})_j}{w_{\CalU}^c (x)}
                  + \frac{w_{\CalU}^c (x)}{(w_{\CalU})_j}
                \bigg]
            < \infty .
            \label{eq:ContinuousCoveringWeightCondition}
          \end{equation}
          For any two such weights $w_{\CalU}^c, v_{\CalU}^c$,
          we have $w_{\CalU}^c \asymp v_{\CalU}^c$.
          \vspace{0.1cm}

    \item Let $\BanachOne$ be a solid Banach function space on $X$
          such that $\BBm(X,X) \hookrightarrow \Bounded(\BanachOne)$.
          Let the kernel ${K : X \times X \to \CC}$ be measurable,
          and assume that there is a measurable kernel ${L : X \times X \to [0,\infty]}$
          satisfying $\MaxKernel{\CalU} K \leq L$ and $\| L \|_{\BBm} < \infty$.
          Define
          \begin{equation}
            v : X \to (0,\infty),
                x \mapsto \frac{u(x)}{w_{\CalU}^c (x)} .
            \label{eq:SpecialLInftyWeight}
          \end{equation}
          Then $K \in \BBm(X,X)$ (so that $\Phi_K : \BanachOne \to \BanachOne$ is bounded),
          and $K$ induces a well-defined and bounded integral operator
          ${\Phi_K : \BanachOne \to \lebesgue^\infty_{1/v}}$.
  \end{enumerate}
\end{theorem}

\begin{rem*}
  The slightly convoluted formulation involving the additional kernel $L$
  is only chosen to avoid having to assume that the maximal kernel $\MaxKernel{\CalU} K$
  is measurable.
\end{rem*}

\begin{proof}
  The proof of Theorem~\ref{thm:BoundednessIntoWeightedLInfty}
  is given in Section~\ref{sec:EmbeddingIntoWeightedLInfty}.
\end{proof}

\section{Application: Simplified conditions for coorbit theory using \texorpdfstring{$\BBm(X)$}{𝓑ₘ(X)}}
\label{sec:IntroCoorbitTheory}

The main idea of coorbit theory is to quantify the ``niceness'' of a function
in terms of the decay of its frame coefficients with respect to a fixed continuous frame $\CalF$.
This approach gives rise to a family of function spaces---the so-called \emph{coorbit spaces}---%
and an associated discretization theory for these spaces.

Coorbit theory was originally introduced by Feichtinger and Gröchenig
in the seminal papers~\cite{FeichtingerCoorbit0,FeichtingerCoorbit1,FeichtingerCoorbit2}
for frames generated by an integrable, irreducible group representation.
It was later generalized to arbitrary continuous frame; see \cite{GeneralizedCoorbit1,GeneralizedCoorbit2}.
Our presentation is based on the recent contribution by Kempka et al.~\cite{kempka2015general},
which generalizes and streamlines coorbit theory for continuous frames
(\emph{general coorbit theory}).

In this introduction, we only present a sketch of the theory
and give an honest but simplified account of how the kernel spaces $\BBm$
can be used to simplify the process of verifying the necessary assumptions.
The reader interested in the finer details is referred to Section~\ref{sec:CoorbitTheory}.

Let $\CalH$ be a separable Hilbert space, and let $\XTuple = \ProductTupleX$,
where $\XIndexTuple{j}$ ($j \in \{1,2\}$) is a $\sigma$-compact,
second countable, locally compact Hausdorff space with Borel $\sigma$-algebra $\CalF_j$,
and a Radon measure $\mu_j$ with $\supp \mu_j = X_j$.
Furthermore, fix a \emph{continuous Parseval frame} $\Psi = (\psi_x)_{x \in X} \subset \CalH$,
which by definition means that the \emph{voice transform}
\begin{equation}
  V_\Psi f : X \to \CC, x \mapsto \langle f, \psi_x \rangle_{\CalH}
  \label{eq:VoiceTransformDefinition}
\end{equation}
is measurable for each $f \in \Hil$ and satisfies
\begin{equation}
  \| f \|_{\CalH}^2 = \| V_\Psi f \|_{\lebesgue^2 (\mu)}^2
  \qquad \forall \, f \in \CalH .
  \label{eq:ParsevalFrameCondition}
\end{equation}
We will only be working with Parseval frames as above; for a detailed treatment of general continuous
frames, we refer to \cite{AliContinuousFrames,SpeckbacherReproducingPairs1}.

The idea of coorbit theory is to generalize the description \eqref{eq:ParsevalFrameCondition}
of the $\CalH$-norm in terms of the $L^2$-norm of the voice transform to different decay conditions
regarding the voice transform.
Somewhat more precisely, we say that a solid Banach function space $\BanachOne$ on $X$ is \emph{rich},
if $\Indicator_K \in \BanachOne$ for all compact sets $K \subset X$.
Given such a rich solid Banach function space $\BanachOne$
satisfying certain additional technical conditions, the associated \emph{coorbit space}
$\Co (\BanachOne)$ is defined as
\[
  \Co (\BanachOne)
  := \Co_\Psi (\BanachOne)
  := \big\{
      f \in \Reservoir
      \,\,\colon\,
      V_\Psi f \in \BanachOne
    \big\} ,
  \qquad \text{with norm} \qquad
  \| f \|_{\Co (\BanachOne)}
  := \| V_\Psi f \|_{\BanachOne} .
\]
Here, the space $\Reservoir$---the ``reservoir'' from which the elements of $\Co(\BanachOne)$
are taken---can be thought of as a generalization of the space of (tempered) distributions
to the general setting considered here; see Lemma~\ref{lem:ReservoirConditions}
for the precise definition.

The first component of coorbit theory is a set of ``compatibility criteria''
between the \emph{reproducing kernel} (or Gramian kernel)
\begin{equation}
  K_\Psi : X \times X \to \CC, (x,y) \mapsto \langle \psi_y, \psi_x \rangle_{\CalH}
  \label{eq:ReproducingKernelDefinition}
\end{equation}
of the frame $\Psi$, and the Banach function space $\BanachOne$, which ensure that
the coorbit spaces are actually well-defined Banach spaces.
In their usual formulation, these conditions are quite technical.
Using the kernel spaces $\BBm$, however, one can formulate a set of conditions
which is less tedious to verify.
To conveniently formulate these conditions, we will assume the following:
\begin{equation}
  \begin{split}
    & \CalU = (U_j)_{j \in J}
    \text{ is a product admissible covering of } X, \\
    & u : X \to (0,\infty) \text{ is $\CalU$-moderate and locally bounded}, \\
    & v : X \to [1,\infty)
    \text{ is measurable with } v(x) \gtrsim \max
                                            \big\{
                                              \| \psi_x \|_{\Hil}, \,\,
                                              [w_{\CalU}^c (x)]^{-1} \, u(x)
                                            \big\} , \\
    & m_0 \!:\! X \!\times\! X \to (0,\infty)
    \text{ is measurable with } m_0(x,y) \!\leq\! C \, u(x) \, u(y) \text{ and } m_0(x,y) \!=\! m_0(y,x), \\
    & \BanachOne \neq \{ 0 \}
    \text{ is a solid Banach function space on $X$, and }
    \| \Phi_K \|_{\BanachOne \to \BanachOne} \leq \| K \|_{\BBi_{m_0}} \,\,\, \forall \, K \in \BBi_{m_0} .
  \end{split}
  \label{eq:CoorbitWeightConditions}
\end{equation}
Here, $w_{\CalU}^c$ is as defined in Part~(1) of Theorem~\ref{thm:BoundednessIntoWeightedLInfty}.

\begin{rem*}
\emph{(i):} By means of Proposition~\ref{prop:NewKernelModuleBasicProperties3}, we see that the condition
$\| \Phi_K \|_{\BanachOne \to \BanachOne} \leq \| K \|_{\BBi_{m_0}}$
is satisfied if $\BanachOne = \lebesgue_\kappa^{p,q}(\mu)$
and $\frac{\kappa(x)}{\kappa(y)} \leq m_0 (x,y)$.

\medskip{}

\emph{(ii):} If only $\| \Phi_K \|_{\BanachOne \to \BanachOne} \leq C_0 \, \| K \|_{\BBi_{m_0}}$ holds,
then one can achieve $C_0 = 1$ by replacing $m_0$ with $C_0 \cdot m_0$.
Thus, there is no real loss of generality in assuming $C_0 = 1$.
\end{rem*}

Our simplified condition for the well-definedness of the coorbit spaces reads as follows:
\begin{theorem}\label{thm:IntroductionCoorbitWellDefinedConditions}
  Let $\XTuple$ and $\Psi$ as above, and assume that Condition~\eqref{eq:CoorbitWeightConditions}
  is satisfied.
  With $m_v$ as in Equation~\eqref{eq:SeparableMatrixWeight}
  and $M_{\CalU} K_\Psi$ as in Equation~\eqref{eq:MaximalKernelDefinition}, assume that
  \begin{equation}
    K_\Psi \in \AAi_{m_v}
    \quad \text{and} \quad
    \text{there exists } L \in \BBi_{m_0} \text{ satisfying } M_{\CalU} K_\Psi \leq L .
    \label{eq:CoorbitWellDefinedConditions}
  \end{equation}
  Then all technical assumptions imposed in \cite[Sections 2.3 and 2.4]{kempka2015general}
  are satisfied; consequently, the coorbit space $\Co_\Psi (\BanachOne)$ is a well-defined Banach space.
\end{theorem}

\begin{rem*}
  The conclusion of the theorem is deliberately left somewhat vague;
  the finer details are given in Proposition~\ref{pro:coorbitsWithBm}.
\end{rem*}

The second component of coorbit theory is the \emph{discretization theory}.
This provides conditions under which a sampled version $\Psi_d = (\psi_{x_i})_{i \in I}$
of the continuous frame $\Psi = (\psi_x)_{x \in X}$ can be used to describe the coorbit space
$\Co_{\Psi} (\BanachOne)$.
More precisely, if the discretization theory is applicable, then there are solid sequence spaces
$\BanachOne^{\flat}, \BanachOne^{\sharp} \subset \CC^{I}$ associated to $\BanachOne$
(see Equation~\eqref{eq:SequenceSpaceNorms}) such that the coefficient and synthesis maps
\begin{equation}
  C_{\Psi_d} : \Co_\Psi(\BanachOne) \to \BanachOne^{\flat},
              f \mapsto \big( \langle f, \psi_{x_i} \rangle \big)_{i \in I}
  \quad \text{and} \quad
  D_{\Psi_d} : \BanachOne^{\sharp} \to \Co_\Psi (\BanachOne),
              (c_i)_{i \in I} \mapsto \sum_{i \in I} c_i \, \psi_{x_i}
  \label{eq:CoefficientSynthesisOperator}
\end{equation}
are well-defined and bounded, and such that $C_{\Psi_d}$ has a bounded linear left-inverse,
while $D_{\Psi_d}$ has a bounded linear right-inverse.
In the language of \cite{GroechenigDescribingFunctions}, the former means
that the sampled frame $\Psi_d$ forms a \emph{Banach frame}
and the latter implies that $\Psi_d$ is an \emph{atomic decomposition}
for $\Co_\Psi (\BanachOne)$.
Put briefly, we say that $\Psi_d$ forms a \emph{Banach frame decomposition} of $\Co_\Psi (\BanachOne)$
if both conditions are satisfied.

To conveniently state our simplified criteria for the applicability of the discretization theory,
we require the notion of the \emph{oscillation} $\oscUG (K)$ of a kernel $K : X \times X \to \CC$
with respect to a covering $\CalU = (U_i)_{i \in I}$ and a (not necessarily measurable)
\emph{phase function} $\Gamma : X \times X \to S^1$, where $S^1 := \{ z \in \CC \colon |z| = 1 \}$.
This oscillation is defined as
\begin{equation}
  \oscUG(K) : X \times X \rightarrow [0,\infty], \quad
              (x, y) \mapsto \sup_{z \in \CalU(y)}
                               \big| K(x, y) - \Gamma(y, z) K(x, z) \big|,
  \label{eq:OscillationDefinition}
\end{equation}
with $\CalU(y) := \bigcup_{i \in I \text{ with } y \in U_i} U_i$
as in Equation~\eqref{eq:MaximalKernelDefinition}.

Furthermore, we need the notion of an \emph{admissible covering}.
Precisely, a family $\CalU = (U_i)_{i \in I}$ of subsets of $X$ is an \emph{admissible covering}
of $X$, if it satisfies the following properties:
\begin{enumerate}
 \item $\CalU$ is a covering of $X$; that is, $X = \bigcup_{i \in I} U_i$;

 \item $\CalU$ is locally finite, meaning that each $x \in X$ has an open neighborhood
        intersecting only finitely many of the $U_i$;

 \item each $U_i$ is measurable, relatively compact, and has non-empty interior;

 \item the intersection number
        \(
          \sigma(\CalU)
          := \sup_{i \in I} |\{ \ell \in I \colon U_\ell \cap U_i \neq \emptyset \}|
        \)
        is finite.
\end{enumerate}

\begin{theorem}\label{thm:IntroductionCoorbitDiscretizationConditions}
  Suppose that the assumptions of Theorem~\ref{thm:IntroductionCoorbitWellDefinedConditions}
  hold, and define ${m := m_v + m_0}$.
  If there exist
  an admissible covering $\widetilde{\CalU} = (\widetilde{U}_i)_{i \in I}$ of $X$,
  a phase function $\Gamma : X \times X \to S^1$, and some $L \in \BBm$ such that
  \[
    \oscUtG (K_\Psi) \leq L
    \quad \text{ and } \quad
    \| L \|_{\BBm} \cdot \bigl( 2 \, \| K_\Psi \|_{\BBm} + \| L \|_{\BBm} \bigr) < 1,
  \]
  and if for each $i \in I$ some $x_i \in \widetilde{U}_i$ is chosen,
  then all technical assumptions imposed in \mbox{\cite[Theorem~2.48]{kempka2015general}}
  are satisfied; consequently, $\Psi_d = (\psi_{x_i})_{i \in I}$ is a Banach frame decomposition
  for $\Co_\Psi(\BanachOne)$.
\end{theorem}

\begin{rem*}
  Again, the conclusion of the theorem is deliberately not completely precise;
  a rigorous version is given in Proposition~\ref{pro:coorbitsWithBm2}.
\end{rem*}

\section{Proving \texorpdfstring{Schur's}{Schurʼs} test for weighted mixed-norm Lebesgue spaces}
\label{sec:MixedNormSchur}

In this section, we prove our generalization of Schur's test
to (weighted) mixed-norm Lebesgue spaces; that is,
we prove Theorems~\ref{thm:SchurTestSufficientUnweighted} and \ref{thm:SchurNecessity}.
We will assume throughout that $(X_i, \CalF_i, \mu_i)$ and $(Y_i, \CalG_i, \nu_i)$ are
$\sigma$-finite measure spaces for $i = 1,2$.
Define $X := X_1 \times X_2$ and $Y := Y_1 \times Y_2$, which we equip with the respective
product $\sigma$-algebras $\CalF := \CalF_1 \otimes \CalF_2$ or $\CalG := \CalG_1 \otimes \CalG_2$,
and with the product measures $\mu := \mu_1 \otimes \mu_2$ and $\nu := \nu_1 \otimes \nu_2$.

\subsection{Sufficiency of the generalized Schur condition}
\label{sub:MixedNormSchurSufficient}

In this subsection, we prove that the integral operator
$\Phi_K : \lebesgue^{p,q}(\nu) \to \lebesgue^{p,q}(\mu)$ is indeed well-defined and bounded
if $C_i (K) < \infty$ for all $i \in \{1,2,3,4\}$,
with $C_i (K)$ as defined in Equation~\eqref{eq:MixedNormSchurConstants}.

To simplify the subsequent proofs, we first show that it suffices to only consider non-negative
kernels $K \geq 0$ and functions $f \geq 0$, and show that
\(
  \|\Phi_K \, f\|_{\lebesgue^{p,q}(\mu)}
  \leq [\max_{i \in J} C_i (K)]
       \cdot \|f\|_{\lebesgue^{p,q}(\nu)}
\)
for some subset $J \subset \{1,2,3,4\}$ that might depend on $p,q$.

To see why this is enough, suppose we know that
\(
  \|\Phi_{|K|} \, f\|_{\lebesgue^{p,q}(\mu)}
  \leq \bigl[ \max_{i \in J} C_i (|K|) \bigr]
       \cdot \|f\|_{\lebesgue^{p,q}(\nu)}
\)
for all $f \geq 0$.
Now, first note that $C_i (K) = C_i (|K|)$;
hence, if $C_i (K) < \infty$, then $C_i (|K|) < \infty$ as well.
Next, note that if $C_i (K) < \infty$ for all $i \in J$, then
\[
  |\Phi_K f (x)|
  \leq \int_{Y}
         |K|(x,y) \, |f|(y)
       \, d \nu(y)
  =    [\Phi_{|K|} \, |f|](x)
  <    \infty
\]
for $f \in \lebesgue^{p,q}(\nu)$ and almost all $x \in X$.
Hence, $\Phi_K \, f$ is an almost-everywhere well-defined function
which is measurable by Fubini's theorem, and which satisfies
\[
  \|\Phi_K \, f\|_{\lebesgue^{p,q}(\mu)}
  \leq \big\| \, \Phi_{|K|} |f| \, \big\|_{\lebesgue^{p,q}(\mu)}
  \leq \big[ \max_{i \in J} C_i (|K|) \big] \cdot \| \, |f| \, \|_{\lebesgue^{p,q}(\nu)}
  =    \big[ \max_{i \in J} C_i (K) \big] \cdot \| f \|_{\lebesgue^{p,q}(\nu)},
\]
as desired.
Therefore, in the following proofs we will concentrate on the case where $K,f$ are non-negative.
This has the advantage that $\Phi_K \, f : X \to [0,\infty]$ is always a well-defined measurable
function, as a consequence of Tonelli's theorem.

\medskip{}

We start our proof by considering the case $p \leq q$.
In the proof and in the remainder of the paper, we will frequently use the (elementary) estimate
\begin{equation}
  \esssup_{\omega \in \Omega}
    \int_{\Lambda}
      F(\omega,\lambda)
    \, d \mu(\lambda)
  \leq \int_{\Lambda}
         \esssup_{\omega \in \Omega}
           F(\omega,\lambda)
       \, d \mu(\lambda),
  \label{eq:EsssupOfIntegral}
\end{equation}
which holds for any measurable function $F : \Omega \times \Lambda \to [0,\infty]$,
where $(\Omega,\nu)$ and $(\Lambda,\mu)$ are $\sigma$-finite measure spaces.
To see this, first note that $\lambda \mapsto \esssup_{\omega \in \Omega} F(\omega,\lambda)$
is measurable, for instance as a consequence of Lemma~\ref{lem:CountableLInfinityCharacterization}.
Next, Equation~\eqref{eq:EsssupOfIntegral} is trivial if the right-hand side is infinite.
On the other hand, if the right-hand side is finite,
we see for arbitrary $g \in \lebesgue^1 (\nu)$ with $g \geq 0$ that
\[
  \int_\Omega \!
    g(\omega) \!\!
    \int_\Lambda \!\!
      F(\omega,\lambda)
    \, d\mu(\lambda)
  \, d \nu(\omega)
  = \!\! \int_\Lambda \!
           \int_\Omega \!
             g(\omega)
             F(\omega,\lambda)
           \, d \nu(\omega)
         \, d\mu(\lambda)
  \leq \!
       \|g\|_{\lebesgue^1(\nu)} \!
       \int_\Lambda \!
         \esssup_{\omega \in \Omega}
           F(\omega,\lambda)
       \, d\mu(\lambda).
\]
In view of the characterization of the $\lebesgue^\infty(\nu)$-norm by duality
(see \mbox{\cite[Theorem~6.14]{FollandRA}}), this implies
\(
  \|\omega \mapsto \int_\Lambda F(\omega,\lambda) \, d \mu(\lambda)\|_{\lebesgue^\infty(\nu)}
  \leq \int_\Lambda \esssup_{\omega \in \Omega} F(\omega,\lambda) \, d \mu (\lambda)
\),
which is precisely \eqref{eq:EsssupOfIntegral}.

With these preparations, we now prove Theorem~\ref{thm:SchurTestSufficientUnweighted}
for the case $p \leq q$.

\begin{proposition}\label{prop:SchurTestMixedCase1}
  Let $\XIndexTuple{i}$ and $\YIndexTuple{i}$ be $\sigma$-finite measure spaces
  for $i \in \{ 1, 2 \}$, and let
  $\XTuple = \ProductTupleX$ and $\YTuple = \ProductTupleY$.

  Assume that $K : X \times Y \to \CC$ is measurable and satisfies
  $C_i (K) < \infty$ for $i \in \{1,2,3\}$, where $C_i (K)$ is as defined in
  Equation~\eqref{eq:MixedNormSchurConstants}.
  Finally, assume that $p,q \in [1, \infty]$ with $p \leq q$.
  Then
  \[
    \Phi_{K}
    : \lebesgue^{p,q}(\nu) \to \lebesgue^{p,q}(\mu)
  \]
  is well-defined (with absolute convergence of the defining integral for $\mu$-almost all $x \in X$)
  and bounded, with
  \(
    \vertiii{\Phi_{K}}_{\lebesgue^{p,q}(\nu) \to \lebesgue^{p,q}(\mu)}
    \leq \max_{1 \leq i \leq 3} C_i (K)
    .
  \)
\end{proposition}

\begin{proof}
  As discussed at the beginning of Section~\ref{sub:MixedNormSchurSufficient},
  it suffices to consider non-negative kernels $K \geq 0$, and to show that
  \(
    \| \Phi_K \, f \|_{\lebesgue^{p,q}(\mu)}
    \leq \max \{ C_1 (K) , C_2 (K), C_3(K) \} \cdot \|f\|_{\lebesgue^{p,q}(\nu)}
  \)
  for $f \in \lebesgue^{p,q} (\nu)$ with $f \geq 0$.
  We divide the proof of this fact into three steps.

  \medskip{}

  \noindent
  \textbf{Step 1} \emph{(the case $p = \infty$):}
  This implies $q = \infty$, since $p \leq q$.
  For $f \in \lebesgue^{p,q}(\nu) = \lebesgue^\infty(\nu)$ with $f \geq 0$, we have
  \[
    0
    \leq \Phi_K f (x)
    \leq \int_Y K(x,y) \, f(y) \, d \nu (y)
    \leq \int_Y K(x,y) \, d \nu (y) \cdot \|f\|_{\lebesgue^{\infty}(\nu)}
    \leq C_1 (K) \cdot \|f\|_{\lebesgue^{p,q}(\nu)}
    <    \infty
  \]
  for $\mu$-almost all $x \in X$, and hence
  $\| \Phi_K f \|_{\lebesgue^{p,q}(\mu)} \leq C_1 (K) \cdot \| f \|_{\lebesgue^{p,q}(\nu)}$.

  \medskip{}

  \noindent
  \textbf{Step 2} \emph{(the case $p = 1$):}
  If we also have $q = 1$, then $\lebesgue^{p,q} = \lebesgue^1$, so that the standard version
  of Schur's test (see \cite[Theorem~6.18]{FollandRA}) shows that
  $\Phi_K : \lebesgue^1(\nu) \to \lebesgue^1 (\mu)$ is well-defined and bounded, with
  \(
    \vertiii{\Phi_K}_{\lebesgue^{p,q} \to \lebesgue^{p,q}}
    = \vertiii{\Phi_K}_{\lebesgue^1 \to \lebesgue^1}
    \leq \max \{ C_1(K), C_2(K) \}
    ,
  \)
  as desired.

  For the case $q \in (1,\infty]$, define
  \[
    H :
    X_2 \times Y \to [0,\infty],
    (x_2, y) \mapsto \int_{X_1} K \big( (x_1,x_2), y \big) \, d \mu_1 (x_1),
  \]
  and note that
  \(
    \int_{X_2} H(x_2, y) \, d \mu_2(x_2)
    \leq C_2 (K)
  \)
  for $\nu$-almost all $y \in Y$.

  Now, we distinguish two cases.
  First, if $q = \infty$, we see for $f \in \lebesgue^{p,q}(\nu)$ with $f \geq 0$ that
  \begin{align*}
    \|\Phi_K f\|_{\lebesgue^{p,q}(\mu)}
    & = \esssup_{x_2 \in X_2}
          \int_{X_1}
            \int_Y
              K \big( (x_1,x_2), y \big)
              f(y)
            \, d \nu(y)
          \, d \mu_1 (x_1) \\
    ({\scriptstyle{\text{Tonelli's theorem}}})
    & = \esssup_{x_2 \in X_2}
          \int_{Y_2}
            \int_{Y_1}
              H \big( x_2, (y_1,y_2) \big)
              f(y_1,y_2)
            \, d\nu_1(y_1)
          \, d \nu_2(y_2) \\
    ({\scriptstyle{\text{Hölder}}})
    & \leq \esssup_{x_2 \in X_2}
             \int_{Y_2}
               \|f(\bullet,y_2)\|_{\lebesgue^1 (\nu_1)}
               \esssup_{y_1 \in Y_1}
                 H \big( x_2, (y_1,y_2) \big)
             \, d \nu_2(y_2) \\
    & \leq \|f\|_{\lebesgue^{1,\infty}(\nu)} \cdot C_3 (K) < \infty.
  \end{align*}

  It remains to consider the case $q \in (1,\infty)$.
  Here, we see by repeated applications of Tonelli's theorem and Hölder's inequality
  for $f \in \lebesgue^{p,q}(\nu)$ and $h \in \lebesgue^{q'}(\mu_2)$
  with $f, h \geq 0$ that
  \begin{align*}
    & \int_{X_2}
        \big\| (\Phi_K f)(\bullet, x_2) \big\|_{\lebesgue^1(\mu_1)} \cdot h(x_2)
      \, d \mu_2 (x_2) \\
    ({\scriptstyle{\text{Tonelli}}})
    & = \int_Y
          f(y)
          \int_{X_2}
            h(x_2) \cdot [H(x_2, y)]^{\frac{1}{q'} + \frac{1}{q}}
          \, d \mu_2 (x_2)
        \, d \nu(y) \\
    ({\scriptstyle{\text{Hölder}}})
    & \leq \int_Y
             f(y)
             \left[
               \int_{X_2}
                 [h(x_2)]^{q'}
                 H(x_2, y)
               \, d \mu_2(x_2)
             \right]^{\frac{1}{q'}}
             \left[
               \int_{X_2}
                 H(x_2,y)
               \, d \mu_2(x_2)
             \right]^{\frac{1}{q}}
           \, d \nu(y) \\
    ({\scriptstyle{\text{Tonelli}}})
    & \leq [C_2 (K)]^{\frac{1}{q}}
           \int_{Y_2}
             \int_{Y_1}
               \!\! f(y_1, y_2)
               \left[
                 \int_{X_2}
                   [h(x_2)]^{q'}
                   H \big( x_2, (y_1, y_2) \big)
                 \, d \mu_2(x_2)
               \right]^{\frac{1}{q'}} \!\!\!
             \, d \nu_1(y_1)
           \, d \nu_2(y_2) \\
    ({\scriptstyle{\text{Eq.~}\eqref{eq:EsssupOfIntegral}, \, p = 1}})
    & \leq [C_2 (K)]^{\frac{1}{q}} \!
           \int_{Y_2} \!\!
             \|f(\bullet,y_2)\|_{\lebesgue^{p}(\nu_1)} \!
             \left[ \!
               \int_{X_2} \!\!
                 [h(x_2)]^{q'}
                 \esssup_{y_1 \in Y_1}
                   H \big( x_2, (y_1,y_2) \big)
               \, d \mu_2(x_2)
             \right]^{\!\frac{1}{q'}} \!\!\!\!
           d \nu_2(y_2) \\
    ({\scriptstyle{\text{Hölder, Tonelli}}})
    & \leq [C_2 (K)]^{\frac{1}{q}}
           \, \|f\|_{\lebesgue^{p,q}(\nu)}
           \left[
             \int_{X_2}
               [h(x_2)]^{q'}
               \int_{Y_2} \!
                 \esssup_{y_1 \in Y_1}
                   H \big( x_2, (y_1,y_2) \big)
               \, d \nu_2(y_2)
             d \mu_2(x_2)
           \right]^{\frac{1}{q'}} \\
    & \leq [C_2 (K)]^{\frac{1}{q}}
           \, \|f\|_{\lebesgue^{p,q}(\nu)}
           [C_3 (K)]^{\frac{1}{q'}}
           \left[
             \int_{X_2}
               [h(x_2)]^{q'}
             d \mu_2(x_2)
           \right]^{\frac{1}{q'}} \\
    & =    [C_2(K)]^{\frac{1}{q}}
           [C_3(K)]^{\frac{1}{q'}}
           \|f\|_{\lebesgue^{p,q}(\nu)}
           \|h\|_{\lebesgue^{q'} (\mu_2)}
      < \infty.
  \end{align*}
  Using the characterization of the $\lebesgue^{q}(\mu_2)$-norm by duality
  (see \cite[Theorem~6.14]{FollandRA}), and recalling again that $p = 1$, this implies
  \(
    \|\Phi_K f\|_{\lebesgue^{p,q}(\mu)}
    \leq [C_2(K)]^{\frac{1}{q}} \,
         [C_3(K)]^{\frac{1}{q'}} \,
         \|f\|_{\lebesgue^{p,q}(\nu)}
    <    \infty
  \).
  This completes the proof for the case $p =1$.

  \medskip{}

  \noindent
  \textbf{Step 3} \emph{(the case $p \in (1,\infty)$):}
  Define $r := \tfrac{q}{p}$ (with the understanding that $r = \infty$ if $q = \infty$),
  and note $r \in [1,\infty]$ since $p \leq q$.
  Let $f \in \lebesgue^{p,q}(\nu)$ with $f \geq 0$, and set $g := f^p$.
  A straightforward calculation shows that $g \in \lebesgue^{1,r}(\nu)$, with
  $\|g\|_{\lebesgue^{1,r}(\nu)} = \|f\|_{\lebesgue^{p,q}(\nu)}^{p}$.
  Furthermore, we see by Hölder's inequality and by definition of $C_1 (K)$ that
  \begin{align*}
    0   \leq (\Phi_K f)(x)
      & =    \int_{Y} [K(x,y)]^{\frac{1}{p'} + \frac{1}{p}} f(y) \, d \nu (y)
        \leq \biggl[
               \int_Y K(x,y) \, d \nu(y)
             \biggr]^{\frac{1}{p'}}
             \biggl[
               \int_Y K(x,y) g(y) \, d \nu(y)
             \biggr]^{\frac{1}{p}} \\
      & \leq [C_1(K)]^{\frac{1}{p'}} \cdot [(\Phi_K \, g)(x)]^{\frac{1}{p}}
  \end{align*}
  for $\mu$-almost all $x \in X$.
  Therefore, by employing the result from Step~2, we see that
  \begin{align*}
    \|\Phi_K f\|_{\lebesgue^{p,q}(\mu)}
    & \leq [C_1(K)]^{\frac{1}{p'}}
           \cdot \Big\|
                   x_2 \mapsto \big\|
                                 [(\Phi_K \, g) (\bullet, x_2)]^{1/p}
                               \big\|_{\lebesgue^p(\mu_1)}
                 \Big\|_{\lebesgue^q(\mu_2)} \\
    & =    [C_1(K)]^{\frac{1}{p'}}
           \cdot \Big\|
                   x_2 \mapsto \big\|
                                 (\Phi_K \, g) (\bullet, x_2)
                               \big\|_{\lebesgue^1(\mu_1)}
                 \Big\|_{\lebesgue^r(\mu_2)}^{1/p} \\
    & \leq [C_1(K)]^{\frac{1}{p'}}
           \cdot \Big[ \|g\|_{\lebesgue^{1,r}} \cdot \max_{1 \leq i \leq 3} C_i (K) \Big]^{1/p}
      \leq \|f\|_{\lebesgue^{p,q}(\nu)} \cdot \max_{1 \leq i \leq 3} C_i (K)
      <    \infty.
      \qedhere
  \end{align*}
\end{proof}

For the proof of the case $p > q$, we will make use of the following duality result:

\begin{lemma}\label{lem:AdjointBoundedness}
  Let $(X,\CalF,\mu) = (X_1 \times X_2, \CalF_1 \otimes \CalF_2, \mu_1 \otimes \mu_2)$
  and $(Y,\CalG,\nu) = (Y_1 \times Y_2, \CalG_1 \otimes \CalG_2, \nu_1 \otimes \nu_2)$
  and assume that $\mu_1,\mu_2,\nu_1,\nu_2$ are $\sigma$-finite.

  Let $K : X \times Y \to [0,\infty]$ be measurable, and define the \emph{transposed kernel} as
  \begin{equation}
    K^T : Y \times X \to [0,\infty], (y,x) \mapsto K(x,y) .
    \label{eq:TransposedKernelDefinition}
  \end{equation}
  Let $p,q \in [1,\infty]$.
  Then $\Phi_K : \lebesgue^{p,q}(\nu) \to \lebesgue^{p,q}(\mu)$ is well-defined and bounded
  if and only if $\Phi_{K^T} : \lebesgue^{p',q'}(\mu) \to \lebesgue^{p',q'}(\nu)$ is.
  In this case,
  \(
    \vertiii{\Phi_K}_{\lebesgue^{p,q} \to \lebesgue^{p,q}}
    = \vertiii{\Phi_{K^T}}_{\lebesgue^{p',q'} \to \lebesgue^{p',q'}} .
  \)
\end{lemma}

The proof of the above result is based on the following characterization
of the mixed $\lebesgue^{p,q}$-norm by duality.

\begin{theorem}\label{thm:MixedNormDuality}(easy consequence of \cite[Theorem~2 in Section~2]{MixedLpSpaces})
  Let $(Y_1, \CalG_1, \nu_1)$ and $(Y_2, \CalG_2, \nu_2)$ be $\sigma$-finite measure spaces.
  Let $f : Y_1 \times Y_2 \to \CC$ be measurable, and let $p,q \in [1, \infty]$.
  Then
  \[
    \|f\|_{\lebesgue^{p,q}(\nu)}
    = \sup_{\substack{g : Y_1 \times Y_2 \to [0,\infty)\\
                      \|g\|_{\lebesgue^{p',q'}} \leq 1}}
        \int_{Y}
          |f(y)| \cdot g(y)
        \, d \nu (y) .
  \]
\end{theorem}

\begin{proof}[Proof of Lemma~\ref{lem:AdjointBoundedness}]
  By symmetry (noting that $(K^T)^T = K$), it is enough to prove that if
  the operator ${\Phi_{K^T} : \lebesgue^{p',q'}(\mu) \to \lebesgue^{p',q'}(\nu)}$
  is bounded, then so is $\Phi_K : \lebesgue^{p,q}(\nu) \to \lebesgue^{p,q}(\mu)$,
  with operator norm
  \(
    \vertiii{\Phi_K}_{\lebesgue^{p,q} \to \lebesgue^{p,q}}
    \leq \vertiii{\Phi_{K^T}}_{\lebesgue^{p',q'} \to \lebesgue^{p',q'}} .
  \)
  Furthermore, as explained at the beginning of Section~\ref{sub:MixedNormSchurSufficient},
  it is enough to prove
  \(
    \| \Phi_K f \|_{\lebesgue^{p,q}(\mu)}
    \leq \vertiii{\Phi_{K^T}}_{\lebesgue^{p',q'} \to \lebesgue^{p',q'}}
         \cdot \| f \|_{\lebesgue^{p,q}(\nu)}
  \)
  for $f \in \lebesgue^{p,q}(\nu)$ with $f \geq 0$;
  here we use that $K$ is non-negative in Lemma~\ref{lem:AdjointBoundedness}.

  To see that the last inequality holds, let $f \in \lebesgue^{p,q}(\nu)$
  and $g \in \lebesgue^{p',q'}(\mu)$ with $f,g \geq 0$.
  Tonelli's theorem (which is applicable since all involved functions are non-negative),
  combined with Hölder's inequality for the mixed-norm Lebesgue spaces
  (see \cite[Equation~(1) in Section~2]{MixedLpSpaces}) shows that
  \begin{equation}
    \begin{split}
      \int_X
        [\Phi_K \, f] (x) \cdot g(x)
      \, d \mu(x)
      & = \int_Y
            \int_X
              K^T(y,x) \, f(y) \, g(x)
            \, d \mu(x)
          \, d \nu(y) \\
      & = \int_{Y}
            f(y) \cdot [\Phi_{K^T} \, g] (y)
          \, d \nu(y)
        \leq \|f\|_{\lebesgue^{p,q}(\nu)}
             \cdot \|\Phi_{\! K^T} \, g\|_{\lebesgue^{p',q'}(\nu)} \\
      & \leq \vertiii{\Phi_{K^T}}_{\lebesgue^{p',q'} \to \lebesgue^{p',q'}}
             \cdot \|f\|_{\lebesgue^{p,q}(\nu)}
             \cdot \|g\|_{\lebesgue^{p',q'}(\mu)} < \infty .
    \end{split}
    \label{eq:TransposedKernelBoundednessProof}
  \end{equation}
  On the one hand, this implies%
  \footnote{Otherwise, since $\mu_1,\mu_2$ are $\sigma$-finite, there would be sets $M \in \CalF$,
  as well as $A \in \CalF_1$, $B \in \CalF_2$ with $\mu_1(A) < \infty$ and $\mu_2(B) < \infty$
  and such that $\mu(M \cap (A \times B)) > 0$ and $\Phi_K f \equiv \infty$ on $M \cap (A \times B)$.
  But then the left-hand side of the inequality \eqref{eq:TransposedKernelBoundednessProof}
  would be infinite for the choice $g := \Indicator_{M \cap (A \times B)}$,
  although $g \in \lebesgue^{p',q'}(\mu)$.}
  that $[\Phi_K f](x) < \infty$ for $\mu$-almost all $x \in X$.
  Then, using the dual characterization of the $\lebesgue^{p,q}$-norm
  from Theorem~\ref{thm:MixedNormDuality}, we see that $\Phi_K f \in \lebesgue^{p,q} (\mu)$,
  and furthermore
  \({
    \|\Phi_K f\|_{\lebesgue^{p,q}(\mu)}
    \leq \vertiii{\Phi_{K^T}}_{\lebesgue^{p',q'} \to \lebesgue^{p',q'}}
         \cdot \|f\|_{\lebesgue^{p,q}(\nu)}
  }\).
\end{proof}

Finally, we will also use the following relation between the constants $C_i (K)$
from Equation~\eqref{eq:MixedNormSchurConstants} for the kernel $K$ and the constants
$C_i (K^T)$ for the transposed kernel.

\begin{lemma}\label{lem:SchurConstantsForAdjointKernel}
  Let $(X,\CalF,\mu) = (X_1 \times X_2, \CalF_1 \otimes \CalF_2, \mu_1 \otimes \mu_2)$
  and $(Y,\CalG,\nu) = (Y_1 \times Y_2, \CalG_1 \otimes \CalG_2, \nu_1 \otimes \nu_2)$
  and assume that $\mu_1,\mu_2,\nu_1,\nu_2$ are $\sigma$-finite.

  Let $K : X \times Y \to [0,\infty]$ be measurable, and let the transposed kernel
  $K^T : Y \times X \to [0,\infty]$ be as in Equation~\eqref{eq:TransposedKernelDefinition}.
  Then the constants $C_i$ introduced in Equation~\eqref{eq:MixedNormSchurConstants} satisfy
  \[
    C_1(K^T) = C_2 (K),
    \quad
    C_2 (K^T) = C_1 (K),
    \quad
    C_3 (K^T) = C_4(K),
    \quad \text{and} \quad
    C_4 (K^T) = C_3 (K).
  \]
\end{lemma}

\begin{proof}
  The assertion follows easily from the definitions.
\end{proof}

With this, we can now handle the case $p > q$.

\begin{proposition}\label{prop:SchurTestMixedCase2}
  Let $\XIndexTuple{i}$ and $\YIndexTuple{i}$ be $\sigma$-finite measure spaces for $i \in \{ 1,2 \}$,
  and let $\XTuple = \ProductTupleX$ and $\YTuple = \ProductTupleY$.

  Assume that $K : X \times Y \to \CC$ is measurable, and that the constants
  $C_1(K)$, $C_2(K)$, and $C_4(K)$ introduced in Equation~\eqref{eq:MixedNormSchurConstants}
  are finite.
  Furthermore, let $p,q \in [1, \infty]$ with $p > q$.
  Then
  \[
    \Phi_K : \lebesgue^{p,q}(\nu) \to \lebesgue^{p,q}(\mu)
  \]
  is well-defined and bounded
  (with absolute convergence of the defining integral for $\mu$-almost all $x \in X$), with
  \(
    \vertiii{\Phi_K}_{\lebesgue^{p,q} \to \lebesgue^{p,q}}
    \leq \max \big\{ C_1(K), C_2(K), C_4(K) \big\}
    .
  \)
\end{proposition}

\begin{proof}
  Let $L := |K|^T : Y \times X \to [0,\infty)$ denote the transposed kernel
  of the (pointwise) absolute value of $K$.
  By Lemma~\ref{lem:SchurConstantsForAdjointKernel} and the explanation at the beginning
  of Subsection~\ref{sub:MixedNormSchurSufficient}, we see that
  ${C_1 (L) = C_2(|K|) = C_2(K) < \infty}$,
  $C_2 (L) = C_1(|K|) = C_1(K) < \infty$,
  and finally also ${C_3(L) = C_4(|K|) = C_4 (K) < \infty}$.
  Furthermore, since $p^{-1} < q^{-1}$, we see that the conjugate exponents $p', q'$
  satisfy $\tfrac{1}{p'} = 1 - p^{-1} > 1 - q^{-1} = \tfrac{1}{q'}$, and hence $p' < q'$.

  Therefore, Proposition~\ref{prop:SchurTestMixedCase1} shows that
  $\Phi_L : \lebesgue^{p',q'}(\mu) \to \lebesgue^{p',q'}(\nu)$ is well-defined and bounded, with
  \(
    \vertiii{\Phi_L}_{\lebesgue^{p',q'} \to \lebesgue^{p',q'}}
    \leq \max_{1 \leq i \leq 3} C_i (L)
    = \max\{ C_1(K), C_2(K), C_4(K) \} =: C
    .
  \)
  Since $L^T = |K|$, Lemma~\ref{lem:AdjointBoundedness} shows that
  $\Phi_{|K|} : \lebesgue^{p,q}(\nu) \to \lebesgue^{p,q}(\mu)$ is bounded, with
  $\vertiii{\Phi_{|K|}}_{\lebesgue^{p,q} \to \lebesgue^{p,q}} \leq C$.
  Finally, the reasoning from the beginning of Subsection~\ref{sub:MixedNormSchurSufficient}
  shows that $\Phi_K : \lebesgue^{p,q}(\nu) \to \lebesgue^{p,q}(\mu)$ is bounded, with
  \(
    \vertiii{\Phi_K}_{\lebesgue^{p,q} \to \lebesgue^{p,q}}
    \leq \vertiii{\Phi_{|K|}}_{\lebesgue^{p,q} \to \lebesgue^{p,q}}
    \leq C
    ,
  \)
  as claimed.
\end{proof}

By combining Propositions~\ref{prop:SchurTestMixedCase1} and \ref{prop:SchurTestMixedCase2},
we obtain Theorem~\ref{thm:SchurTestSufficientUnweighted}.

\subsection{Necessity of the generalized Schur condition}
\label{sub:MixedNormSchurNecessary}

In this subsection, we prove Theorem~\ref{thm:SchurNecessity}.
For this, we will need the following technical result,
the proof of which we defer to Appendix~\ref{sec:CountableDualityCharacterization}.

\begin{lemma}\label{lem:CountableDualityCharacterization}
  Let $(\Omega, \CalC)$ be a measurable space, and let
  $(Y, \CalG, \nu) = (Y_1 \times Y_2, \CalG_1 \otimes \CalG_2, \nu_1 \otimes \nu_2)$,
  where $\nu_1, \nu_2$ are $\sigma$-finite measures.

  Finally, let $H : \Omega \times Y \to [0,\infty]$ be measurable.
  Then there is a countable family $(h_n)_{n \in \N}$ of measurable functions
  $h_n : Y \to [0,\infty)$ such that $\| h_n \|_{\lebesgue^{1,\infty}(\nu)} \leq 1$,
  $h_n \in \lebesgue^1(\nu)$, and
  \[
    \| H(\omega, \bullet) \|_{\lebesgue^{\infty,1}(\nu)}
    = \sup_{n \in \N}
        \int_{Y}
          H(\omega, y) \cdot h_n (y)
        \, d \nu (y)
    \qquad \forall \, \omega \in \Omega.
  \]
\end{lemma}

\begin{rem*}
  The claim is non-trivial, since usually neither of the spaces $\lebesgue^{\infty,1} (\nu)$
  or $\lebesgue^{1,\infty} (\nu)$ is separable.
  Furthermore, it should be noted that \emph{the countability of the family $(h_n)_{n \in \N}$
  is crucial for our purposes}.
  Indeed, for each $\omega \in \Omega$ and $k \in \N$, it follows by the dual characterization
  of the mixed Lebesgue norm (see Theorem~\ref{thm:MixedNormDuality}) that there is a function
  $h_{\omega,k} \geq 0$ satisfying $\| h_{\omega,k} \|_{\lebesgue^{1,\infty} (\nu)} \leq 1$ and
  \(
    \int_Y H(\omega,y) h_{\omega,k}(y) \, d \nu(y)
    \geq (1-k^{-1}) \| H(\omega,\bullet) \|_{\lebesgue^{\infty,1}(\nu)} .
  \)
  Yet, in the proof of Part~(3) of Theorem~\ref{thm:SchurNecessity}, we will have to introduce
  an exceptional null-set $N_n \subset \Omega$ for each function $h_n$,
  where in the setting of the proof, $\Omega = X_2$ is equipped with a measure.
  Since the family $(h_n)_{n \in \N}$ is countable,
  we know that $\bigcup_{n \in \N} N_n$ is still a null-set.
  If instead of the $h_n$ we would use the \emph{uncountable} family
  $(h_{\omega,k})_{\omega \in \Omega, k \in \N}$, it could happen that the union of the
  exceptional null-sets $N_{\omega,k}$ is no longer a null-set---in fact,
  these sets could cover all of $\Omega$.
\end{rem*}

\begin{proof}[Proof of Theorem~\ref{thm:SchurNecessity}]
  \textbf{Ad (1):}
  Define $H : Y \to [0,\infty], y \mapsto \int_X K(x,y) \, d \mu(x)$
  and furthermore ${C := \vertiii{\Phi_K}_{\lebesgue^1 \to \lebesgue^1}}$.
  For arbitrary $f \in \lebesgue^1 (\nu)$ with $f \geq 0$, we see by Tonelli's theorem that
  \[
    \int_Y H(y) \cdot f(y) \, d \nu(y)
    = \int_{X}
        \int_Y
          K(x,y) \, f(y)
        \, d \nu(y)
      \, d \mu(x)
    = \|\Phi_K \, f\|_{\lebesgue^1 (\mu)}
    \leq C \cdot \|f\|_{\lebesgue^1 (\nu)} .
  \]
  In view of the dual characterization of the $\lebesgue^\infty$-norm
  (see \cite[Theorem~6.14]{FollandRA}), the preceding estimate implies that
  $C_2 (K) = \|H\|_{\lebesgue^\infty (\nu)} \leq C < \infty$, as desired.
  Note that it is enough to consider only non-negative functions for the dual
  characterization of the $\lebesgue^\infty$-norm of $H$, since $H \geq 0$.
  Also recall that $\nu$ is $\sigma$-finite, so that the dual characterization is indeed applicable.

  \medskip{}

  \noindent
  \textbf{Ad (2):}
  Let $f : Y \to [0,\infty), y \mapsto 1$, and note $f \in \lebesgue^\infty(\nu)$.
  We then get
  \[
    \vertiii{\Phi_K}_{\lebesgue^\infty \to \lebesgue^\infty}
    \geq \vertiii{\Phi_K}_{\lebesgue^\infty \to \lebesgue^\infty} \cdot \|f\|_{\lebesgue^\infty}
    \geq \|\Phi_K \, f\|_{\lebesgue^\infty (\mu)}
    =    \esssup_{x \in X}
           \int_Y
             K(x,y)
           \, d \nu(y)
    =    C_1 (K).
  \]

  \medskip{}

  \noindent
  \textbf{Ad (3):}
  Let $C := \vertiii{\Phi_K}_{\lebesgue^{1,\infty} \to \lebesgue^{1,\infty}}$ and
  \(
    H :
    X_2 \times Y \! \to [0,\infty],
    (x_2, y) \mapsto \!
                     \int_{X_1} \!
                       K \big( (x_1,x_2) ,y \big)
                     \, d \mu_1(x_1) .
  \)
  By Tonelli's theorem, $H$ is $\CalF_2 \otimes \CalG$-measurable.
  Thus, Lemma~\ref{lem:CountableDualityCharacterization} yields a sequence $(h_n)_{n \in \N}$
  of measurable functions $h_n : Y \to [0,\infty)$ with $\| h_n \|_{\lebesgue^{1,\infty}(\nu)} \leq 1$
  and such that
  \begin{equation}
    \| H(x_2, \bullet) \|_{\lebesgue^{\infty,1} (\nu)}
    = \sup_{n \in \N}
        \int_Y
          H(x_2, y) \cdot h_n (y)
        \, d \nu (y)
    \qquad \forall \, x_2 \in X_2 .
    \label{eq:SchurNecessityCountableCharacterization}
  \end{equation}
  For each $n \in \N$, there is a $\mu_2$-null-set $N_n \subset X_2$ with
  \({
    \| (\Phi_K \, h_n) (\bullet,x_2) \|_{\lebesgue^1(\mu_1)}
    \leq \! \| \Phi_K \, h_n \|_{\lebesgue^{1,\infty} (\mu)}
    \leq C
  }\)
  for all $x_2 \in X_2 \setminus N_n$.

  Define $N := \bigcup_{n \in \N} N_n$.
  For $x_2 \in X_2 \setminus N$, we then see by definition of $H$ and by Tonelli's theorem that
  \[
    \int_Y \!
      H(x_2,y) \cdot h_n (y)
    \, d \nu (y)
    \!=\! \int_{X_1} \!
            \int_Y \!
              K \big( (x_1,x_2) , y \big)
              h_n (y)
            \, d \nu(y)
          \, d \mu_1 (x_1)
    = \| (\Phi_K \, h_n) (\bullet, x_2) \|_{\lebesgue^1 (\mu_1)}
    \!\leq\! C .
  \]
  In view of Equation~\eqref{eq:SchurNecessityCountableCharacterization}, this implies
  $\| H(x_2, \bullet) \|_{\lebesgue^{\infty,1}(\nu)} \leq C$ for all $x_2 \in X_2 \setminus N$.
  Directly from the definitions of $H$ and of $C_3(K)$, we see that this implies
  $C_3(K) \leq C = \vertiii{\Phi_K}_{\lebesgue^{1,\infty} \to \lebesgue^{1,\infty}} < \infty$,
  as claimed.

  \medskip{}

  \noindent
  \textbf{Ad (4):}
  Let $K^T$ denote the transposed kernel of $K$.
  By Lemma~\ref{lem:AdjointBoundedness},
  $\Phi_{K^T} : \lebesgue^{1,\infty} (\mu) \to \lebesgue^{1,\infty}(\nu)$
  is bounded, with
  \(
    \vertiii{\Phi_{K^T}}_{\lebesgue^{1,\infty} \to \lebesgue^{1,\infty}}
    \leq \vertiii{\Phi_{K}}_{\lebesgue^{\infty,1} \to \lebesgue^{\infty,1}}
    =:   C
    < \infty.
  \)
  By Part~(3) (applied to $K^T$, with interchanged roles of $\mu$ and $\nu$),
  this implies $C_3(K^T) \leq C$.
  Finally, Lemma~\ref{lem:SchurConstantsForAdjointKernel} shows that
  $C_4(K) = C_3(K^T) \leq C < \infty$, as claimed.
\end{proof}

\section{Proofs for the properties of the kernel modules \texorpdfstring{$\BBm(X,Y)$}{𝓑ₘ(X,Y)}}%
\label{sec:KernelModulePropertiesProof}

In this section, we prove the properties of the kernel module $\BBm (X,Y)$ that are stated in
Propositions~\ref{prop:NewKernelModuleBasicProperties1}--\ref{prop:NewKernelModuleBasicProperties3}.

\begin{proof}[Proof of Proposition~\ref{prop:NewKernelModuleBasicProperties1}]
\textbf{Ad (\ref{enu:NewKernelModuleSolid}):}
This is an immediate consequence of the definitions, once one notes that if $|L| \leq |K|$
holds $\mu \otimes \nu$-almost everywhere, then Tonelli's theorem shows
that for $\mu_2 \otimes \nu_2$-almost every $(x_2, y_2) \in X_2 \times Y_2$,
we have $|L^{(x_2,y_2)}| \leq |K^{(x_2,y_2)}|$ $\mu_1 \otimes \nu_1$-almost everywhere.

\medskip{}

\noindent
\textbf{Ad (\ref{enu:NewKernelModuleFatouProperty}):}
We first prove that $\|\bullet\|_{\AAi(X,Y)}$ satisfies the Fatou property,
for \emph{arbitrary} $\sigma$-finite measure spaces $\XTuple, \YTuple$
(not necessarily of product structure).
Indeed, if a sequence of kernels ${K_n : X \times Y \to [0,\infty]}$
satisfies $K_n \nearrow K$ pointwise,
then the monotone convergence theorem shows that
\(
  \| K_n (x,\bullet) \|_{\lebesgue^1(\nu)} \nearrow \| K (x,\bullet) \|_{\lebesgue^1(\nu)}
\)
and
\(
  \| K_n (\bullet,y) \|_{\lebesgue^1(\mu)} \nearrow \| K (\bullet,y) \|_{\lebesgue^1(\mu)} .
\)
Next, it is easy to see that if $0 \leq G_n \nearrow G$,
then $\| G_n \|_{\lebesgue^\infty} \nearrow \| G \|_{\lebesgue^\infty}$.
If we combine this with the preceding observations, we see
\(
  \esssup_{x \in X}
    \| K_n (x,\bullet) \|_{\lebesgue^1(\nu)}
  \nearrow
  \esssup_{x \in X}
    \| K (x,\bullet) \|_{\lebesgue^1(\nu)}
\)
and
\(
  \esssup_{y \in Y}
    \| K_n (\bullet,y) \|_{\lebesgue^1(\mu)}
  \nearrow
  \esssup_{y \in Y}
    \| K (\bullet,y) \|_{\lebesgue^1(\mu)} .
\)
Recalling the definition of $\| \bullet \|_{\AAi}$,
this implies $\| K_n \|_{\AAi} \nearrow \| K \|_{\AAi}$.

Now we prove the actual claim, first for the unweighted case.
Clearly, $K_n^{(x_2,y_2)} \nearrow K^{(x_2,y_2)}$.
Therefore, the preceding considerations show that
\[
  \Gamma_n (x_2,y_2)
  := \| K_n^{(x_2,y_2)} \|_{\AAi (X_1,Y_1)}
  \nearrow \| K^{(x_2,y_2)} \|_{\AAi (X_1,Y_1)}
  =: \Gamma(x_2,y_2).
\]
Applying the Fatou property for $\|\bullet\|_{\AAi}$, we see
\(
  \| K_n \|_{\BBi}
  = \| \Gamma_n \|_{\AAi (X_2,Y_2)}
  \nearrow \| \Gamma \|_{\AAi (X_2,Y_2)}
  = \| K \|_{\BBi} .
\)
Finally, for the weighted case, note that
$\| K_n \|_{\BBm} = \| m \cdot K_n \|_{\BBi} \nearrow \| m \cdot K \|_{\BBi} = \| K \|_{\BBm}$,
since $m \cdot K_n \nearrow m \cdot K$.

\medskip{}

\noindent
\textbf{Ad (\ref{enu:NewKernelModuleComplete}):}
It is not hard to see that $\| \bullet \|_{\BBm}$ is a function-norm in the sense of Zaanen
(see \cite[Section 63]{ZaanenIntegration}); that is, for $K,L : X \times Y \to [0,\infty]$
measurable and $\alpha \in [0,\infty)$, the following hold:
\begin{align*}
  & \| K \|_{\BBm} = 0
    \quad \Longleftrightarrow \quad
    K = 0 \text{ almost everywhere}, \\
  & \| K + L \|_{\BBm} \leq \| K \|_{\BBm} + \| L \|_{\BBm}
    \qquad \text{and} \qquad
    \| \alpha K \|_{\BBm} = \alpha \, \| K \|_{\BBm}, \\
  \text{as well as} \quad
  & \| K \|_{\BBm} \leq \| L \|_{\BBm} \text{ if } K \leq L \text{ almost everywhere}.
\end{align*}
Since we just showed that $\| \bullet \|_{\BBm}$ also satisfies the Fatou property,
and since $\| K \|_{\BBm} = \| \,|K|\, \|_{\BBm}$ for $K : X \times Y \to \CC$ measurable,
it follows from the theory of normed Köthe spaces that $\big( \BBm(X,Y), \| \bullet \|_{\BBm} \big)$
is indeed a Banach space; see \cite[Section~65, Theorem~1]{ZaanenIntegration}.

\medskip{}

\noindent
\textbf{Ad (\ref{enu:NewKernelModuleEmbedsInOld}):}
We start with the unweighted case.
On the one hand, we have for each $x_2 \in X_2$ that
\begin{align*}
  \esssup_{x_1 \in X_1}
    \int_Y
      |K \big( (x_1,x_2),y \big)|
    \, d \nu(y)
  & = \esssup_{x_1 \in X_1}
        \int_{Y_2}
          \| K^{(x_2,y_2)} (x_1, \bullet) \|_{\lebesgue^1(\nu_1)}
        \, d \nu_2 (y_2) \\
  ({\scriptstyle{\text{Eq. } \eqref{eq:EsssupOfIntegral}}})
  & \leq \int_{Y_2}
           \esssup_{x_1 \in X_1}
             \| K^{(x_2,y_2)} (x_1, \bullet) \|_{\lebesgue^1(\nu_1)}
         \, d \nu_2(y_2)  \\
  & \leq \int_{Y_2}
           \| K^{(x_2,y_2)} \|_{\AAi}
         \, d \nu_2 (y_2) .
\end{align*}
By definition of $\| \bullet \|_{\BBi}$, we have
$\int_{Y_2} \| K^{(x_2,y_2)} \|_{\AAi} \, d \nu_2 (y_2) \leq \| K \|_{\BBi}$
for almost all $x_2 \in X_2$.
Overall, we thus see
\[
  C_1(K)
  = \esssup_{x \in X}
      \int_Y
        |K(x,y)|
      \, d \nu(y)
  \leq \| K \|_{\BBi}.
\]
Now, combine the identity $C_2(K) = C_1(K^T)$ from Lemma~\ref{lem:SchurConstantsForAdjointKernel}
with Part~(\ref{enu:NewKernelModuleAdjoint}) of Proposition~\ref{prop:NewKernelModuleBasicProperties2}
(which will be proven independently) to get
\({
  C_2(K)
  = C_1(K^T)
  \leq \| K^T \|_{\BBi(Y,X)}
  =    \| K \|_{\BBi(X,Y)} .
}\)
But directly from the definition of $\| \bullet \|_{\AAi}$, we see
$\| K \|_{\AAi(X,Y)} = \max \{ C_1(K), C_2(K) \} \leq \| K \|_{\BBi(X,Y)}$.

The weighted case is now a direct consequence of the definitions.
\end{proof}

\begin{proof}[Proof of Proposition~\ref{prop:NewKernelModuleBasicProperties2}]
\textbf{Ad (\ref{enu:NewKernelModuleAdjoint}):}
We start with the unweighted case.
By definition, we see $\| L^T \|_{\AAi(W,V)} = \| L \|_{\AAi(V,W)}$
for arbitrary ($\sigma$-finite) measure spaces $(V,\mathcal{V},\gamma)$
and $(W,\mathcal{W},\theta)$, and any measurable function $L : V \times W \to \CC$.
Also, it is easy to see that ${(K^T)^{(y_2,x_2)} = (K^{(x_2,y_2)})^T}$.
By combining these observations, we see
\[
  \Psi(y_2,x_2)
  := \big\| (K^T)^{(y_2,x_2)} \big\|_{\AAi(Y_1,X_1)}
  =  \big\| (K^{(x_2,y_2)})^T \big\|_{\AAi(Y_1,X_1)}
  =  \big\| K^{(x_2,y_2)} \big\|_{\AAi(X_1,Y_1)}
  =: \Gamma(x_2,y_2).
\]
In other words, $\Psi = \Gamma^T$.
This finally implies $K^T \in \BBi (Y,X)$, since
\[
  \| K^T \|_{\BBi(Y,X)}
  = \| \Psi \|_{\AAi(Y_2,X_2)}
  = \| \Gamma^T \|_{\AAi(Y_2,X_2)}
  = \| \Gamma \|_{\AAi(X_2,Y_2)}
  = \| K \|_{\BBi(X,Y)}
  < \infty.
\]

For the weighted case, note that $K^T \in \BBi_{m^T}(Y,X)$
if $m^T K^T = (m \, K)^T \in \BBi(Y,X)$,
which holds by the unweighted case, since $m \, K \in \BBi(X,Y)$.
Finally, we also see
\[
  \| K^T \|_{\BBi_{m^T}(Y,X)}
  = \| m^T \, K^T \|_{\BBi(Y,X)}
  = \| (m \, K)^T \|_{\BBi(Y,X)}
  = \| m \, K \|_{\BBi(X,Y)}
  = \| K \|_{\BBm(X,Y)}.
\]

\medskip{}

\noindent
\textbf{Ad (\ref{enu:NewKernelModuleMultiplicationProperty}):}
We start with the unweighted case and with non-negative kernels
${K : X \times Y \to [0,\infty]}$ and $L : Y \times Z \to [0,\infty]$.
For brevity, let us define $K_0 (x_2,y_2) := \| K^{(x_2,y_2)} \|_{\AAi (X_1,Y_1)}$
and $L_0(y_2,z_2) := \| L^{(y_2,z_2)} \|_{\AAi (Y_1 , Z_1)}$.
By definition of the product $K \KernelProduct L$, by Tonelli's theorem, and since
$\| L^{(y_2,z_2)} (y_1, \bullet) \|_{\lebesgue^1(\varrho_1)} \leq L_0(y_2,z_2)$
for almost every $y_1 \in Y_1$, we see for all $(x_2, z_2) \in X_2 \times Z_2$ that
\begin{align*}
  & \esssup_{x_1 \in X_1}
      \big\| (K \KernelProduct L)^{(x_2,z_2)} (x_1,\bullet) \big\|_{\lebesgue^1(\varrho_1)} \\
  ({\scriptstyle{\text{Definitions, Tonelli}}})
  & = \esssup_{x_1 \in X_1}
      \int_{Y_2}
        \int_{Y_1}
          K^{(x_2,y_2)} (x_1,y_1)
          \cdot \| L^{(y_2,z_2)} (y_1, \bullet) \|_{\lebesgue^1(\varrho_1)}
        \, d \nu_1 (y_1)
      \, d \nu_2(y_2) \\
  ({\scriptstyle{\text{Eq. } \eqref{eq:EsssupOfIntegral}, \text{ Def.~of } L_0}})
  & \leq \int_{Y_2}
           L_0(y_2,z_2)
           \cdot \esssup_{x_1 \in X_1}
                   \| K^{(x_2,y_2)}(x_1, \bullet) \|_{\lebesgue^1(\nu_1)}
         \, d \nu_2(y_2) \\
  ({\scriptstyle{\text{Def. of } K_0}})
  & \leq \int_{Y_2}
           L_0(y_2, z_2)
           \cdot K_0(x_2,y_2)
         \, d \nu_2 (y_2) .
\end{align*}
Using an almost identical calculation, we see for every $(x_2, z_2) \in X_2 \times Z_2$ that
\[
  \esssup_{z_1 \in Z_1}
    \big\| (K \KernelProduct L)^{(x_2,z_2)} (\bullet,z_1) \big\|_{\lebesgue^1 (\mu_1)}
  \leq \int_{Y_2}
         K_0(x_2, y_2) \cdot L_0(y_2, z_2)
       \, d \nu_2(y_2) .
\]
By definition of $\| \cdot \|_{\AAi(X_1,Z_1)}$, we have thus shown
\[
  \Gamma(x_2,z_2)
  := \| (K \KernelProduct L)^{(x_2,z_2)} \|_{\AAi(X_1,Z_1)}
  \leq \int_{Y_2}
         K_0(x_2, y_2) \cdot L_0(y_2, z_2)
       \, d \nu_2(y_2) .
\]
By another application of Tonelli's theorem, and since
$\| L_0(y_2, \bullet) \|_{\lebesgue^1(\varrho_2)} \leq \| L_0 \|_{\AAi} = \| L \|_{\BBi}$
for almost all $y_2 \in Y_2$, this implies
\begin{align*}
  \esssup_{x_2 \in X_2}
    \| \Gamma (x_2, \bullet) \|_{\lebesgue^1(\varrho_2)}
  & \leq \esssup_{x_2 \in X_2}
           \int_{Y_2}
             K_0 (x_2,y_2) \cdot \| L_0(y_2, \bullet) \|_{\lebesgue^1(\varrho_2)}
           \, d \nu_2 (y_2) \\
  & \leq \| L \|_{\BBi} \cdot \esssup_{x_2 \in X_2} \| K_0 (x_2, \bullet) \|_{\lebesgue^1(\nu_2)}
    \leq \| L \|_{\BBi} \cdot \| K_0 \|_{\AAi}
    =    \| L \|_{\BBi} \cdot \| K \|_{\BBi} .
\end{align*}
Using similar arguments, we also see
\(
  \esssup_{z_2 \in Z_2}
    \| \Gamma (\bullet, z_2) \|_{\lebesgue^1(\mu_2)}
  \leq \| K \|_{\BBi} \cdot \| L \|_{\BBi} .
\)
By definition of $\Gamma$ and of $\| \Gamma \|_{\AAi}$, we have thus shown
\(
  \| K \KernelProduct L \|_{\BBi}
  = \| \Gamma \|_{\AAi}
  \leq \| K \|_{\BBi} \cdot \| L \|_{\BBi}
  < \infty ,
\)
which completes the proof for the unweighted case and non-negative kernels.
Note in particular that we get $(K \KernelProduct L) (x,z) < \infty$
for almost all $(x,z) \in X \times Z$, since $\| K \KernelProduct L \|_{\BBi}$ is finite.

Finally, we consider complex-valued kernels including weights.
Since ${K_0 := \omega \cdot |K| \in \BBi (X,Y)}$ and $L_0 := \sigma \cdot |L| \in \BBi (Y,Z)$
are non-negative, the above considerations imply $K_0 \KernelProduct L_0 \in \BBi (X,Z)$ and
\(
  \| K_0 \KernelProduct L_0 \|_{\BBi}
  \leq \| K_0 \|_{\BBi} \cdot \| L_0 \|_{\BBi}
  =    \| K \|_{\BBi_\omega} \cdot \| L \|_{\BBi_\sigma}.
\)
Since $\tau(x,z) \leq C \cdot \omega(x,y) \cdot \sigma(y,z)$, we have
$\tau(x,z) \cdot |K(x,y)| \cdot |L(y,z)| \leq C \cdot K_0(x,y) \cdot L_0(y,z)$, and thus
\begin{align*}
  \tau(x,z) \cdot |K \KernelProduct L (x,z)|
  & \leq \tau(x,z)
         \cdot \int_Y
                 |K(x,y)| \cdot |L(y,z)|
               \, d \nu(y) \\
  & \leq C \cdot \int_Y
                   K_0(x,y) \cdot L_0(y,z)
                 \, d \nu(y)
  =      C \cdot (K_0 \KernelProduct L_0) (x,z)
  <      \infty
\end{align*}
for almost all $(x,z) \in X \times Z$.
Finally, we see by solidity of $\BBi(X,Z)$ that $K \KernelProduct L \in \BBi_\tau (X,Z)$
with
\(
  \| K \KernelProduct L \|_{\BBi_\tau}
  = \| \tau \cdot (K \KernelProduct L) \|_{\BBi}
  \leq C \cdot \| K_0 \KernelProduct L_0 \|_{\BBi}
  \leq C \cdot \| K \|_{\BBi_\omega} \cdot \| L \|_{\BBi_\sigma}.
\)
\end{proof}

\begin{proof}[Proof of Proposition~\ref{prop:NewKernelModuleBasicProperties3}]
We again start with the unweighted case $v,w,m \equiv 1$ and $C = 1$.
By Theorem~\ref{thm:SchurTestSufficientUnweighted}, it suffices to prove
$C_i (K) \leq \| K \|_{\BBi(X,Y)}$ for $i \in \{1,2,3,4\}$,
with $C_i(K)$ as defined in Equation~\eqref{eq:MixedNormSchurConstants}.
The cases $i \in \{1,2\}$ were already handled in the proof
of Part~(\ref{enu:NewKernelModuleEmbedsInOld}) of Proposition~\ref{prop:NewKernelModuleBasicProperties1}.
Furthermore, once we show $C_3(K) \leq \| K \|_{\BBi(X,Y)}$, a combination of
Lemma~\ref{lem:SchurConstantsForAdjointKernel} and Part~(\ref{enu:NewKernelModuleAdjoint})
of Proposition~\ref{prop:NewKernelModuleBasicProperties2} will
show that $C_4(K) = C_3(K^T) \leq \| K^T \|_{\BBi(Y,X)} = \| K \|_{\BBi(X,Y)}$,
so that it suffices to consider the case $i = 3$.

Let $\Gamma(x_2,y_2) := \| K^{(x_2,y_2)} \|_{\AAi(X_1,Y_1)}$, and note
$\Gamma(x_2,y_2) \geq \esssup_{y_1 \in Y_1} \| K^{(x_2,y_2)}(\bullet,y_1) \|_{\lebesgue^1(\mu_1)}$,
as well as $\| K \|_{\BBi} = \| \Gamma \|_{\AAi}$.
By definition of $C_3(K)$, this implies as desired that
\begin{align*}
  C_3(K)
  & = \esssup_{x_2 \in X_2}
        \int_{Y_2}
          \esssup_{y_1 \in Y_1}
            \| K^{(x_2,y_2)}(\bullet, y_1) \|_{\lebesgue^{1}(\mu_1)}
        \, d \nu_2 (y_2) \\
  & \leq \esssup_{x_2 \in X_2}
           \| \Gamma(x_2, \bullet) \|_{\lebesgue^1(\nu_2)}
    \leq \| \Gamma \|_{\AAi}
    =    \| K \|_{\BBi} .
\end{align*}

Finally, we handle the weighted case.
As noted in the discussion around Equation~\eqref{eq:WeightedKernelDefinition},
if we define $K_{v,w} : X \times Y \to \CC, (x,y) \mapsto \frac{v(x)}{w(y)} \cdot K(x,y)$, then
$\Phi_K : \lebesgue^{p,q}_w (\nu) \to \lebesgue^{p,q}_v (\mu)$ is well-defined and bounded
if and only if $\Phi_{K_{v,w}} : \lebesgue^{p,q}(\nu) \to \lebesgue^{p,q}(\mu)$ is,
and in this case we have
\(
  \| \Phi_K \|_{\lebesgue^{p,q}_w (\nu) \to \lebesgue^{p,q}_v (\mu)}
  = \| \Phi_{K_{v,w}} \|_{\lebesgue^{p,q}(\nu) \to \lebesgue^{p,q}(\mu)} .
\)
Next, by our assumptions on $v,w,m$, and by the unweighted case, we see
that $\Phi_{K_{v,w}} : \lebesgue^{p,q}(\nu) \to \lebesgue^{p,q}(\mu)$ is indeed well-defined
and bounded, with
\[
  \| \Phi_{K_{v,w}} \|_{\lebesgue^{p,q}(\nu) \to \lebesgue^{p,q}(\mu)}
  \leq \| K_{v,w} \|_{\BBi}
  \vphantom{\overset{a}{\leq}}
  \,\smash{\overset{(\dagger)}{\leq}}\, C \cdot \| m \cdot K \|_{\BBi}
  =    C \cdot \| K \|_{\BBm} .
\]
Here, the step marked with $(\dagger)$ used that
$|K_{v,w}(x,y)| = \big| \frac{v(x)}{w(y)} \cdot K(x,y) \big| \leq C \cdot |(m \cdot K)(x,y)|$.
\end{proof}

\section{\texorpdfstring{Spaces compatible with the kernel modules $\BBm$}
                        {Spaces compatible with the new kernel modules}}
\label{sec:StructureOfBBmCompatibleSpaces}

In this section, we prove the necessary conditions for spaces compatible with $\BBm$
that we stated in Section~\ref{sub:RestrictionsOnSolidSpaces}.
The proof itself is presented in Section~\ref{sub:StructureOfBBmCompatibleSpaces},
preceded by a \emph{richness result} for the space $\BBm (X,Y)$
which is given in Section~\ref{sub:KernelModuleRichnessResults} below.
This result, which provides sufficient conditions on $f : X \to \CC$ and $g : Y \to \CC$
which guarantee that the \emph{tensor product}
${f \otimes g : X \times Y \to \CC, (x,y) \mapsto f(x) \cdot g(y)}$ belongs to $\BBm (X,Y)$,
will be useful for proving Theorem~\ref{thm:NecessaryConditionsForCompatibleSpaces} in
Section~\ref{sub:StructureOfBBmCompatibleSpaces}.

\subsection{A richness result for the kernel modules \texorpdfstring{$\BBm (X,Y)$}{𝓑ₘ(X,Y)}}
\label{sub:KernelModuleRichnessResults}

Before we state our richness result, we fix the following notation for the spaces
introduced in Definition~\ref{def:SumAndIntersectionSpaces}:
\begin{equation}\label{eq:IntersectionSpaceSymbol}
  \IntersectionSpaceSymbol
  := \IntersectionSpaceSymbol (\mu)
  := \lebesgue^1 (\mu)
     \cap \lebesgue^\infty(\mu)
     \cap \lebesgue^{1,\infty}(\mu)
     \cap \lebesgue^{\infty,1}(\mu),
  \quad \text{with}\quad
  \| \bullet \|_{\IntersectionSpaceSymbol}
  := \| \bullet \|_{\lebesgue^1
                    \cap \lebesgue^\infty
                    \cap \lebesgue^{1,\infty}
                    \cap \lebesgue^{\infty,1}}.
\end{equation}
Likewise, $\SumSpaceSymbol := \SumSpaceSymbol(\mu)$, where
\begin{equation}\label{eq:SumSpaceSymbol}
  \SumSpaceSymbol (\mu)
  := \lebesgue^1 (\mu)
     + \lebesgue^\infty (\mu)
     + \lebesgue^{1,\infty} (\mu)
     + \lebesgue^{\infty,1} (\mu),
  \quad \text{with}\quad
  \| \bullet \|_{\SumSpaceSymbol}
  := \| \bullet \|_{\lebesgue^1
     + \lebesgue^\infty
     + \lebesgue^{1,\infty}
     + \lebesgue^{\infty,1}}.
\end{equation}
Given a measurable function $w : X \to (0,\infty)$,
the weighted spaces $\IntersectionSpaceSymbol_w$ and $\SumSpaceSymbol_w$
are defined as in Equation~\eqref{eq:WeightedSpaceDefinition}.
This notation will allow for a succinct formulation of the following results.

\begin{rem*}
  We will see in Theorem~\ref{thm:DualOfIntersection} that
  $\| \bullet \|_{\SumSpaceSymbol}$ is indeed a norm, and that with this norm
  $\SumSpaceSymbol$ becomes a Banach space.
\end{rem*}

Having introduced the proper notation, we can now state and prove the announced richness result
for the kernel module $\BBm (X,Y)$.

\begin{lemma}\label{lem:KernelModuleRichnessResult}
  With notation and assumptions as in Definition~\ref{def:NewKernelModule},
  let $m : X \times Y \to (0,\infty)$, $v : X \to (0,\infty)$, and $w : Y \to (0,\infty)$
  be measurable and such that
  \(
    m(x,y)
    \leq C \cdot {v}(x) \cdot {w}(y)
  \)
  for all $x \in X$ and $y \in Y$ and some $C > 0$.

  If $f \in \GoodFunctions_{v}(\mu)$ and $g \in \GoodFunctions_{w}(\nu)$, then
  $f \otimes g \in \BBm (X \times Y)$ with
  \(
    \| f \otimes g \|_{\BBm} \leq 2C \, \| f \|_{\GoodFunctions_{v}} \cdot \| g \|_{\GoodFunctions_{w}}.
  \)
  Here,
  \(
    f \otimes g : X \times Y \to \CC, (x,y) \mapsto f(x) \cdot g(y) .
  \)
\end{lemma}

\begin{proof}
  We first consider the case where $m,{v},{w} \equiv 1$ and $C = 1$.
  Directly from the definitions, we see for arbitrary $(x_2,y_2) \in X_2 \times Y_2$ that
  \begin{align*}
    & \esssup_{x_1 \in X_1}
        \big\| (f \otimes g)^{(x_2,y_2)} (x_1, \bullet) \big\|_{\lebesgue^1 (\nu_1)}
      = \| f(\bullet,x_2) \|_{\lebesgue^\infty} \cdot \| g(\bullet,y_2) \|_{\lebesgue^1} \\
    \quad \text{and} \quad
    & \esssup_{y_1 \in Y_1}
        \big\| (f \otimes g)^{(x_2,y_2)} (\bullet, y_1) \big\|_{\lebesgue^1 (\mu_1)}
      = \| f(\bullet, x_2) \|_{\lebesgue^1} \cdot \| g(\bullet, y_2) \|_{\lebesgue^\infty} .
  \end{align*}
  Therefore,
  \[
    \Gamma(x_2,y_2)
    := \| (f \otimes g)^{(x_2,y_2)} \|_{\AAi (X_1 , Y_1)}
    \leq \| f(\bullet,x_2) \|_{\lebesgue^\infty} \cdot \| g(\bullet,y_2) \|_{\lebesgue^1}
         + \| f(\bullet, x_2) \|_{\lebesgue^1} \cdot \| g(\bullet, y_2) \|_{\lebesgue^\infty} ,
  \]
  which easily implies
  \begin{align*}
    & \esssup_{x_2 \in X_2}
        \big\| \Gamma(x_2,\bullet) \big\|_{\lebesgue^1(\nu_2)}
      \leq \| f \|_{\lebesgue^\infty} \, \| g \|_{\lebesgue^1}
           + \| f \|_{\lebesgue^{1,\infty}} \, \| g \|_{\lebesgue^{\infty,1}} \\
    \quad \text{and} \quad
    & \esssup_{y_2 \in Y_2}
        \big\| \Gamma(\bullet,y_2) \big\|_{\lebesgue^1(\mu_2)}
      \leq \| f \|_{\lebesgue^{\infty,1}} \, \| g \|_{\lebesgue^{1,\infty}}
           + \| f \|_{\lebesgue^{1}} \, \| g \|_{\lebesgue^\infty} .
  \end{align*}
  Overall, this implies
  \(
    \| f \otimes g \|_{\BBi}
    = \| \Gamma \|_{\AAi}
    \leq 2 \, \| f \|_{\GoodFunctions(\mu)} \, \| g \|_{\GoodFunctions (\nu)} .
  \)

  It remains to consider the case including the weights.
  But by assumption on the weights, we have
  $|m \cdot (f \otimes g)| \leq C \cdot ({v} \cdot |f|) \otimes ({w} \cdot |g|)$,
  where ${v} \cdot |f| \in \GoodFunctions(\mu)$ and ${w} \cdot |g| \in \GoodFunctions(\nu)$.
  By the unweighted case, this implies $m \cdot (f \otimes g) \in \BBi$.
  In other words, $f \otimes g \in \BBm$, and
  \[
    \| f \otimes g \|_{\BBm}
    = \big\| m \cdot (f \otimes g) \big\|_{\BBi}
    \leq C \cdot \big\| ({v} \cdot |f|) \otimes ({w} \cdot |g|) \big\|_{\BBi}
    \leq 2C \cdot \big\| {v} \cdot |f| \big\|_{\GoodFunctions}
            \cdot \big\| {w} \cdot |g| \big\|_{\GoodFunctions}
    =    2C \cdot \| f \|_{\GoodFunctions_{v}} \cdot \| g \|_{\GoodFunctions_{w}} ,
  \]
  as claimed.
\end{proof}

In view of the preceding lemma, it is natural to ask for which spaces $\BanachOne$ we can guarantee
that $\BanachOne \cap \GoodFunctions_{v}$ is nontrivial.
The next lemma shows that this is always the case if $\BanachOne$
is a non-trivial \emph{solid} function space on $\XTuple$;
the definition of these spaces was given in Section~\ref{sub:RestrictionsOnSolidSpaces}.

\begin{lemma}\label{lem:SolidSpacesContainGoodFunctions}
  Let ${\XTuple := \ProductTupleX}$,
  where $\XIndexTuple{1}$ and $\XIndexTuple{2}$ are $\sigma$-finite measure spaces.
  Let $\BanachOne$ be a solid function space on $X$ with $\BanachOne \neq \{0\}$,
  and let ${{v} : X \to (0,\infty)}$ be measurable.
  Then there exists a measurable set $E \subset X$ such that
  $\mu(E) > 0$ and ${\Indicator_E \in \BanachOne \cap \GoodFunctions_{v} \cap \GoodFunctions}$.
\end{lemma}

\begin{proof}
  $\BanachOne$ is non-trivial and solid; hence, there exists a non-negative function
  $g \in \BanachOne$ such that ${\{ x \in X \colon g(x) > 0 \}}$ has positive measure.
  Thus, ${E_{1} := \{ x \in X \colon g(x) \geq n_0^{-1} \}}$ has positive measure for a suitable
  $n_0 \in \N$.
  Next, since $X = \bigcup_{k \in \N} \{ x \in X \colon {v}(x) \leq k \}$,
  we see by $\sigma$-additivity that there is $k \in \N$ such that
  $E_2 := E_1 \cap \{ x \in X \colon {v}(x) \leq k \}$ has positive measure.

  Furthermore, since $X_1, X_2$ are $\sigma$-finite, we have $X_i = \bigcup_{n \in \N} X_i^{(n)}$
  for certain $X_i^{(n)} \in \CalF_i$ satisfying $X_i^{(n)} \subset X_i^{(n+1)}$ and
  $\mu_i \big( X_i^{(n)} \big) < \infty$ for all $n \in \N$ and $i \in \{ 1,2 \}$.
  Because of $X = \bigcup_{n \in \N} \big( X_1^{(n)} \times X_2^{(n)} \big)$,
  there is thus some $n \in \N$ such that $E := E_2 \cap \big( X_1^{(n)} \times X_2^{(n)} \big)$
  has positive measure.

  Define $f := \Indicator_E$, and note that $g \geq n_0^{-1}$ on $E_1 \supset E_2 \supset E$,
  so that $0 \leq f \leq n_0 \cdot g$.
  By solidity of $\BanachOne$ and since $g \in \BanachOne$, this implies $f \in \BanachOne$.
  Furthermore, $0 \leq (1 + {v}) \cdot f \leq (1+k) \cdot \Indicator_{X_1^{(n)} \times X_2^{(n)}}$,
  so that
  \begin{align*}
      \| (1+{v}) \cdot f\|_{\lebesgue^1(\mu)}
    & \leq (1+k) \cdot \mu_1 \big( X_1^{(n)} \big) \cdot \mu_2 \big( X_2^{(n)} \big)
      < \infty , \\
      \| (1+{v}) \cdot f\|_{\lebesgue^\infty(\mu)}
    & \leq 1+k < \infty, \\
      \| (1+{v}) \cdot f\|_{\lebesgue^{1,\infty}(\mu)}
    & \leq (1+k) \cdot \mu_1 \big( X_1^{(n)} \big)
      < \infty, \\
      \| (1+{v}) \cdot f\|_{\lebesgue^{\infty,1}(\mu)}
    & \leq (1+k) \cdot \mu_2 \big( X_2^{(n)} \big) < \infty ,
  \end{align*}
  and thus $f \in \GoodFunctions_{v} \cap \GoodFunctions$.
  Finally, $\mu(E) > 0$ holds by our choice of $n_0, k$, and $n$.
\end{proof}

\subsection{\texorpdfstring{Necessary conditions for spaces compatible with the kernel modules $\BBm$}
                           {Necessary conditions for spaces compatible with the new kernel modules}}
\label{sub:StructureOfBBmCompatibleSpaces}

In this section, we prove Theorem~\ref{thm:NecessaryConditionsForCompatibleSpaces},
considering both parts of the theorem individually.

\medskip{}

\noindent
\textbf{Proof of Part~(\ref{enu:KernelCoDomainEmbedding})
of Theorem~\ref{thm:NecessaryConditionsForCompatibleSpaces}:}
Since $\BanachTwo$ is a solid function space, convergence in $\BanachTwo$
implies local convergence in measure (that is, convergence in measure
on sets of finite measure); see
\mbox{\hspace{1sp}\cite[Lemma~2.2.8]{VoigtlaenderPhDThesis}}.
The same holds for the space $\GoodFunctions_v(\mu)$
introduced in Definition~\ref{def:SumAndIntersectionSpaces}
and Equation~\eqref{eq:IntersectionSpaceSymbol}.
Since $X$ is $\sigma$-finite, local convergence in measure determines the limit uniquely,
up to changes on a set of measure zero.
Thus, the closed graph theorem shows that the continuous embedding
$\GoodFunctions_v(\mu) \hookrightarrow \BanachTwo$ holds if
and only if we have $\GoodFunctions_v(\mu) \subset \BanachTwo$ as sets, which we now prove.

Since $\BanachOne \neq \{0\}$ is a solid function space on $Y$,
Lemma~\ref{lem:SolidSpacesContainGoodFunctions} produces a measurable set $E \subset Y$
satisfying $\nu(E) \in (0,\infty)$ and such that $\Indicator_E \in \BanachOne \cap \GoodFunctions_w(\nu)$.
Define $g := \Indicator_E / \nu(E)$.
Now, let $f \in \GoodFunctions_v(\mu)$ be arbitrary and note $|f| \in \GoodFunctions_v(\mu)$ as well.
According to Lemma~\ref{lem:KernelModuleRichnessResult}, we have $K := |f| \otimes g \in \BBm(X,Y)$.
By assumption, this implies that $\Phi_K : \BanachOne \to \BanachTwo$ is well-defined,
and thus $h := \Phi_K [\Indicator_E] \in \BanachTwo$.
But we have
\[
  (\Phi_K \Indicator_E)(x)
  = |f(x)| \cdot \int_Y g(y) \cdot \Indicator_E (y) \, d \nu(y)
  = |f(x)|
  \qquad \forall \, x \in X,
\]
and thus $|f| = \Phi_K [\Indicator_E] \in \BanachTwo$.
Since $\BanachTwo$ is solid, we see $f \in \BanachTwo$,
and thus $\GoodFunctions_v(\mu)\subset \BanachTwo$, since $f \in \GoodFunctions_v(\mu)$ was arbitrary.
As seen above, this proves $\GoodFunctions_v(\mu) \hookrightarrow \BanachTwo$.
\hfill$\square$

\medskip{}

\noindent
\textbf{Proof of Part~(\ref{enu:KernelDomainEmbedding})
of Theorem~\ref{thm:NecessaryConditionsForCompatibleSpaces}:}
Our proof of Part~(\ref{enu:KernelDomainEmbedding})
of Theorem~\ref{thm:NecessaryConditionsForCompatibleSpaces}
is crucially based on the following description of the space
$\SumSpaceSymbol(\mu)$ as the associate space of $\IntersectionSpaceSymbol(\mu)$.
The proof of this result is surprisingly involved, and thus postponed to
Appendix~\ref{sec:DualCharacterizationOfSumSpace}.

\begin{theorem}\label{thm:DualOfIntersection}
Let $(X,\CalF,\mu) = (X_1\otimes X_2,\CalF_2 \otimes \CalF_2, \mu_1 \otimes \mu_2)$,
where $\mu_1,\mu_2$ are $\sigma$-finite.
Let furthermore $\IntersectionSpaceSymbol$ and $\SumSpaceSymbol$ be
as in \eqref{eq:IntersectionSpaceSymbol} and \eqref{eq:SumSpaceSymbol}, respectively.
Then the space $\big( \SumSpaceSymbol, \|\bullet\|_{\SumSpaceSymbol} \big)$ is a Banach space.

Furthermore, if $f : X \to \CC$ is measurable, then $f \in \SumSpaceSymbol$ if and only if
$f \cdot g \in \lebesgue^1 (\mu)$ for all $g \in \IntersectionSpaceSymbol$.

Finally, for $f \in \SumSpaceSymbol$, we have
\begin{equation*}
  \frac{1}{16} \, \|f\|_{\SumSpaceSymbol}
  \leq \sup \Big\{
              \int_{X} |f \cdot g| \, d\mu
              \, \colon \,
              g \in \IntersectionSpaceSymbol \text{ with } \|g\|_{\IntersectionSpaceSymbol} \leq 1
            \Big\}  \leq 16 \, \|f\|_{\SumSpaceSymbol}.
\end{equation*}
\end{theorem}

With this dual characterization of $\SumSpaceSymbol = \SumSpace$ at hand, we can now complete
the proof of Theorem~\ref{thm:NecessaryConditionsForCompatibleSpaces}.
\begin{proof}[Proof of Part~(\ref{enu:KernelDomainEmbedding})
of Theorem~\ref{thm:NecessaryConditionsForCompatibleSpaces}]

By the same reasoning as in the proof of Part~(\ref{enu:KernelCoDomainEmbedding}),
the \emph{continuous} embedding
$\BanachOne \hookrightarrow \SumSpaceSymbol_{1/w}(\nu)$
holds if and only if we have $\BanachOne \subset \SumSpaceSymbol_{1/w}(\nu)$ as sets.

Thus, let $f \in \BanachOne$.
Proving $f \in \SumSpaceSymbol_{1/w}(\nu)$ means proving that
$g := \frac{1}{w} \cdot f \in \SumSpaceSymbol(\nu)$.
For this, it suffices by Theorem~\ref{thm:DualOfIntersection} to prove that
$g \cdot h \in \lebesgue^{1}(\nu)$ for all
\(
  h \in \IntersectionSpaceSymbol(\nu).
\)
Thus, let $h \in \IntersectionSpaceSymbol(\nu)$ be arbitrary,
and define $\psi := \frac{1}{w} \cdot |h|$, noting that $\psi \in \IntersectionSpaceSymbol_w(\nu)$.
The rest of the proof again proceeds similar to the proof of Part~(\ref{enu:KernelCoDomainEmbedding}):
By applying Lemma~\ref{lem:SolidSpacesContainGoodFunctions}
(to the measure space $(X,\mu)$ and with $\BanachOne = \lebesgue^1(\mu)$,
noting that $\lebesgue^1(\mu)$ is non-trivial since $\mu(X) \neq 0$ and $\mu$ is $\sigma$-finite),
we obtain a measurable set $E \subset X$ such that $\mu(E) > 0$ and such that
$\varphi := \Indicator_E \in \IntersectionSpaceSymbol(\mu) \cap \IntersectionSpaceSymbol_v(\mu)$.
Finally, Lemma~\ref{lem:KernelModuleRichnessResult} shows that
$K := \varphi \otimes \psi \in \BBm(X,Y)$; by our assumption, this means that
$\big( \Phi_K |f| \big) (x) < \infty$ for $\mu$-almost all $x \in X$, since $|f| \in \BanachOne$.
In particular since $\mu(E) > 0$, there is some $x \in E$ such that
\[
  \infty
  > \big( \Phi_K |f| \big) (x)
  = \varphi(x)
    \cdot \int_Y
            |f(y)| \cdot \psi(y)
          \, d \nu(y)
  = \int_Y
      |g(y)| \cdot |h(y)|
    \, d \nu(y) ,
\]
which means that indeed $g \cdot h \in \lebesgue^1(\nu)$, as was to be shown.
\end{proof}

\section{Boundedness of \texorpdfstring{$\Phi_K : \BanachOne \to \lebesgue^\infty_{1/v}$}
                                       {ΦK from 𝐀 into a weighted 𝑳∞ space}}
\label{sec:EmbeddingIntoWeightedLInfty}

In this section, we prove Theorem~\ref{thm:BoundednessIntoWeightedLInfty}.
For doing so, we will need the following auxiliary result, the proof of which can be found
in Appendix~\ref{sub:CartesianProductSumNormProof}.

\begin{lemma}\label{lem:CartesianProductSumNorm}
  Let $\XTuple = \ProductTupleX$, where $\XIndexTuple{i}$ is a $\sigma$-finite measure space
  for $i \in \{ 1, 2 \}$.
  For $V \in \CalF_1$ and $W \in \CalF_2$, we have
  \[
    \| \Indicator_{V \times W} \|_{\SumSpace}
    \geq \min \{ 1, \mu_1(V), \mu_2(W), \mu(V \times W) \}.
  \]
\end{lemma}

Using this estimate, we can now prove Theorem~\ref{thm:BoundednessIntoWeightedLInfty}.

\begin{proof}[Proof of Theorem~\ref{thm:BoundednessIntoWeightedLInfty}]
  \textbf{Ad (1):}
  Since $J$ is countable, there exists a (not necessarily injective)
  surjection $\N \to J, n \mapsto j_n$.
  Define $\Omega_n := U_{j_n} \setminus \bigcup_{\ell = 1}^{n-1} U_{j_\ell}$, noting that
  the $\Omega_n$ are pairwise disjoint and satisfy
  $\biguplus_{n \in \N} \Omega_n = \bigcup_{n \in \N} U_{j_n} = \bigcup_{j \in J} U_j = X$,
  as well as $\Omega_n \subset U_{j_n}$.
  Using this partition, define
  \[
    w_{\CalU}^c : X \to (0,\infty),
                  x \mapsto \sum_{n \in \N}
                              \Big[
                                (w_{\CalU})_{j_n} \cdot \Indicator_{\Omega_n} (x)
                              \Big] .
  \]
  Clearly, $w_{\CalU}^c$ is well-defined (and in particular $(0,\infty)$-valued) and measurable.

  To prove Equation~\eqref{eq:ContinuousCoveringWeightCondition},
  first note that since $\CalU$ is a product-admissible covering,
  there is $C_0 > 0$ satisfying $(w_{\CalU})_i \leq C_0 \cdot (w_{\CalU})_j$ for all $i,j \in J$
  for which $U_i \cap U_j \neq \emptyset$.
  Now, let $j \in J$ and $x \in U_j$ be arbitrary.
  We have $x \in \Omega_n$ for a unique $n \in \N$, and then $w_{\CalU}^c (x) = (w_{\CalU})_{j_n}$.
  Furthermore, since $x \in U_j \cap \Omega_n \subset U_j \cap U_{j_n}$, we have
  $(w_{\CalU})_{j_n} \leq C_0 \cdot (w_{\CalU})_j$ and
  $(w_{\CalU})_{j} \leq C_0 \cdot (w_{\CalU})_{j_n}$, meaning
  \(
    \frac{w_{\CalU}^c (x)}{(w_{\CalU})_j}
    + \frac{(w_{\CalU})_j}{w_{\CalU}^c (x)}
    \leq 2 \, C_0 .
  \)
  Overall, we see
  \(
    \sup_{j \in J}
      \sup_{x \in U_j}
      \Big[
        \frac{w_{\CalU}^c (x)}{(w_{\CalU})_j}
        + \frac{(w_{\CalU})_j}{w_{\CalU}^c (x)}
      \Big]
    \leq 2 \, C_0
    < \infty ,
  \)
  as desired.

  Finally, let $v_{\CalU}^c : X \to (0, \infty)$ satisfy
  \eqref{eq:ContinuousCoveringWeightCondition};
  that is,
  \({
    C_v := \sup_{j \in J}
             \sup_{x \in U_j}
               \Big[
                 \frac{(w_{\CalU})_j}{v_{\CalU}^c (x)}
                 + \frac{v_{\CalU}^c (x)}{(w_{\CalU})_j}
               \Big]
        < \infty .
  }\)
  Let $x \in X$ and note that $x \in U_j$ for some $j \in J$.
  Hence,
  \(
    \vphantom{\sum^i}
    v_{\CalU}^c (x)
    \leq C_v \cdot (w_{\CalU})_j
    \leq 2 C_0 C_v \cdot w_{\CalU}^c (x) .
  \)
  Conversely,
  \(
    w_{\CalU}^c (x)
    \leq 2 C_0 \cdot (w_{\CalU})_j
    \leq 2 C_0 C_v \cdot v_{\CalU}^c (x) .
  \)
  Overall, this shows that $w_{\CalU}^c \asymp v_{\CalU}^c$, as claimed.

  \medskip{}

  \noindent
  \textbf{Ad (2):}
  First, note that for arbitrary $x,y \in X$, we have $x \in U_j \subset \CalU(x)$ for some $j \in J$,
  and hence $|K(x,y)| \leq \MaxKernel{\CalU} K (x,y) \leq L(x,y)$.
  Since $K$ is measurable and because of $\| L \|_{\BBm} < \infty$,
  this implies $K, |K| \in \BBm(X,X) \hookrightarrow \Bounded(\BanachOne)$.
  In case of $\BanachOne = \{ 0 \}$, the operator $\Phi_K : \BanachOne \to \lebesgue_{1/v}^\infty$
  is trivially bounded, so that we can assume in the following that $\BanachOne \neq \{ 0 \}$.
  Note that this implies $\mu(X) > 0$.
  As in Definition~\ref{def:ProductAdmissibleCovering}, let us write $U_j = V_j \times W_j$.

  Now, let $f \in \BanachOne$.
  For arbitrary $j \in J$ and $x,y \in U_j$, we then have $x \in U_j \subset \CalU(y)$,
  and hence $|K(x,z)| \leq \MaxKernel{\CalU} K (y,z) \leq L(y,z)$ for arbitrary $z \in X$.
  Therefore,
  \[
    |\Phi_K f (x)|
    \leq \int_X
           |K(x,z)| \cdot |f(z)|
         \, d \mu(z)
    \leq \int_X
           L(y,z) \cdot |f(z)|
         \, d \mu (z)
    =    (\Phi_L \, |f|) (y) .
  \]
  Since this holds for all $x,y \in U_j$, we see that if we define
  \[
    \Theta_j := \sup_{x \in U_j} |\Phi_K f (x)| \qquad \text{for } j \in J,
  \]
  then $\Theta_j \leq (\Phi_L \, |f|)(y)$ for all $y \in U_j$, meaning that
  \(
    \Theta_j \cdot \Indicator_{U_j} (y) \leq (\Phi_L \, |f|) (y)
  \)
  for all $j \in J$ and $y \in X$.
  Since $\Phi_L \, |f| \in \BanachOne$ and by solidity of $\BanachOne$, this implies
  \begin{equation}
    \Theta_j \cdot \big\| \Indicator_{U_j} \big\|_{\BanachOne}
    \leq \big\|
           \Phi_L \, |f|
         \big\|_{\BanachOne}
    \leq C^{(1)} \cdot \| \Phi_L \|_{\BBm} \cdot \| f \|_{\BanachOne}
    =:   C^{(2)} \cdot \| f \|_{\BanachOne}
    \qquad \forall \, j \in J .
    \label{eq:WeightedLInfityEmbeddingStep1}
  \end{equation}
  Here, the constant $C^{(1)} = C^{(1)}(\BanachOne, m)$ is provided by our assumption
  $\BBm(X,X) \hookrightarrow \Bounded(\BanachOne)$.

  Next, Part~(\ref{enu:KernelDomainEmbedding})
  of Theorem~\ref{thm:NecessaryConditionsForCompatibleSpaces} shows
  because of $\BBm(X,X) \hookrightarrow \Bounded(\BanachOne)$ and $\mu(X) \neq 0$,
  and thanks to our assumption $m(x,y) \leq C \cdot u(x) \cdot u(y)$ that
  $\BanachOne \hookrightarrow \SumSpaceSymbol_{1/u}$.
  As before, we use the notation $\SumSpaceSymbol := \SumSpace$ for brevity.
  Hence, there is $C^{(3)} = C^{(3)} (\BanachOne, u) > 0$ such that
  \(
    \| \Indicator_{U_j} \|_{\SumSpaceSymbol_{1/u}}
    \leq C^{(3)} \cdot \| \Indicator_{U_j} \|_{\BanachOne}
  \)
  for all $j \in J$.

  Now, since $u$ is $\CalU$-moderate, there is $C^{(4)} > 0$ satisfying
  $u (y) \leq C^{(4)} \cdot u(x)$ and hence also $\frac{1}{u(x)} \leq C^{(4)} \cdot \frac{1}{u(y)}$
  for all $j \in J$ and all $x, y \in U_j$.
  This entails
  $\frac{1}{u(x)} \Indicator_{U_j} (y) \leq C^{(4)} \cdot \frac{1}{u(y)} \Indicator_{U_j}(y)$
  for all $x \in U_j$ and all $y \in X$.
  In combination with Lemma~\ref{lem:CartesianProductSumNorm}, this implies
  \begin{align*}
    \frac{(w_{\CalU})_j}{u(x)}
    & = \frac{\min \{ 1, \mu_1 (V_j), \mu_2 (W_j), \mu(V_j \times W_j) \}}{u(x)}
      \leq \frac{1}{u(x)} \cdot \| \Indicator_{V_j \times W_j} \|_{\SumSpaceSymbol}
      =    \frac{1}{u(x)} \cdot \| \Indicator_{U_j} \|_{\SumSpaceSymbol} \\
    & \leq C^{(4)} \cdot \Big\| \frac{1}{u} \, \Indicator_{U_j} \Big\|_{\SumSpaceSymbol}
      =    C^{(4)} \cdot \| \Indicator_{U_j} \|_{\SumSpaceSymbol_{1/u}}
      \leq C^{(3)} C^{(4)} \cdot \| \Indicator_{U_j} \|_{\BanachOne}
    \qquad \forall \, j \in J \text{ and } x \in U_j .
  \end{align*}

  By construction of the weight $w_{\CalU}^c$, the constant
  $C^{(5)} := \sup_{j \in J} \sup_{x \in U_j} \frac{w_{\CalU}^c (x)}{(w_{\CalU})_j}$
  is finite; furthermore, by definition of the weight $v$
  (see Equation~\eqref{eq:SpecialLInftyWeight}), we see
  $\frac{1}{v(x)} = \frac{w_{\CalU}^c (x)}{u(x)} \leq C^{(5)} \cdot \frac{(w_{\CalU})_j}{u(x)}$
  for all $x \in U_j$.
  By combining these observations with Equation~\eqref{eq:WeightedLInfityEmbeddingStep1}
  and recalling the definition of $\Theta_j$, we conclude that
  \[
    \frac{1}{v(x)} \, |\Phi_K f (x)|
    \leq \Theta_j \cdot \frac{1}{v(x)}
    \leq C^{(3)} C^{(4)} C^{(5)} \cdot \Theta_j \cdot \| \Indicator_{U_j} \|_{\BanachOne}
    \leq C^{(2)} C^{(3)} C^{(4)} C^{(5)} \cdot \| f \|_{\BanachOne}
  \]
  for all $j \in J$ and $x \in U_j$.
  Since $X = \bigcup_{j \in J} U_j$, this implies
  $\| \Phi_K f \|_{\lebesgue^{\infty}_{1/v}} \leq C^{(6)} \cdot \| f \|_{\BanachOne}$.
  Note that the constant $C^{(6)} := C^{(2)} C^{(3)} C^{(4)} C^{(5)}$
  satisfies $C^{(6)} = C^{(6)} (\BanachOne, \CalU, u, m, w_{\CalU}^c, L)$; that is,
  it is independent of the choice of $f \in \BanachOne$.
\end{proof}

\section{Application: General coorbit theory with mixed-norm Lebesgue spaces}
\label{sec:CoorbitTheory}

In this section, we first give a more precise exposition of coorbit theory,
based on the article \cite{kempka2015general} by Kempka et al.
We then give rigorous formulations and proofs of
Theorems~\ref{thm:IntroductionCoorbitWellDefinedConditions}
and \ref{thm:IntroductionCoorbitDiscretizationConditions}.

\subsection{A formal review of general coorbit theory}%
\label{sub:CoorbitReview}

As in Section~\ref{sec:IntroCoorbitTheory}, we assume throughout that
$\CalH$ is a separable Hilbert space, and that $\XTuple = \ProductTupleX$,
where $\XIndexTuple{j}$ ($j \in \{1,2\}$) is a $\sigma$-compact,
locally compact Hausdorff space with Borel $\sigma$-algebra $\CalF_j$,
and a Radon measure $\mu_j$ with $\supp \mu_j = X_j$.
Furthermore, we assume that $\CalF_1 \otimes \CalF_2$ is the Borel $\sigma$-algebra on $X$
and that $\mu$ is a Radon measure; this holds for instance if $X_1, X_2$ are second-countable;
see \cite[Theorem~7.20]{FollandRA}.
Finally, we assume that $\Psi = (\psi_x)_{x \in X} \subset \CalH$ is a continuous Parseval frame
(see Equation~\eqref{eq:ParsevalFrameCondition})
with reproducing kernel $K_\Psi$ as in Equation~\eqref{eq:ReproducingKernelDefinition}.

As mentioned in Section~\ref{sec:IntroCoorbitTheory},
the reproducing kernel $K_\Psi$ and the Banach function space $\BanachOne$
have to satisfy certain technical conditions to ensure that the coorbit space $\Co_\Psi (\BanachOne)$
is well-defined.
In the following two lemmas, we first collect the conditions from \cite{kempka2015general}
which ensure that the reservoir $\Reservoir$ is well-defined, and then the conditions which guarantee
that also the coorbit spaces $\Co_\Psi (\BanachOne)$ are well-defined.
Simplifications of these conditions will be presented in the next subsection.

As previously in Equation~\eqref{eq:SeparableMatrixWeight},
for a measurable function $v : X \to (0,\infty)$, we define
${ m_v : X \times X \rightarrow (0,\infty), (x,y) \mapsto \max\{v(x)/v(y),v(y)/v(x)\} }$.

\begin{lemma}[{\cite[Section 2.3]{kempka2015general}}]\label{lem:ReservoirConditions}
  Let $v : X \to [1,\infty)$ be measurable and
  assume that there is a constant $C_B > 0$ such that
  \begin{equation}
    \|\psi_x\|_{\Hil} \leq C_B \, v(x) \quad \text{for all} \quad x \in X,
    \qquad \text{and} \qquad
    K_\Psi \in \AAi_{m_v}.
    \label{eq:CoorbitReservoirAssumptions}
  \end{equation}
  Then all results in \cite[Section~2.3]{kempka2015general} apply.
  In particular, the following hold:
  \begin{itemize}
   \item[(1)] The space
              \(
                \Hil^1_v
                := \{
                     f \in \Hil
                     \colon
                     \| V_\Psi f \|_{\lebesgue^1_v (\mu)} < \infty
                   \}
                ,
              \)
              equipped with the norm
              \(
               \|f\|_{\Hil^1_v} := \|V_\Psi f\|_{\lebesgue^1_v(\mu)}
              \)
              is a Banach space which satisfies $\Hil^1_v \hookrightarrow \Hil$.
              Furthermore, there is a null-set $N \subset X$ such that $\psi_x \in \Hil^1_v$
              for all $x \in X \setminus N$.
              Finally, the map $X \setminus N \to \Hil^1_v, x \mapsto \psi_x$ is Bochner measurable.
              \vspace{0.1cm}

   \item[(2)] The extension of $V_\Psi$ to the antidual\footnote{That is,
              $\Reservoir = \{ \varphi : \Hil_v^1 \to \CC \colon \varphi \text{ continuous and antilinear} \}$,
              where $\varphi$ is called antilinear if $\varphi(x + y) = \varphi(x) + \varphi(y)$
              and $\varphi(\alpha x) = \overline{\alpha} \, \varphi(x)$ for all $x,y \in \Hil_v^1$
              and $\alpha \in \CC$.}
              $\Reservoir := (\Hil^1_v)^\urcorner$ of $\Hil^1_v$, given by
              \begin{equation}
                V_{\Psi} : \Reservoir \to \lebesgue^\infty_{1/v}(\mu), f \mapsto V_{\Psi} f
                \quad \text{with} \quad
                V_{\Psi} f \colon X \setminus N \to \CC,
                                  x \mapsto \langle f,\psi_x\rangle_{(\Hil^1_v)^\urcorner,\Hil^1_v}
              \end{equation}
              is a well-defined, continuous and injective map from $(\Hil^1_v)^\urcorner$
              into $\lebesgue^\infty_{1/v}(\mu)$, and $f \mapsto \| V_\Psi f \|_{\lebesgue_{1/v}^\infty}$
              is an equivalent norm on $\Reservoir$.
              \vspace{0.1cm}

   \item[(2)] For ${F \in \lebesgue^\infty_{1/v}(\mu)}$, we have $F = \Phi_{K_\Psi}(F)$ if and only if
              $F = V_\Psi f$ for some $f \in (\Hil^1_v)^\urcorner$.
  \end{itemize}
\end{lemma}

\begin{rem*}
  The possible issue that $\psi_x \notin \Hil_v^1$ on a null-set can be circumvented by redefining
  $\psi_x = 0$ on this null-set.
  Thus, as in \cite{kempka2015general} (see \cite[Remark~2.16]{kempka2015general}),
  we assume in the following that $\psi_x \in \Hil_v^1$ for \emph{all} $x \in X$.
\end{rem*}

\begin{lemma}[{\cite[Section 2.4]{kempka2015general}}]\label{lem:CoorbitConditions}
  Suppose that the assumptions of Lemma~\ref{lem:ReservoirConditions} are satisfied.
  Let $\BanachOne$ be a rich solid Banach function space on $X$, and assume that
  \begin{enumerate}
    \item The integral operator $\Phi_{K_\Psi}$ acts continuously on $\BanachOne$;
          \vspace{0.1cm}

    \item We have $\Phi_{K_\Psi}(\BanachOne) \hookrightarrow \lebesgue^\infty_{1/v} (\mu)$,
          meaning that there is a constant $C > 0$ such that
          \begin{equation}
            \| \Phi_{K_\Psi} (f) \|_{\lebesgue^\infty_{1/v}}
            \leq C \cdot \| \Phi_{K_\Psi} (f) \|_{\BanachOne}
            \qquad \forall \, f \in \BanachOne .
            \label{eq:CoorbitLInftyEmbeddingAssumption}
          \end{equation}
  \end{enumerate}
  Then all results in \cite[Section~2.4]{kempka2015general} apply; in particular the space
  \[
    \Co (\BanachOne)
    = \Co_{\Psi,v} (\BanachOne)
    := \big\{ f \in (\Hil^1_v)^{\urcorner} \colon V_\Psi f \in \BanachOne \big\}
    \quad \text{with norm} \quad
    \| f \|_{\Co (\BanachOne)} := \| V_\Psi f \|_{\BanachOne}
  \]
  is a Banach space; also, $F \in \BanachOne$ is of the form $F = V_\Psi f$
  for some $f \in \Co(\BanachOne)$ if and only if ${F = \Phi_{K_\Psi}(F)}$.
\end{lemma}

\begin{remark}
  By definition, $\Co_{\Psi,v} (\BanachOne)$ is dependent on the weight $v$
  and the analyzing frame $\Psi$.
  However, in \cite[Lemma~2.26]{kempka2015general}, it is shown that for any other weight $\tilde{v}$
  for which the assumptions of Lemmas~\ref{lem:ReservoirConditions}--\ref{lem:CoorbitConditions}
  are satisfied, $\Co_{\Psi,v} (\BanachOne)= \Co_{\Psi,\tilde{v}} (\BanachOne)$.
  Furthermore, in \cite[Lemma~2.29]{kempka2015general}, it is shown that
  $\Co_{\Psi,v} (\BanachOne) = \Co_{\widetilde{\Psi},v} (\BanachOne)$ for any continuous
  Parseval frame $\widetilde{\Psi} = \{\widetilde{\psi}_x\}_{x\in X}$ for which the mixed kernel
  \(
    K_{\Psi,\widetilde{\Psi}} :
    X \times X \rightarrow \CC,
    (x,y) \mapsto \langle \psi_y, \widetilde{\psi}_x \rangle
  \)
  and its transpose $K_{\Psi,\widetilde{\Psi}}^T$ are both contained in $\AAi_{m_v}$
  and induce bounded integral operators on $\BanachOne$.
  All in all, this justifies the compact notation $\Co (\BanachOne)$ for $\Co_{\Psi,v} (\BanachOne)$.
\end{remark}

To formally state the discretization theory for the coorbit spaces,
we first need to properly define the sequence spaces $\BanachOne^\flat$
and $\BanachOne^\sharp$ that occur in the definition
of the analysis and synthesis operators; see Equation~\eqref{eq:CoefficientSynthesisOperator}.

\begin{definition}[{see \cite[Section~2.2]{kempka2015general}}]\label{def:SequenceSpaces}
  Let $\CalU = (U_i)_{i \in I}$ be an admissible covering of $X$
  and let $\BanachOne$ be a rich solid Banach function space on $X$.
  The spaces $\BanachOne^\flat = \BanachOne^\flat(\CalU)$
  and $\BanachOne^\sharp = \BanachOne^\sharp(\CalU)$ are comprised by
  the sequences $(\lambda_i)_{i\in I} \in \CC^I$ for which the norms
  \begin{equation}
    \| (\lambda_i)_{i \in I} \|_{\BanachOne^\flat}
      := \bigg\|
           \sum_{i \in I}
             |\lambda_i| \, \Indicator_{U_i}
         \bigg\|_{\BanachOne}
    \qquad \text{and} \qquad
    \| (\lambda_i)_{i \in I} \|_{\BanachOne^\sharp}
      := \bigg\|
           \sum_{i \in I}
             \frac{|\lambda_i|}{\mu(U_i)} \cdot \Indicator_{U_i}
         \bigg\|_{\BanachOne}
    \label{eq:SequenceSpaceNorms}
  \end{equation}
  are finite, respectively.
\end{definition}

As a convenient shorthand, we will use the notation
$\|K\|_\AAvA := \max \{\|m_v \, K\|_{\AAi}, \| \Phi_K \|_{\BanachOne \to \BanachOne}\}$
for a given kernel $K$, where $m_v$ is as defined in Equation~\eqref{eq:SeparableMatrixWeight}.

\begin{theorem}[{\cite[Theorem 2.48]{kempka2015general}}]\label{thm:coorbitsdisc}
  Suppose that the continuous frame $\Psi = (\psi_x)_{x \in X}$,
  the weight $v : X \to [1,\infty)$ and the solid Banach function space $\BanachOne$
  satisfy the assumptions of Lemma~\ref{lem:CoorbitConditions}
  (and thus also those of Lemma~\ref{lem:ReservoirConditions}).

  Assume $\| \, |K_\Psi| \, \|_{\AAvA} < \infty$.
  Let $\CalU = (U_i)_{i\in I}$ an admissible covering of $X$,
  let ${\Gamma : X \times X \to S^1}$, let $\oscUG$ be as defined
  in Equation~\eqref{eq:OscillationDefinition}, and let
  $L : X \times X \to [0,\infty]$ be measurable with $\oscUG (K_\Psi) \leq L$.
  Define
  \[
    \delta
    = \delta(v, \BanachOne, L)
    := \max \big\{
               \| L \|_\AAvA,
               \quad
               \big\| L^T \big\|_\AAvA
            \big\}
    \in [0,\infty]
    .
  \]
  If
  \begin{equation}\label{eq:DiscretizationInvertibilityEstimate}
    \delta \cdot \big( 2 \, \| \, | K_\Psi | \, \|_{\AAvA} + \delta \big) < 1,
  \end{equation}
  and if for each $i \in I$ some $x_i \in U_i$ is chosen arbitrarily, then
  there exists a ``dual'' family ${\Phi_d = (\varphi_i)_{i\in I} \subset \Hil^1_v \cap \Co(\BanachOne)}$,
  such that the following are true:
  \begin{itemize}
   \item[(1)] A given $f \in (\Hil^1_v)^\urcorner$ is an element of $\Co(\BanachOne)$ if and only if
              \(
                \big( \langle f, \psi_{x_i} \rangle_{(\Hil^1_v)^{\urcorner}, \Hil_v^1} \big)_{i\in I}
                \in \BanachOne^\flat(\UU)
                ,
              \)
              if and only if
              \(
                \big( \langle f, \varphi_{i} \rangle_{(\Hil_v^1)^{\urcorner}, \Hil_v^1} \big)_{i\in I}
                \in \BanachOne^\sharp(\UU)
                .
              \)
              In this case, we have
              \begin{equation}
                \|f\|_{\Co(\BanachOne)}
                \asymp \big\|
                         \big(
                           \langle f, \psi_{x_i}\rangle_{(\Hil^1_v)^{\urcorner}, \Hil_v^1}
                         \big)_{i \in I}
                       \big\|_{\BanachOne^\flat(\UU)}
                \asymp \big\|
                         \big(
                           \langle f, \varphi_{i} \rangle_{(\Hil^1_v)^{\urcorner}, \Hil_v^1}
                         \big)_{i \in I}
                       \big\|_{\BanachOne^\sharp (\UU)} .
              \end{equation}

   \item[(2)] If $(\lambda_i)_{i\in I} \in \BanachOne^\sharp(\UU)$,
              then $\sum_{i\in I} \lambda_i \, \psi_{x_i} \in \Co(\BanachOne)$ with
              \(
                \big\| \sum_{i \in I} \lambda_i \, \psi_{x_i} \big\|_{\Co(\BanachOne)}
                \lesssim \|(\lambda_i)_{i\in I}\|_{\BanachOne^\sharp(\UU)}
                ,
              \)
              with unconditional convergence of the series in the weak-$\ast$ topology
              induced by $(\Hil^1_v)^\urcorner$.\\
              If  $(\lambda_i)_{i \in I} \in \BanachOne^\flat(\UU)$,
              then $\sum_{i\in I} \lambda_i \, \varphi_{i} \in \Co(\BanachOne)$ with
              \(
                \big\| \sum_{i\in I} \lambda_i \, \varphi_{i} \big\|_{\Co(\BanachOne)}
                \lesssim \|(\lambda_i)_{i \in I}\|_{\BanachOne^\flat(\UU)}
                .
              \)
              \vspace{0.1cm}

   \item[(3)] For all $f\in\Co(\BanachOne)$, we have
              \(
                \sum_{i \in I} \,
                  \langle f, \varphi_{i} \rangle_{(\Hil_v^1)^{\urcorner}, \Hil_v^1}
                  \,\, \psi_{x_i}
                = f
                = \sum_{i \in I} \,
                    \langle f, \psi_{x_i}\rangle_{(\Hil_v^1)^{\urcorner}, \Hil_v^1}
                    \,\, \varphi_{i}
                .
              \)
  \end{itemize}
\end{theorem}

\begin{rem*}
  In \cite{kempka2015general}, the above conditions are formulated directly in terms of
  $\oscUG (K_\Psi)$ and not using the auxiliary kernel $L$.
  This, however, has the minor issue that in general $\oscUG (K_\Psi)$ might not be measurable,
  since $\Gamma$ might not be measurable and also since the definition of $\oscUG(K_\Psi)$
  involves taking a supremum over a (potentially) uncountable set.
  The above formulation circumvents these problems;
  it can be obtained via straightforward minor modifications of the arguments
  given in \cite{kempka2015general}.
\end{rem*}

\subsection{Simplification of the well-definedness conditions}%
\label{sub:CoorbitWelldefinedSimplification}

While Lemma~\ref{lem:ReservoirConditions} relies only on properties of the reproducing kernel $K_\Psi$,
Lemma~\ref{lem:CoorbitConditions} must be verified individually for every target space $\BanachOne$,
except when we generally have $\AAi_{m_v} \hookrightarrow \Bounded(\BanachOne,\BanachOne)$ and
$\AAi_{m_v}(\BanachOne) \hookrightarrow \lebesgue^\infty_{1/v}(\mu)$.
As shown in \cite[Corollary~4]{GeneralizedCoorbit1}, both conditions are satisfied if
$\BanachOne = \lebesgue^p(\mu)$ and if there is an admissible covering $\CalU = (U_i)_{i \in I}$
such that the maximal kernel $\MaxKernel{\CalU}(K_\Psi)$ is in $\AAi_{m_v}$
and ${\sup_{i \in I} \sup_{x,y \in U_i} v(x) / v(y) < \infty}$.

For more general choices of $\BanachOne$, either condition may be violated.
In particular, the embedding $\AAi_{m_v} \hookrightarrow \Bounded(\BanachOne,\BanachOne)$
is not generally true when $\BanachOne = \lebesgue^{p,q}_w(\mu)$.
This is a significant obstruction to the development of a coherent coorbit theory for such spaces.

Using the kernel algebras $\BBm$, we will now state a set of
conditions which ensures that Lemmas~\ref{lem:ReservoirConditions} and \ref{lem:CoorbitConditions}
can be applied for a larger family of functions spaces, including weighted, mixed-norm Lebesgue spaces.

\begin{proposition}\label{pro:coorbitsWithBm}
   Let $\XTuple$ and $\Psi = (\psi_x)_{x \in X} \subset \Hil$ as at the beginning
   of Section~\ref{sub:CoorbitReview}.
   Furthermore, assume that Conditions~\eqref{eq:CoorbitWeightConditions}
   and \eqref{eq:CoorbitWellDefinedConditions} are satisfied.

   Then the conditions of Lemmas~\ref{lem:ReservoirConditions} and \ref{lem:CoorbitConditions}
   are satisfied.
   Therefore, $\Co(\BanachOne) = \Co_{\Psi,v}(\BanachOne)$ is a well-defined Banach space.
\end{proposition}

\begin{proof}
  Condition~\eqref{eq:CoorbitReservoirAssumptions} is an immediate consequence
  of Conditions~\eqref{eq:CoorbitWeightConditions} and \eqref{eq:CoorbitWellDefinedConditions}.
  Thus, all assumptions of Lemma~\ref{lem:ReservoirConditions} are satisfied.

  Next, since $\BBi_{m_0}(X) \hookrightarrow \Bounded(\BanachOne)$ and $\BanachOne \neq \{ 0 \}$
  as well as $m_0(x,y) \leq C \cdot u(x) \, u(y)$ by Condition~\eqref{eq:CoorbitWeightConditions},
  Theorem~\ref{thm:NecessaryConditionsForCompatibleSpaces}(\ref{enu:KernelCoDomainEmbedding})
  shows that
  \(
    (\lebesgue^1 \cap \lebesgue^\infty \cap \lebesgue^{1,\infty} \cap \lebesgue^{\infty,1})_u
    \hookrightarrow \BanachOne.
  \)
  Since $u$ is locally bounded and since each compact set $K \subset X$ satisfies
  $K \subset K_1 \times K_2$ for suitable compact sets $K_i \subset X_i$,
  we see that the space on the left-hand side contains $\Indicator_K$
  for each compact $K \subset X$, which shows that $\BanachOne$ is rich.

  Further, note that for each $x \in X$, we have $x \in U_j \subset \CalU(x)$ for some $j \in J$,
  and hence ${L(x,y) \geq (M_{\CalU} K_\Psi ) (x,y) \geq |K_\Psi (x,y)|}$.
  Since $L \in \BBi_{m_0}$ and since $K_\Psi$ is measurable,
  this implies by solidity of $\BBi_{m_0}$ that
  $K_\Psi, |K_\Psi| \in \BBi_{m_0} (X) \hookrightarrow \Bounded(\BanachOne)$, so that
  $\Phi_{K_\Psi} : \BanachOne \to \BanachOne$ and $\Phi_{|K_\Psi|} : \BanachOne \to \BanachOne$
  are bounded.

  Finally, since $u$ is $\CalU$-moderate with $m_0(x,y) \leq C \cdot u(x) \, u(y)$,
  and since ${M_{\CalU} K_\Psi \leq L \in \BBi_{m_0}(X)}$ and
  $\BBi_{m_0}(X) \hookrightarrow \Bounded(\BanachOne)$,
  Part~(2) of Theorem~\ref{thm:BoundednessIntoWeightedLInfty} shows
  for $v_0 : X \to (0,\infty), x \mapsto [w_{\CalU}^c (x)]^{-1} \, u(x)$
  that $\Phi_{K_\Psi} : \BanachOne \to \lebesgue_{1/v_0}^\infty (\mu)$ is well-defined and bounded.
  Note that Condition~\eqref{eq:CoorbitWeightConditions} implies that $v \gtrsim v_0$,
  so that also $\Phi_{K_\Psi} : \BanachOne \to \lebesgue_{1/v}^\infty (\mu)$ is well-defined and bounded.
  By Lemma~\ref{lem:BoundednessEmbeddingEquivalence}, this implies that
  $\Phi_{K_\Psi} (\BanachOne) \hookrightarrow \lebesgue_{1/v}^\infty (\mu)$.
  We have thus verified all assumptions of Lemma~\ref{lem:CoorbitConditions}.
\end{proof}

\subsection{Simplification of the discretization conditions}%
\label{sub:CoorbitDiscretizationSimplification}

Here we present our unified conditions for discretization in coorbit spaces $\Co(\BanachOne)$
for the setting where $\BBi_{m_{0}}(X) \hookrightarrow \Bounded(\BanachOne)$.

\begin{proposition}\label{pro:coorbitsWithBm2}
   Suppose that the assumptions of Proposition~\ref{pro:coorbitsWithBm} hold,
   and set ${m := m_v + m_0}$.
   If there exist an admissible covering $\widetilde{\CalU} = ( \widetilde{U}_i )_{i \in I}$ of $X$,
   a phase function ${\Gamma : X \times X \to S^1}$, and some $L \in \BBm$ such that
   \begin{equation}\label{eq:KernelsInSmallAlgebra}
     K_\Psi, L \in \BBm \quad \text{and} \quad \oscUtG(K_\Psi) \leq L
   \end{equation}
   and
   \begin{equation}\label{eq:IdMinusDiscNormEstimate}
    \| L \|_{\BBi_{m}} \cdot (2 \, \| K_\Psi \|_{\BBi_{m}} + \| L \|_{\BBi_{m}}) < 1,
   \end{equation}
   then all assumptions of Theorem~\ref{thm:coorbitsdisc} (with $\widetilde{\CalU}$ instead of $\CalU$)
   are satisfied.

   In particular, if for each $i \in I$ a point $x_i \in \widetilde{U}_i$ is chosen arbitrarily,
   then there exists a ``dual'' family $\Phi_d = (\varphi_i)_{i\in I} \subset \Hil^1_{v} \cap \Co(\BanachOne)$,
   such that the following are true:
   \begin{itemize}
     \item[(1)] $f\in (\Hil^1_{v})^\urcorner$ is an element of $\Co(\BanachOne)$ if and only if
                \(
                  \big(
                    \langle f, \psi_{x_i} \rangle_{(\Hil_v^1)^{\urcorner}, \Hil_v^1}
                  \big)_{i \in I} \in \BanachOne^\flat(\widetilde{\UU})
                \)
                if and only if
                \(
                  \big(
                    \langle f, \varphi_{i} \rangle_{(\Hil_v^1)^{\urcorner}, \Hil_v^1}
                  \big)_{i\in I} \in \BanachOne^\sharp(\widetilde{\UU})
                  ,
                \)
                and in this case
                \begin{equation}
                  \| f \|_{\Co(\BanachOne)}
                  \asymp \big\|
                           \big(
                             \langle f, \psi_{x_i} \rangle_{(\Hil_v^1)^{\urcorner}, \Hil_v^1}
                           \big)_{i \in I}
                         \big\|_{\BanachOne^\flat(\widetilde{\UU})}
                  \asymp \big\|
                           \big(
                             \langle f, \varphi_{i} \rangle_{(\Hil_v^1)^{\urcorner}, \Hil_v^1}
                           \big)_{i\in I}
                         \big\|_{\BanachOne^\sharp(\widetilde{\UU})}
                  .
                \end{equation}

     \item[(2)] If $(\lambda_i)_{i\in I} \in \BanachOne^\sharp(\widetilde{\UU})$,
                then $\sum_{i\in I} \lambda_i \, \psi_{x_i} \in \Co(\BanachOne)$ with
                \(
                  \|
                    \sum_{i \in I}
                      \lambda_i \, \psi_{x_i}
                  \|_\BanachOne
                  \lesssim \|(\lambda_i)_{i \in I}\|_{\BanachOne^\sharp(\widetilde{\UU})}
                  ,
                \)
                with unconditional convergence of the series in the weak-$\ast$ topology
                induced by $(\Hil^1_{v})^\urcorner$.
                Further, if  $(\lambda_i)_{i\in I} \! \in \! \BanachOne^\flat(\widetilde{\UU})$,
                then $\sum_{i\in I} \lambda_i \, \varphi_{i} \! \in \! \Co(\BanachOne)$ and
                \(
                  \|\sum_{i \in I} \lambda_i \, \varphi_{i}\|_\BanachOne
                  \lesssim \|(\lambda_i)_{i \in I}\|_{\BanachOne^\flat(\widetilde{\UU})}
                  .
                \)

     \item[(3)] For all $f\in\Co(\BanachOne)$, we have
                \(
                  \sum_{i \in I} \,
                    \langle f, \varphi_{i}\rangle_{(\Hil_v^1)^{\urcorner}, \Hil_v^1}
                    \,\, \psi_{x_i}
                  = f
                  = \sum_{i \in I} \,
                      \langle f, \psi_{x_i}\rangle_{(\Hil_v^1)^{\urcorner}, \Hil_v^1}
                      \,\, \varphi_{i}
                  .
                \)
  \end{itemize}
\end{proposition}

\begin{proof}
  Since by assumption Proposition~\ref{pro:coorbitsWithBm} is applicable, that proposition
  shows that the assumptions of Lemma~\ref{lem:CoorbitConditions} are satisfied.

  Next, recall from Part~(4) of Proposition~\ref{prop:NewKernelModuleBasicProperties1} that
  $\| \bullet \|_{\AAi_{m_v}} \leq \| \bullet \|_{\BBi_{m_v}} \leq \| \bullet \|_{\BBi_m}$.
  Furthermore, Condition~\eqref{eq:CoorbitWeightConditions} shows that
  \(
    \| \Phi_\bullet \|_{\BanachOne \to \BanachOne}
    \leq \| \bullet \|_{\BBi_{m_0}}
    \leq \| \bullet \|_{\BBi_m}
    .
  \)
  Overall, we thus see with $\| \bullet \|_{\AAvA}$ as defined
  before Theorem~\ref{thm:coorbitsdisc} that $\| K \|_{\AAvA} \leq \| K \|_{\BBm}$
  for each measurable kernel $K$.

  Since $K_\Psi, L \in \BBm$, we thus see that $|K_\Psi|, L \in \AAvA$, where we note that
  $\oscUtG(K_\Psi) \leq L$, as required in Theorem~\ref{thm:coorbitsdisc}.
  Furthermore, since $m_v^T = m_v$ and $m_0^T = m_0$ we also have $m^T = m$,
  so that Proposition~\ref{prop:NewKernelModuleBasicProperties2} shows
  $\| K^T \|_{\BBm} = \| K \|_{\BBm}$ for each measurable kernel $K$.
  Overall, we thus see that the constant $\delta$ in Theorem~\ref{thm:coorbitsdisc} satisfies
  \[
    \delta
    = \max \big\{ \| L \|_{\AAvA}, \| L^T \|_{\AAvA} \big\}
    \leq \max \big\{ \| L \|_{\BBm}, \| L^T \|_{\BBm} \big\}
    =    \| L \|_{\BBm},
  \]
  and hence
  \[
    \delta \cdot \big( 2 \, \| \,|K_\Psi|\, \|_{\AAvA} + \delta \big)
    \leq \| L \|_{\BBm} \cdot \big( 2 \, \| \, |K_\Psi| \, \|_{\BBm} + \| L \|_{\BBm} \big)
    =    \| L \|_{\BBm} \cdot \big( 2 \, \| K_\Psi \|_{\BBm} + \| L \|_{\BBm} \big)
    < 1;
  \]
  see Equation~\eqref{eq:IdMinusDiscNormEstimate}.
  This completes the proof.
\end{proof}

\appendix

\section{A dual characterization of the space
\texorpdfstring{$\lebesgue^1 + \lebesgue^\infty + \lebesgue^{1,\infty} + \lebesgue^{\infty,1}$}
               {𝑳¹ + 𝑳∞ + 𝑳¹∞ + 𝑳∞¹}}
\label{sec:DualCharacterizationOfSumSpace}

The main objective of this appendix is to prove Theorem~\ref{thm:DualOfIntersection}.
That is, we show that if $F : X_1 \times X_2 \to \CC$ satisfies
$F \cdot G \in \lebesgue^1 (\mu_1 \otimes \mu_2)$ for all
$G \in \lebesgue^1 \cap \lebesgue^\infty \cap \lebesgue^{1,\infty} \cap \lebesgue^{\infty,1}$,
then ${F \in \lebesgue^1 + \lebesgue^\infty + \lebesgue^{1,\infty} + \lebesgue^{\infty,1}}$,
with a corresponding norm estimate.
Along the way, we will also obtain everything necessary to prove
Lemma~\ref{lem:CartesianProductSumNorm}; see Appendix~\ref{sub:CartesianProductSumNormProof}.

The general structure of the proof of Theorem~\ref{thm:DualOfIntersection} is as follows:
First, we show that there is an equivalent norm $\| \bullet \|_\ast$ for the space
${\lebesgue^1 + \lebesgue^\infty + \lebesgue^{1,\infty} + \lebesgue^{\infty,1}}$,
such that with this norm,
${\lebesgue^1 + \lebesgue^\infty + \lebesgue^{1,\infty} + \lebesgue^{\infty,1}}$
is a so-called \emph{normed Köthe space}, whose defining \emph{function norm}
satisfies the \emph{Fatou property}.
Once this is established, the claim of Theorem~\ref{thm:DualOfIntersection}
is proven in Appendix~\ref{sub:DualOfIntersectionProof} as a consequence of
the \emph{Luxemburg representation theorem}.
All of these notions (Köthe spaces, function norms, etc) will be recalled below.

In the following, we will always use the convention $\infty - \lambda = \infty$
for $\lambda \in \R$.
Note that this implies $(\theta - \lambda) + \lambda = \theta = (\theta + \lambda) - \lambda$
for all $\theta \in [0,\infty]$ and $\lambda \in [0,\infty)$.

We begin by studying the norm that defines the space $\lebesgue^1 + \lebesgue^\infty$.

\begin{lemma}\label{lem:LebesgueSumFunctionNorm}
  Let $(X,\CalF,\mu)$ be a measure space.
  For any measurable function $f : X \to [0,\infty]$, define
  \[
    \varrho(f)
    := \inf \big\{
              \|g\|_{\lebesgue^\infty} + \|h\|_{\lebesgue^1}
              \colon
              g, h : X \to [0,\infty] \text{ measurable and } f = g + h
            \big\} \in [0,\infty]
  \]
  and $X_{f,\lambda} := \{ x \in X \colon f(x) > \lambda \}$, for $\lambda \in [0,\infty]$.
  Then the following properties hold:
  \begin{enumerate}
    \item We have
          \(
            \varrho(f)
            = \min_{\lambda \in [0,\infty)}
                \big[ \lambda + \|\Indicator_{X_{f,\lambda}} \cdot (f - \lambda)\|_{\lebesgue^1} \big]
          \).
          In particular, the minimum is attained.

    \item We have
          \(
            \varrho(f)
            = \inf_{\lambda \in [0,\infty) \cap \QQ}
                \big[ \lambda + \|\Indicator_{X_{f,\lambda}} \cdot (f - \lambda)\|_{\lebesgue^1} \big]
          \).
          \vspace{0.1cm}

    \item If $F : X \to \CC$ is measurable and $\alpha := \varrho(|F|) \in [0,\infty]$, then
          \[
            \big\| F \cdot (1-\Indicator_{X_{|F|,2\alpha}}) \big\|_{\lebesgue^\infty}
            \leq 2 \, \varrho(|F|)
            \quad \text{and} \quad
            \big\| F \cdot \Indicator_{X_{|F|,2\alpha}} \big\|_{\lebesgue^1}
            \leq 2 \, \varrho(|F|) .
          \]
  \end{enumerate}
\end{lemma}

\begin{proof}
  We start by showing that the auxiliary function
  \[
    \Psi_f :
    [0,\infty) \to [0,\infty],
    \lambda \mapsto \lambda + \|\Indicator_{X_{f,\lambda}} \cdot (f - \lambda) \|_{\lebesgue^1}
  \]
  is lower semicontinuous; this means that if
  $(\lambda_n)_{n \in \N} \subset [0,\infty)$ satisfies $\lambda_n \to \lambda \in [0,\infty)$,
  then $\Psi_f (\lambda) \leq \liminf_{n \to \infty} \Psi_f (\lambda_n)$.
  Since the identity map $\lambda \mapsto \lambda$ is continuous, it suffices to
  consider only the second summand in the definition of $\Psi_f (\lambda)$.
  If $x \in X_{f,\lambda}$, then $f(x) > \lambda_n$ for all $n \geq n_x$,
  for a suitable $n_x \in \N$.
  From this, we derive
  \[
    \Indicator_{X_{f,\lambda}} (x) \cdot (f(x) - \lambda)
    \leq \liminf_{n \to \infty}
           \big[ \Indicator_{X_{f,\lambda_n}} (x) \cdot (f(x) - \lambda_n) \big]
    \qquad \forall \, x \in X,
  \]
  where all involved functions are non-negative.
  Thus, an application of Fatou's lemma shows as claimed that
  \({
    \|\Indicator_{X_{f,\lambda}} \cdot (f - \lambda)\|_{\lebesgue^1}
    \leq \liminf_{n \to \infty}
           \|\Indicator_{X_{f,\lambda_n}} \cdot (f - \lambda_n)\|_{\lebesgue^1}
  }\).

  \medskip{}

  We now prove each claim individually.

  \medskip{}

  \noindent
  \textbf{Ad (1):}
  Define
  \(
    \varrho^\ast (f)
    := \inf_{\lambda \in [0,\infty)} \,
         [\lambda + \|\Indicator_{X_{f,\lambda}} \cdot (f - \lambda)\|_{\lebesgue^1}]
    =  \inf_{\lambda \geq 0}
         \Psi_f (\lambda)
  \).
  We first show ${\varrho (f) \leq \varrho^\ast (f)}$.
  Indeed, if $\lambda \in [0,\infty)$ is arbitrary, define
  \(
    g_\lambda
    := \min\{f, \lambda\}
    = f \cdot \Indicator_{f \leq \lambda} + \lambda \cdot \Indicator_{X_{f,\lambda}}
  \),
  and $h_\lambda := \Indicator_{X_{f,\lambda}} \cdot (f - \lambda)$.
  Then $g_\lambda, h_\lambda : X \to [0,\infty]$ are measurable and satisfy $f = g_\lambda + h_\lambda$.
  Therefore,
  \[
    \varrho(f)
    \leq \|g_\lambda\|_{\lebesgue^\infty} + \|h_\lambda\|_{\lebesgue^1}
    \leq \lambda + \|\Indicator_{X_{f,\lambda}} \cdot (f - \lambda)\|_{\lebesgue^1}.
  \]
  Since this holds for all $\lambda \in [0,\infty)$, we see $\varrho(f) \leq \varrho^\ast (f)$.

  In case of $\varrho(f) = \infty$, the preceding derivations imply
  \(
    \infty
    =    \varrho(f)
    \leq \varrho^\ast(f)
    \leq \lambda
         + \|
             \Indicator_{X_{f,\lambda}}
             \cdot (f - \lambda)
           \|_{\lebesgue^1}
  \)
  for \emph{every} $\lambda\in [0,\infty)$, which shows that the desired equality holds
  and that the infimum is attained.
  Therefore, let us assume $\varrho(f) < \infty$, and let $g,h : X \to [0,\infty]$ be measurable
  with $f = g + h$ and $\|g\|_{\lebesgue^\infty} + \|h\|_{\lebesgue^1} < \infty$.
  Define $\lambda := \|g\|_{\lebesgue^\infty}$.
  For $\mu$-almost every $x \in X$, we then have $|g(x)| \leq \lambda < \infty$, and hence
  $h (x) = f(x) - g(x) \geq f(x) - \lambda$.
  This implies $0 \leq \Indicator_{X_{f,\lambda}} \cdot (f - \lambda) \leq h$ $\mu$-almost everywhere,
  and hence
  \[
    \varrho^\ast (f)
    \leq \lambda + \big\| \Indicator_{X_{f,\lambda}} \cdot (f - \lambda) \big\|_{\lebesgue^1}
    \leq \|g\|_{\lebesgue^\infty} + \|h\|_{\lebesgue^1}.
  \]
  Since this holds for all admissible choices of $g,h$, we get $\varrho^\ast (f) \leq \varrho(f)$.
  It remains to show that the infimum in the definition of $\varrho^\ast (f)$ is attained.
  Choose a sequence
  $(\lambda_n)_{n \in \N} \subset [0,\infty)$ satisfying $\Psi_f (\lambda_n) \to \varrho^\ast (f)$.
  Note that $0 \leq \lambda_n \leq \Psi_f (\lambda_n) \to \varrho^\ast(f) \leq \varrho(f) < \infty$,
  so that the sequence $(\lambda_n)_{n \in \N}$ is bounded.
  Therefore, there is a subsequence $(\lambda_{n_\ell})_{\ell \in \N}$
  and some $\lambda \in [0,\infty)$ satisfying ${\lambda_{n_\ell} \to \lambda}$.
  By lower semicontinuity of $\Psi_f$, this implies
  \({
    \varrho^\ast (f)
    \leq \Psi_f (\lambda)
    \leq {\displaystyle{\liminf_{\ell \to \infty}}} \,
           \Psi_f(\lambda_{n_\ell})
    = \varrho^\ast(f)
    ,
  }\)
  proving that the infimum is attained.

  \medskip{}

  \noindent
  \textbf{Ad (2):}
  Let $\varrho^{\natural} (f) := \inf_{\lambda \in [0,\infty) \cap \QQ} \Psi_f (\lambda)$.
  By Part (1), we see $\varrho(f) = \varrho^\ast (f) \leq \varrho^{\natural}(f)$.

  To prove the converse estimate, we first show that $\Psi_f$ is right continuous.
  To see this, let $(\lambda_n)_{n \in \N} \subset [0,\infty)$ be a non-increasing sequence.
  Then $(f(x) - \lambda_n)_{n \in \N}$ is a non-decreasing sequence which converges pointwise to
  $f(x) - \lambda$, where $\lambda = \inf_{n \in \N} \lambda_n = \lim_{n \to \infty} \lambda_n$.
  Furthermore, we have
  $\Indicator_{X_{f,\lambda_n}} \leq \Indicator_{X_{f,\lambda_{n+1}}} \leq \Indicator_{X_{f,\lambda}}$.
  Finally, if $x \in X$ is such that $f(x) > \lambda$, then $f(x) > \lambda_n$ for all $n \geq n_x$
  (for a suitable $n_x \in \N$), proving that
  $\Indicator_{X_{f,\lambda_n}} \nearrow \Indicator_{X_{f,\lambda}}$ pointwise.
  Overall, we see
  \(
    0 \leq     \Indicator_{X_{f,\lambda_n}} \cdot (f - \lambda_n)
      \nearrow \Indicator_{X_{f,\lambda}} \cdot (f - \lambda)
  \),
  so that the monotone convergence theorem shows
  \(
    \|\Indicator_{X_{f,\lambda_n}} \cdot (f - \lambda_n)\|_{\lebesgue^1}
    \to \|\Indicator_{X_{f,\lambda}} \cdot (f - \lambda)\|_{\lebesgue^1}
  \).
  In view of this, we easily see that $\Psi_f$ is right continuous.

  Now, let $\lambda \in [0,\infty)$.
  There is a non-increasing sequence $(\lambda_n)_{n \in \N} \subset [0,\infty) \cap \QQ$
  such that $\lambda_n \to \lambda$.
  Note that $\varrho^\natural (f) \leq \Psi_f (\lambda_n)$ for all $n \in \N$,
  and hence $\varrho^{\natural} (f) \leq \lim_{n \to \infty} \Psi_f(\lambda_n) = \Psi_f (\lambda)$.
  Since $\lambda \in [0,\infty)$ was arbitrary, this implies
  \(
    \varrho^{\natural} (f)
    \leq \inf_{\lambda \in [0,\infty)}
           \Psi_f(\lambda)
    = \varrho^\ast (f)
    = \varrho (f)
    .
  \)

  \medskip{}

  \noindent
  \textbf{Ad (3)}
  The claim is trivial in case of $\alpha = \infty$, so that we can assume $\alpha < \infty$.
  Part~(1) shows that there is some $\lambda \in [0,\infty)$ satisfying
  $\alpha = \lambda + \|\Indicator_{X_{|F|,\lambda}} \cdot (|F| - \lambda)\|_{\lebesgue^1}$.
  In particular, this implies that $2 \alpha \geq \alpha \geq \lambda$.
  Next, note that if $|F(x)| > 2 \alpha$,
  then $|F(x)| \leq 2 \cdot (|F(x)| - \alpha) \leq 2 \cdot (|F(x)| - \lambda)$.
  Therefore,
  \[
    \big\| F \cdot \Indicator_{X_{|F|,2\alpha}} \big\|_{\lebesgue^1}
    \leq 2 \cdot \big\| \Indicator_{X_{|F|,2\alpha}} \cdot (|F| - \lambda) \big\|_{\lebesgue^1}
    \leq 2 \cdot \big\| \Indicator_{X_{|F|,\lambda}} \cdot (|F| - \lambda) \big\|_{\lebesgue^1}
    \leq 2 \alpha .
  \]
  The estimate $\|F \cdot (1-\Indicator_{X_{|F|,2\alpha}})\|_{\lebesgue^\infty} \leq 2 \alpha$
  is trivial.
\end{proof}

As a first corollary of the preceding lemma, we can now show that $\varrho$ is a
so-called \emph{function-norm} which satisfies the \emph{Fatou property}.
Before we prove this, we recall the pertinent definitions for the convenience of the reader.

\begin{definition}\label{def:FunctionNorm}(see \cite[§§ 63 and 65]{ZaanenIntegration})

  Let $(X,\CalF,\mu)$ be a $\sigma$-finite measure space.
  We denote by $\CalM^+$ the set of all equivalence classes of measurable functions $f : X \to [0,\infty]$,
  where two functions are equivalent if they agree $\mu$-almost everywhere.
  By the usual abuse of notation, we will often not distinguish between a function and its equivalence class.

  A map $\varrho : \CalM^+ \to [0,\infty]$ is called a \emph{function seminorm} if it satisfies
  the following properties:
  \begin{enumerate}
    \item $\varrho(f) = 0$ if $f \in \CalM^+$ with $f = 0$ almost everywhere;

    \item $\varrho(a \, f) = a \, \varrho(f)$ for all $f \in \CalM^+$ and $a \in [0,\infty)$;

    \item $\varrho (f + g) \leq \varrho(f) + \varrho(g)$ for all $f,g \in \CalM^+$;

    \item if $f,g \in \CalM^+$ satisfy $f \leq g$ almost everywhere, then $\varrho(f) \leq \varrho(g)$.
  \end{enumerate}
  A function seminorm is called a \emph{function norm} if it has the additional property that
  $f = 0$ almost everywhere for every $f \in \CalM^+$ with $\varrho(f) = 0$.

  A function seminorm $\varrho$ is said to have the \emph{Fatou property} if
  for every sequence $(f_n)_{n \in \N} \subset \CalM^+$ with $\liminf_{n \to \infty} \varrho(f_n) < \infty$,
  we have $\varrho(\liminf_{n \to \infty} f_n) \leq \liminf_{n \to \infty} \varrho(f_n)$.
\end{definition}

\begin{rem*}
  The definition of the Fatou property given above is not the one given in \cite{ZaanenIntegration},
  but it is equivalent, as shown in \cite[§65, Theorem 3]{ZaanenIntegration}.
\end{rem*}

\begin{proposition}\label{prop:LebesgueSumFunctionNormFatouProperty}
  Let $(X,\CalF,\mu)$ be a $\sigma$-finite measure space.
  The map $\varrho : \CalM^+ \to [0,\infty]$ introduced in Lemma~\ref{lem:LebesgueSumFunctionNorm}
  is a function norm which satisfies the Fatou property.
\end{proposition}

\begin{proof}
  The first two properties in Definition~\ref{def:FunctionNorm} are trivially satisfied.

  Next, if $f, g \in \CalM^+$ with $f \leq g$, then
  \(
    0
    \leq \Indicator_{X_{f,\lambda}} \cdot (f - \lambda)
    \leq \Indicator_{X_{g,\lambda}} \cdot (g - \lambda)
    ,
  \)
  and hence
  \(
    \lambda + \|\Indicator_{X_{f,\lambda}} \cdot (f - \lambda)\|_{\lebesgue^1}
    \leq \lambda + \|\Indicator_{X_{g,\lambda}} \cdot (g - \lambda) \|_{\lebesgue^{1}}
  \)
  for all $\lambda \in [0,\infty)$.
  In view of the first part of Lemma~\ref{lem:LebesgueSumFunctionNorm}, this implies
  $\varrho(f) \leq \varrho(g)$.

  Furthermore, if $f,g \in \CalM^+$ and $f = f_1 + f_2$ as well as $g = g_1 + g_2$
  for measurable functions $f_1, f_2, g_1, g_2 : X \to [0,\infty]$,
  then $f + g = (f_1 + g_1) + (f_2 + g_2)$.
  By definition of $\varrho$, this implies
  \[
    \varrho(f+g)
    \leq \|f_1 + g_1\|_{\lebesgue^\infty} + \|f_2 + g_2\|_{\lebesgue^1}
    \leq \big( \|f_1\|_{\lebesgue^\infty} + \|f_2\|_{\lebesgue^1} \big)
      +  \big( \|g_1\|_{\lebesgue^\infty} + \|g_2\|_{\lebesgue^1} \big).
  \]
  Since this holds for all admissible $f_1,f_2,g_1,g_2$, we get by definition of $\varrho$
  that $\varrho(f+g) \leq \varrho(f) + \varrho(g)$.

  To verify that $\varrho$ is a function norm, let $f \in \CalM^+$ satisfy $\varrho(f) = 0$.
  By the first part of Lemma~\ref{lem:LebesgueSumFunctionNorm}, there is $\lambda \in [0,\infty)$
  such that $0 = \varrho(f) = \lambda + \|\Indicator_{X_{f,\lambda}} \cdot (f - \lambda)\|_{\lebesgue^1}$.
  This implies $\lambda = 0$, and then $0 = \|f\|_{\lebesgue^1}$,
  so that we see $f = 0$ almost everywhere.

  \medskip{}

  Finally, we verify the Fatou property.
  Let $(f_n)_{n \in \N} \subset \CalM^+$ with $\theta := \liminf_{n \to \infty} \varrho(f_n) < \infty$.
  Set $f := \liminf_{n \to \infty} f_n$.
  Choose a subsequence $(f_{n_k})_{k \in \N}$ such that $\varrho(f_{n_k}) \to \theta$.
  The first part of Lemma~\ref{lem:LebesgueSumFunctionNorm} yields a sequence
  $(\lambda_k)_{k \in \N} \subset [0,\infty)$ satisfying
  \(
    \varrho(f_{n_k})
    = \lambda_k + \|\Indicator_{f_{n_k} > \lambda_k} \cdot (f_{n_k} - \lambda_k)\|_{\lebesgue^1}
  \)
  for all $k \in \N$.
  In particular, $0 \leq \lambda_k \leq \varrho(f_{n_k}) \to \theta$, so that $(\lambda_k)_{k \in \N}$
  is a bounded sequence.
  Thus, there is a subsequence $(\lambda_{k_\ell})_{\ell \in \N}$
  and some $\lambda \in [0,\infty)$ satisfying $\lambda_{k_\ell} \to \lambda$.
  \vspace{0.1cm}

  Note $f \leq \liminf_{\ell \to \infty} f_{n_{k_\ell}}$, and hence
  \(
    f - \lambda
    \leq \liminf_{\ell \to \infty} ( f_{n_{k_\ell}} - \lambda_{k_\ell} )
    = {\raisebox{0.1cm}{$\displaystyle{\sup_{L \in \N}}$} \,\,
       \raisebox{0.05cm}{$\displaystyle{\inf_{\ell \geq L}}$}}
        \, (f_{n_{k_{\ell}}} - \lambda_{k_\ell})
  \).
  Thus, for each $x \in X$ satisfying $f(x) - \lambda > 0$, there is $L_x \in \N$
  such that $f_{n_{k_\ell}} (x) - \lambda_{k_\ell} \geq \frac{f(x) - \lambda}{2} > 0$
  for all $\ell \geq L_x$.
  Overall, we see
  \({
    \Indicator_{X_{f,\lambda}} (x) \cdot (f(x) - \lambda)
    \leq \liminf_{\ell \to \infty}
           \Big[
             \Indicator_{f_{n_{k_\ell}} > \lambda_{k_\ell}} (x)
             \cdot \big( f_{n_{k_\ell}} (x) -\! \lambda_{k_\ell} \big) \!
           \Big] ,
  }\)
  where all involved functions are non-negative.
  Therefore, Fatou's lemma and Part~(1) of Lemma~\ref{lem:LebesgueSumFunctionNorm} imply that
  \begin{align*}
    \varrho(f)
    & \leq \lambda + \| (f - \lambda) \cdot \Indicator_{X_{f,\lambda}}\|_{\lebesgue^1}
      \leq \lim_{\ell \to \infty}
             \lambda_{k_\ell}
           + \Big\|
               \liminf_{\ell \to \infty}
                 \big[
                   (f_{n_{k_\ell}} - \lambda_{k_\ell})
                   \cdot \Indicator_{f_{n_{k_\ell}} > \lambda_{k_\ell}}
                 \big]
             \Big\|_{\lebesgue^1} \\
    & \leq \liminf_{\ell \to \infty}
             \Big(
               \lambda_{k_\ell}
               + \big\|
                   (f_{n_{k_\ell}} - \lambda_{k_\ell})
                   \cdot \Indicator_{f_{n_{k_\ell}} > \lambda_{k_\ell}}
                 \big\|_{\lebesgue^1}
             \Big)
      =     \lim_{\ell \to \infty}
              \varrho(f_{n_{k_\ell}})
      =     \theta
      =     \liminf_{n \to \infty} \varrho(f_n).
    \qedhere
  \end{align*}
\end{proof}

Up to now, we have shown that the norm defining the space $\lebesgue^1 + \lebesgue^\infty$
is a function norm satisfying the Fatou property.
Our next goal is to construct a function-norm $\varrho_{\otimes}$ satisfying the Fatou property
and such that $F \mapsto \varrho_{\otimes}(|F|)$ is equivalent to the norm defining the space ${\SumSpace}$.
For this, the following property will be crucial.

\begin{lemma}\label{lem:LebesgueSumNormMeasurable}
  Let $(X, \CalF, \mu)$ be a $\sigma$-finite measure space, and let $(Y,\CalG)$ be a measurable space.
  Let $\varrho$ be as defined in Lemma~\ref{lem:LebesgueSumFunctionNorm}.

  If $F : X \times Y \to [0,\infty]$ is measurable with respect to the product $\sigma$-algebra
  $\CalF \otimes \CalG$, then the map
  \[
    Y \to [0,\infty], y \mapsto \varrho \big(F (\bullet, y) \big)
  \]
  is measurable.
\end{lemma}

\begin{proof}
  For each $\lambda \in [0,\infty)$, the map
  \[
    Y \to [0,\infty],
    y \mapsto \Big\|
                \Indicator_{F (\bullet,y) > \lambda} \cdot \big( F(\bullet,y) - \lambda \big)
              \Big\|_{\lebesgue^1}
              = \int_X
                  \Indicator_{F(x,y) > \lambda} \cdot \bigl( F(x,y) - \lambda \bigr)
                \, d \mu (x)
  \]
  is measurable as a consequence of the Fubini-Tonelli theorem.
  Now, the formula
  \[
    \varrho \big( F(\bullet,y) \big)
    = \inf_{\lambda \in [0,\infty) \cap \QQ}
        \Big[
          \lambda
          + \big\|
              \Indicator_{F(\bullet,y) > \lambda}
              \cdot \big( F(\bullet,y) - \lambda \big)
            \big\|_{\lebesgue^1}
        \Big]
  \]
  given in the second part of Lemma~\ref{lem:LebesgueSumFunctionNorm} shows that
  $y \mapsto \varrho \big( F(\bullet,y) \big)$ is measurable as the infimum of a countable family of
  non-negative measurable functions.
\end{proof}

In view of Lemma~\ref{lem:LebesgueSumNormMeasurable}, we see that
if $(X, \CalF, \mu)$ and $(Y, \CalG, \nu)$ are $\sigma$-finite measure spaces,
and if $F : X \times Y \to [0,\infty]$ is measurable, then the map
$y \mapsto \varrho \big( F(\bullet,y) \big)$ is a measurable non-negative function
to which we can again apply $\varrho$, which is as defined in Lemma~\ref{lem:LebesgueSumFunctionNorm},
but now on $Y$ instead of on $X$.
Thus, the following definition makes sense.

\begin{definition}\label{def:IteratedLebesgueSumNorm}
  Let $(X, \CalF, \mu)$ and $(Y, \CalG, \nu)$ be $\sigma$-finite measure spaces.
  Let $\varrho_X$ and $\varrho_Y$ be as defined in Lemma~\ref{lem:LebesgueSumFunctionNorm},
  applied to the measure spaces $(X,\CalF,\mu)$ or $(Y,\CalG,\nu)$, respectively.

  For every measurable function $F : X \times Y \to [0,\infty]$, define
  \[
    \varrho_{\otimes} (F)
    := \varrho_Y \Big( y \mapsto \varrho_X \big( F(\bullet,y) \big) \Big)
    \in [0,\infty].
  \]
\end{definition}

\begin{lemma}\label{lem:IteratedLebesgueSumNormIsFunctionNorm}
  Let $(X, \CalF, \mu)$ and $(Y, \CalG, \nu)$ be $\sigma$-finite measure spaces.
  The map $\varrho_{\otimes}$ introduced in Definition~\ref{def:IteratedLebesgueSumNorm}
  is a function norm on $(X \times Y, \CalF \otimes \CalG, \mu \otimes \nu)$
  which satisfies the Fatou property.
\end{lemma}

\begin{proof}
  The first two properties in Definition~\ref{def:FunctionNorm} are clear.
  Next, if ${F,G : X \times Y \to [0,\infty]}$ are measurable, then
  \(
    \varrho_X \big( [F+G](\bullet,y) \big)
    \leq \varrho_X \big( F(\bullet,y) \big) + \varrho_X \big( G(\bullet,y) \big)
  \)
  for all $y \in Y$, since $\varrho_X$ is a function norm.
  By the monotonicity and subadditivity of $\varrho_Y$, this entails
  \begin{align*}
    \varrho_{\otimes} (F + G)
    & = \varrho_Y \Big(  y  \mapsto  \varrho_X \big( [F+G](\bullet,y) \big) \Big) \\
    & \leq \varrho_Y \Big( y \mapsto \varrho_X \big( F(\bullet,y) \big) \Big)
           + \, \varrho_Y \Big( y \mapsto \varrho_X \big( G(\bullet,y) \big) \Big)
      = \varrho_{\otimes}(F) + \varrho_{\otimes} (G),
  \end{align*}
  as desired.
  The monotonicity of $\varrho_{\otimes}$ follows easily from that of $\varrho_X$ and $\varrho_Y$.

  Next, if $\varrho_{\otimes} (F) = 0$, then since $\varrho_Y$ is a function-norm,
  there is a $\nu$-null-set $N \subset Y$ such that $\varrho_X (F(\bullet,y)) = 0$
  for all $y \in Y \setminus N$.
  Since $\varrho_X$ is a function-norm, this implies $F(\bullet,y) = 0$ $\mu$-almost everywhere
  for $y \in Y \setminus N$.
  Hence, Tonelli's theorem shows
  $\|F\|_{\lebesgue^1} = \int_Y \int_X F(x,y) d \mu(x) d \nu (y) = 0$,
  and hence $F = 0$ almost everywhere with respect to $\mu \otimes \nu$.

  It remains to verify the Fatou property.
  If $(F_n)_{n \in \N}$ is a sequence of measurable functions $F_n : X \times Y \to [0,\infty]$
  and $F = \liminf_{n \to \infty} F_n$, then $F(\bullet,y) = \liminf_{n \to \infty} F_n (\bullet,y)$,
  so that the Fatou property for $\varrho_X$ implies
  $\varrho_X \big( F(\bullet,y) \big) \leq \liminf_{n \to \infty} \varrho_X \big( F_n (\bullet, y) \big)$
  for all $y \in Y$.
  Strictly speaking, this only follows from the Fatou property if the right-hand side is finite;
  but otherwise the estimate is trivially satisfied.
  By the monotonicity and the Fatou property of $\varrho_Y$, we thus see
  \begin{align*}
    \varrho_\otimes (F)
    & = \varrho_Y \Big( y \mapsto \varrho_X \big( F(\bullet,y) \big) \Big)
      \leq \varrho_Y
           \Big(
             y \mapsto \liminf_{n \to \infty} \varrho_X \big( F_n (\bullet,y) \big)
           \Big) \\
    & \leq \liminf_{n \to \infty}
             \varrho_Y \Big( y \mapsto \varrho_X \big( F_n (\bullet,y) \big) \Big)
      =    \liminf_{n \to \infty}
             \varrho_\otimes (F_n).
    \qedhere
  \end{align*}
\end{proof}

The following proposition shows that the norm $F \mapsto \varrho_{\otimes}(|F|)$ is equivalent
to the defining norm of the space $\SumSpace$.

\begin{proposition}\label{prop:IteratedFunctionNormEquivalentToSumNorm}
  Let $(X, \CalF, \mu)$ and $(Y, \CalG, \nu)$ be $\sigma$-finite measure spaces.
  With $\varrho_{\otimes}$ as in Definition~\ref{def:IteratedLebesgueSumNorm} and
  $\| \bullet \|_{\SumSpace}$ as introduced in Definition~\ref{def:SumAndIntersectionSpaces},
  we then have
  \[
    \frac{1}{16} \cdot \|F\|_{\SumSpace}
    \leq \varrho_{\otimes} (|F|)
    \leq \|F\|_{\SumSpace}
  \]
  for each measurable $F : X \times Y \to \CC$.
\end{proposition}

\begin{rem*}
  In the terminology of solid function spaces, the proposition shows
  (in combination with Lemma~\ref{lem:IteratedLebesgueSumNormIsFunctionNorm}) that the space
  ${\SumSpace}$ with its canonical norm satisfies the \emph{weak Fatou property};
  see \mbox{\cite[§ 65]{ZaanenIntegration}} for the definition.
  Here, we would like to remark that there is a characterization of (possibly infinite) families
  $(X_i)_{i \in I}$ of solid Banach spaces for which the sum $\sum_{i \in I} X_i$
  with its natural norm satisfies the \emph{Fatou property};
  see \cite{BanachLatticeSumsFatouProperty}.
  This characterization, however, is quite technical, and we were unable to verify it in our setting.
  Thus, although we know that $\SumSpace$ with its natural norm satisfies the \emph{weak} Fatou
  property, we could not confirm whether it actually satisfies the Fatou property as well.
\end{rem*}

\begin{proof}
  We first show $\varrho_{\otimes} (|F|) \leq \|F\|_{\SumSpace}$,
  which is trivial if the right-hand side is infinite.
  Thus, we can assume $\|F\|_{\SumSpace} < \infty$.
  Let $F_1,\dots,F_4 : X \times Y \to \CC$ be measurable with $F = F_1 + \cdots + F_4$ and such that
  \({
    \|F_1\|_{\lebesgue^1}
    + \|F_2\|_{\lebesgue^\infty}
    + \|F_3\|_{\lebesgue^{1,\infty}}
    + \|F_4\|_{\lebesgue^{\infty,1}}
    < \infty.
  }\)
  With $F_y^{(1)} := F_2 (\bullet,y) + F_4 (\bullet,y)$ and
  $F_y^{(2)}:= F_1 (\bullet,y) + F_3(\bullet,y)$, we have
  \({
    F(\bullet,y)
    = F_y^{(1)} + F_y^{(2)}
    ,
  }\)
  and therefore $\big| F(\bullet, y) \big| \leq \big| F_y^{(1)} \big| + \big| F_y^{(2)} \big|$.
  Using the monotonicity and the definition of $\varrho_X$ (see Lemma~\ref{lem:LebesgueSumFunctionNorm}),
  this implies for each $y \in Y$ that
  \begin{align*}
    \varrho_X \big( |F(\bullet,y)| \big)
    & \leq \varrho_X \big( \, \big| F_y^{(1)} \big| +  \big| F_y^{(2)} \big| \, \big)
      \leq \big\| F_y^{(1)} \big\|_{\lebesgue^\infty}
           + \big\| F_y^{(2)} \big\|_{\lebesgue^1} \\
    & \leq \big(\underbrace{
             \| F_2(\bullet, y) \|_{\lebesgue^\infty}
             + \| F_3 (\bullet, y) \|_{\lebesgue^1}
           }_{\textstyle =:G_1 (y)}\big)
           + \big(\underbrace{
               \| F_1(\bullet, y) \|_{\lebesgue^1}
               + \| F_4 (\bullet, y) \|_{\lebesgue^{\infty}}
             }_{\textstyle =:G_2 (y)}\big).
  \end{align*}
  By using the monotonicity of $\varrho_Y$ and the definition of $\varrho_Y$, we finally arrive at
  \begin{align*}
    \varrho_{\otimes} (|F|)
    & = \varrho_Y \Big( y \mapsto \varrho_X \big( |F(\bullet,y)| \big) \Big)
      \leq \varrho_Y \big( G_1 + G_2 \big) \\
    & \leq \|G_1\|_{\lebesgue^\infty} + \|G_2\|_{\lebesgue^1}
      \leq \|F_2\|_{\lebesgue^\infty} + \|F_3\|_{\lebesgue^{1,\infty}}
           + \|F_1\|_{\lebesgue^1} + \|F_4\|_{\lebesgue^{\infty,1}} .
  \end{align*}
  Since this holds for all admissible $F_1,\dots,F_4$, we see
  $\varrho_{\otimes} (|F|) \leq \|F\|_{\SumSpace}$.

  \medskip{}

  We now prove $\|F\|_{\SumSpace} \leq 16 \, \varrho_{\otimes} (|F|)$.
  In case of $\varrho_{\otimes}(|F|) = \infty$ this is trivial, so that we can assume
  $\alpha := \varrho_{\otimes}(|F|) \in [0,\infty)$.
  Define $G : Y \to [0,\infty], y \mapsto \varrho_X \big( |F(\bullet,y)| \big)$
  and note $\alpha = \varrho_Y (G)$.
  Note that $G$ is measurable by Lemma~\ref{lem:LebesgueSumNormMeasurable}.
  Let
  \[
    A := \big\{ y \in Y \colon G(y) > 2 \alpha \big\}
    \quad \text{and} \quad
    B := \big\{ (x,y) \in X \times Y \colon |F(x,y)| > 2 G(y) \big\}.
  \]
  Finally, define $F_1,\dots,F_4 : X \times Y \to \CC$ by
  \[
    F_1 (x,y) := F(x,y) \cdot \Indicator_{A}(y) \cdot \Indicator_{B}(x,y),
    \qquad \quad
    F_2 (x,y) := F(x,y) \cdot \Indicator_{A^c}(y) \cdot \Indicator_{B^c}(x,y)
  \]
  and
  \[
    F_3 (x,y) := F(x,y) \cdot \Indicator_{A^c}(y) \cdot \Indicator_{B}(x,y),
    \qquad \quad
    F_4 (x,y) := F(x,y) \cdot \Indicator_{A}(y) \cdot \Indicator_{B^c}(x,y).
  \]
  We clearly have $F = F_1 + \cdots + F_4$.
  Furthermore, by the last part of Lemma~\ref{lem:LebesgueSumFunctionNorm},
  and since $G(y) = \varrho_X \big( |F(\bullet,y)| \big)$ and $\alpha = \varrho_Y (G)$,
  we see the following:
  \begin{enumerate}
    \item \(
            \|F_1 (\bullet,y)\|_{\lebesgue^1}
            = \Indicator_A (y)
              \cdot \big\|
                      F(\bullet,y) \cdot \Indicator_{|F(\bullet,y)| > 2 G(y)}
                    \big\|_{\lebesgue^1}
            \leq \Indicator_A (y) \cdot 2 \varrho_X \big( |F(\bullet,y)| \big)
            =    \Indicator_A (y) \cdot 2 G(y)
          \),
          and hence
          \(
            \|F_1\|_{\lebesgue^1(\mu \otimes \nu)}
            \leq 2 \| G \cdot \Indicator_A \|_{\lebesgue^1 (\nu)}
            =    2 \| G \cdot \Indicator_{G > 2 \alpha} \|_{\lebesgue^1(\nu)}
            \leq 4 \varrho_Y (G)
            =    4 \, \varrho_{\otimes} (|F|)
          \).

    \item \(
            |F_2 (x,y)|
            \leq 2 G(y) \, \Indicator_{A^c} (y)
            \leq 4 \, \alpha
            = 4 \, \varrho_{\otimes} (|F|)
          \)
          for all $(x,y) \in X \times Y$,
          which shows that $\|F_2\|_{\lebesgue^\infty} \leq 4 \, \varrho_{\otimes} (|F|)$.

    \item Similar to the estimate for $\|F_1\|_{\lebesgue^1}$, we see that
          \[
            \quad \qquad
            \|F_3 (\bullet,y)\|_{\lebesgue^1}
            \leq \Indicator_{A^c}(y)
                 \cdot \| F(\bullet,y) \cdot \Indicator_{|F(\bullet,y)| > 2 G(y)} \|_{\lebesgue^1}
            \leq \Indicator_{A^c}(y) \cdot 2 G(y)
            \leq 4 \, \alpha
            =    4 \, \varrho_{\otimes} (|F|)
          \]
          for all $y \in Y$, and hence
          $\|F_3\|_{\lebesgue^{1,\infty}} \leq 4 \, \varrho_{\otimes}(|F|)$.

    \item $|F_4(x,y)| \leq 2 G(y) \cdot \Indicator_A (y)$ for all $x \in X$ and $y \in Y$.
          From this estimate, it follows that
          \(
            \|F_4\|_{\lebesgue^{\infty,1}}
            \leq 2 \, \| G \cdot \Indicator_{A}\|_{\lebesgue^1}
            \leq 4 \, \varrho_{\otimes} (|F|)
          \).
          Here, the last step was justified in the estimate of $\|F_1\|_{\lebesgue^1}$ above.
  \end{enumerate}
  Overall, we see
  \(
    \|F\|_{\SumSpace}
    \leq \|F_1\|_{\lebesgue^1}
         + \|F_2\|_{\lebesgue^\infty}
         + \|F_3\|_{\lebesgue^{1,\infty}}
         + \|F_4\|_{\lebesgue^{\infty,1}}
    \leq 16 \, \varrho_{\otimes} (|F|)
  \),
  which completes the proof.
\end{proof}

\subsection{Proof of Lemma \ref{lem:CartesianProductSumNorm}}
\label{sub:CartesianProductSumNormProof}

We will now use Lemma~\ref{lem:LebesgueSumFunctionNorm} and
Proposition~\ref{prop:IteratedFunctionNormEquivalentToSumNorm} to prove
Lemma~\ref{lem:CartesianProductSumNorm}.

\begin{proof}[Proof of Lemma~\ref{lem:CartesianProductSumNorm}]
  With the function norm $\varrho_{\otimes}$ introduced in Definition~\ref{def:IteratedLebesgueSumNorm},
  Proposition~\ref{prop:IteratedFunctionNormEquivalentToSumNorm} shows that
  $\| \Indicator_{V \times W} \|_{\SumSpace} \geq \varrho_{\otimes} (\Indicator_{V \times W})$.
  Next, with $\varrho$ as in Lemma~\ref{lem:LebesgueSumFunctionNorm},
  and writing $\varrho_{X_1}$ and $\varrho_{X_2}$ to indicate the space on which $\varrho$ acts,
  a direct computation shows that
  \[
    \varrho_{\otimes} (\Indicator_{V \times W})
    = \varrho_{X_2} \Big(
                      y \mapsto \varrho_{X_1} \big( \Indicator_{V \times W} (\bullet, y) \big)
                    \Big)
    = \varrho_{X_2} \Big(
                      y \mapsto \Indicator_W (y) \cdot \varrho_{X_1} (\Indicator_V)
                    \Big)
    = \varrho_{X_1} (\Indicator_V) \cdot \varrho_{X_2} (\Indicator_W) .
  \]
  To complete the proof, it therefore suffices to show that
  $\varrho_{X_1} (\Indicator_V) \geq \min\{ 1, \mu_1(V) \}$
  and likewise $\varrho_{X_2} (\Indicator_W) \geq \min\{ 1, \mu_2(W) \}$,
  since this implies the claimed estimate
  \begin{align*}
    \| \Indicator_{V \times W} \|_{\SumSpace}
    & \geq \varrho_{\otimes} (\Indicator_{V \times W})
      =    \varrho_{X_1} (\Indicator_V) \cdot \varrho_{X_2} (\Indicator_W) \\
    & \geq \min \big\{ 1, \mu_1 (V) \big\} \cdot \min \big\{ 1, \mu_2 (W) \big\} \\
    & \geq \min \big\{ 1, \mu_1 (V), \mu_2 (W), \mu_1 (V) \cdot \mu_2 (W) \big\} \\
    & =    \min \big\{ 1, \mu_1 (V), \mu_2 (W), \mu(V \times W) \big\} .
  \end{align*}

  We only prove that $\varrho_{X_1} (\Indicator_V) \geq \min\{ 1, \mu_1(V) \}$,
  since $\varrho_{X_2} (\Indicator_W) \geq \min\{ 1, \mu_2(W) \}$
  can be shown with the same arguments.
  Let $V_{\lambda} := \{x_1\in X_1 \colon \Indicator_V(x_1) > \lambda \}$, for $\lambda\geq 0$.
  Recall from Lemma~\ref{lem:LebesgueSumFunctionNorm} that
  \[
    \varrho_{X_1} (\Indicator_V)
    = \inf_{\lambda \in [0,\infty)}
        \Big[
          \lambda
          + \|
              \Indicator_{V_{\lambda}} \cdot (\Indicator_V - \lambda)
            \|_{\lebesgue^1}
        \Big] .
  \]
  Now, in case of $\lambda \geq 1$, we trivially have
  \(
    \lambda
    + \|
        \Indicator_{V_{\lambda}} \cdot (\Indicator_V - \lambda)
      \|_{\lebesgue^1}
    \geq \lambda
    \geq 1
    \geq \min \{ 1, \mu_1 (V) \}.
  \)
  Finally, if $0 \leq \lambda < 1$
  then $\Indicator_{V_{\lambda}} = \Indicator_V$,
  and hence
  \(
    \Indicator_{V_{\lambda}} \cdot (\Indicator_V - \lambda)
    = \Indicator_V \cdot (\Indicator_V - \lambda)
    = (1 - \lambda) \cdot \Indicator_V ,
  \)
  which implies that
  \[
    \lambda
    + \big\|
        \Indicator_{V_{\lambda}} \cdot (\Indicator_V - \lambda)
      \big\|_{\lebesgue^1}
    = \lambda + (1 - \lambda) \cdot \big\| \Indicator_V \big\|_{\lebesgue^1}
    = \lambda + (1 - \lambda) \cdot \mu_1 (V)
    \geq \min \{ 1, \mu_1 (V) \}.
  \]
  In combination, the two cases show that indeed
  $\varrho_{X_1}(\Indicator_V) \geq \min\{ 1, \mu_1 (V) \}$.
\end{proof}

\subsection{Proving Theorem~\ref{thm:DualOfIntersection} using the Lorentz-Luxemburg representation theorem}
\label{sub:DualOfIntersectionProof}

We now proceed with the proof of Theorem~\ref{thm:DualOfIntersection},
based on the Lorentz-Luxemburg representation theorem.
To that end, we first recall the concept of associate function seminorms.

\begin{definition}\label{def:AssociateNorm}(see \cite[§ 68]{ZaanenIntegration})
  Let $(X,\CalF,\mu)$ be a $\sigma$-finite measure space, and let $\varrho$ be a function
  seminorm on $X$.
  The \emph{associate function seminorm} $\varrho'$ of $\varrho$ is defined as
  \[
    \varrho ' (f)
    := \sup \Big\{
              \int_X f \cdot g \, d \mu
              \quad \colon \quad
              g : X \to [0,\infty] \text{ measurable and } \varrho(g) \leq 1
            \Big\}
    \in [0, \infty]
  \]
  for $f : X \to [0,\infty]$ measurable.
\end{definition}

As shown in \cite[Theorem 1 in § 68]{ZaanenIntegration}, $\varrho'$ is always
a function seminorm which satisfies the Fatou property.
We now compute---up to a constant factor---the associated seminorm $\varrho_{\otimes}'$
of the function norm $\varrho_{\otimes}$.

\begin{lemma}\label{lem:LebesgueSumNormAssociateNorm}
  Let $(X,\CalF,\mu)$ and $(Y,\CalG,\nu)$ be $\sigma$-finite measure spaces,
  and let $\varrho_{\otimes}$ as in Definition~\ref{def:IteratedLebesgueSumNorm}.
  For $F : X \times Y \to [0,\infty]$ measurable, let us write
  \[
    \|F\|_{\IntersectionSpace}
    := \IntersectionNorm{F}.
  \]
  Then, the associate seminorm $\varrho_{\otimes}'$ satisfies
  \[
    \|F\|_{\IntersectionSpace}
    \leq \varrho_{\otimes}' (F)
    \leq 16 \cdot \|F\|_{\IntersectionSpace}
  \]
  for any measurable function $F : X \times Y \to [0,\infty]$.
  In particular, $\varrho_{\otimes}'$ is a function norm, not just a function seminorm.
\end{lemma}

\begin{proof}
  We first prove the right-hand estimate.
  If the right-hand side is infinite, this estimate is trivial;
  hence, we can assume that ${\theta := \|F\|_{\IntersectionSpace} < \infty}$.
  Let $G : X \times Y \to [0,\infty]$ be measurable with $\varrho_{\otimes} (G) \leq 1$.
  Since $\varrho_{\otimes}$ is a function norm
  (see Lemma~\ref{lem:IteratedLebesgueSumNormIsFunctionNorm}), this implies that
  ${G < \infty}$ almost everywhere; see \cite[Theorem~1 in § 63]{ZaanenIntegration}.
  Hence, we can assume ${G < \infty}$ everywhere.
  Let $\eps > 0$.
  By Proposition~\ref{prop:IteratedFunctionNormEquivalentToSumNorm},
  there exist measurable functions ${G_1,\dots,G_4 : X \times Y \to \CC}$ such that
  $G = G_1 + \dots + G_4$ with
  \(
    \|G_1\|_{\lebesgue^1}
    + \|G_2\|_{\lebesgue^\infty}
    + \|G_3\|_{\lebesgue^{1,\infty}}
    + \|G_4\|_{\lebesgue^{\infty,1}}
    \leq 16 + \eps
  \).
  Therefore, Hölder's inequality for the mixed-norm Lebesgue spaces
  (see \cite[Equation~(1) in Section~2]{MixedLpSpaces}) shows that
  \begin{align*}
    \int_{X \times Y} \!\!\!
      F \!\cdot\! G
    \, d (\mu \otimes \nu)
    & \leq \sum_{j=1}^4 \int_{X \times Y} F \cdot |G_j| \, d (\mu \otimes \nu) \\
    & \leq \|F\|_{\lebesgue^\infty}       \, \|G_1\|_{\lebesgue^1}
           + \|F\|_{\lebesgue^1}          \, \|G_2\|_{\lebesgue^\infty}
           + \|F\|_{\lebesgue^{\infty,1}} \, \|G_3\|_{\lebesgue^{1,\infty}}
           + \|F\|_{\lebesgue^{1,\infty}} \, \|G_4\|_{\lebesgue^{\infty,1}} \\
    & \leq \theta \cdot (\|G_1\|_{\lebesgue^1}
                         + \|G_2\|_{\lebesgue^\infty}
                         + \|G_3\|_{\lebesgue^{1,\infty}}
                         + \|G_4\|_{\lebesgue^{\infty,1}}) \\
    & \leq (16 + \eps) \cdot \| F \|_{\IntersectionSpace} .
  \end{align*}
  Since $\eps > 0$ was arbitrary, and by definition of the associate norm,
  this proves the right-hand estimate.

  \medskip{}

  For proving the left-hand estimate, we can assume that $\theta := \varrho_{\otimes}' (F) < \infty$.
  First, note that if $G \in \lebesgue^1 (\mu \otimes \nu)$ with $\|G\|_{\lebesgue^1} \leq 1$,
  then $\varrho_{\otimes}(|G|) \leq 1$, by Proposition~\ref{prop:IteratedFunctionNormEquivalentToSumNorm}.
  Therefore, the definition of $\varrho_{\otimes}'$ yields
  \(
    \big| \int_{X \times Y} F \cdot G \, d (\mu \otimes \nu) \big|
    \leq \int_{X \times Y} F \cdot |G| \, d (\mu \otimes \nu)
    \leq \varrho_{\otimes}' (F)
  \).
  By the dual characterization of the $\lebesgue^\infty$-norm (see \cite[Theorem~6.14]{FollandRA}),
  this implies $\|F\|_{\lebesgue^\infty} \leq \varrho_{\otimes}' (F) < \infty$.
  In the same way (taking $G$ such that $\|G\|_{\lebesgue^\infty} \leq 1$),
  we also see $\|F\|_{\lebesgue^1} \leq \varrho_{\otimes}' (F) < \infty$.
  In particular, this implies that $F$ is finite almost everywhere.

  Finally, note that if $G \in \lebesgue^{1,\infty} (\mu \otimes \nu)$
  with $\|G\|_{\lebesgue^{1,\infty}} \leq 1$, then $\varrho_{\otimes}(|G|) \leq 1$,
  by Proposition~\ref{prop:IteratedFunctionNormEquivalentToSumNorm}.
  Therefore,
  \(
    \big| \int_{X \times Y} F \cdot G \, d (\mu \otimes \nu) \big|
    \leq \int_{X \times Y} F \cdot |G| \, d (\mu \otimes \nu)
    \leq \varrho_{\otimes}' (F)
  \).
  By the dual characterization of the $\lebesgue^{\infty,1}$-norm
  (see Theorem~\ref{thm:MixedNormDuality}, or \cite[Theorem~2 in Section~2]{MixedLpSpaces}),
  this implies ${\|F\|_{\lebesgue^{\infty,1}} \leq \varrho_{\otimes} '(F) < \infty}$.
  Again, we get in the same way (by taking $G$ such that $\|G\|_{\lebesgue^{\infty,1}} \leq 1$)
  that $\|F\|_{\lebesgue^{1,\infty}} \leq \varrho_{\otimes}' (F) < \infty$.
  Overall, these considerations establish the left-hand estimate.

  \medskip{}

  The left-hand estimate also shows that $\varrho_{\otimes}'$ is a function norm, since if
  $\varrho_{\otimes}'(F) = 0$, then in particular $\|F\|_{\lebesgue^\infty} = 0$,
  and thus $F = 0$ almost everywhere.
\end{proof}

We will derive Theorem~\ref{thm:DualOfIntersection} as a consequence of the preceding lemma
and of three beautiful results from the theory of Köthe spaces that we now recall.

\begin{theorem}\label{thm:LorentzLuxemburg}
  (Lorentz-Luxemburg representation theorem; see \cite[Theorem~1 in § 71]{ZaanenIntegration})

  Let $(X,\CalF,\mu)$ be a $\sigma$-finite measure space, and let $\varrho$ be a function seminorm
  on $X$ that satisfies the Fatou property.
  Then $\varrho = \varrho''$; that is, $\varrho$ coincides with the associate seminorm
  of the associate seminorm $\varrho'$ of $\varrho$.
\end{theorem}

\begin{proposition}\label{prop:SecondDualEasyCharacterization}
  Let $(X, \CalF, \mu)$ be a $\sigma$-finite measure space,
  and let $\varrho$ be a function seminorm on $X$.
  If the associate seminorm $\varrho'$ of $\varrho$ is in fact a function \emph{norm},
  then we have the following equivalence for every measurable function $f : X \to \CC$:
  \[
    \varrho '' (|f|) < \infty
    \quad \Longleftrightarrow \quad
    \forall \, g : X \to \CC \text{ measurable with } \varrho' (|g|) < \infty:
    \int_X |f \cdot g| \, d \mu < \infty.
  \]
\end{proposition}

\begin{proof}
  Since $\varrho'$ is a function norm, \cite[Theorem~4 in § 68]{ZaanenIntegration} shows that
  $\varrho$ is \emph{saturated}.
  Therefore, \cite[Corollary in § 71]{ZaanenIntegration} yields the claim.
\end{proof}

\begin{proposition}\label{prop:FunctionSpaceCompleteness}
  (consequence of \cite[Theorem~1 in § 65]{ZaanenIntegration})

  Let $(X,\CalF,\mu)$ be a $\sigma$-finite measure space, and let $\varrho$ be a function
  \emph{norm} on $X$ which satisfies the Fatou property.
  Then the space
  \[
    L_\varrho
    := \big\{
         f : X \to \CC \quad \colon \quad f \text{ measurable and } \varrho(|f|) < \infty
       \big\}
  \]
  is a Banach space when equipped with the norm $\|f\|_{L_\varrho} := \varrho(|f|)$.
  As usual, one identifies two elements of $L_\varrho$ if they agree almost everywhere.
\end{proposition}

We can finally prove Theorem~\ref{thm:DualOfIntersection}.

\begin{proof}[Proof of Theorem~\ref{thm:DualOfIntersection}]
  We first show that $(\SumSpace, \|\bullet\|_{\SumSpace})$ is a Banach space.
  It is not hard to verify that $\| \bullet \|_{\SumSpace}$ is a seminorm.
  Further, Proposition~\ref{prop:IteratedFunctionNormEquivalentToSumNorm}
  states (in the language of Proposition~\ref{prop:FunctionSpaceCompleteness})
  that ${\SumSpace = L_{\varrho_{\otimes}}}$, and that the (semi)-norms
  $\|\bullet\|_\ast := \varrho_{\otimes}(|\bullet|)$ and $\|\bullet\|_{\SumSpace}$ are equivalent.
  Furthermore, Lemma~\ref{lem:IteratedLebesgueSumNormIsFunctionNorm} shows that the function
  \emph{semi}-norm $\varrho_{\otimes}$ is in fact a function \emph{norm},
  and that $\varrho_{\otimes}$ satisfies the Fatou property.
  In particular, this implies that $\| \bullet \|_{\SumSpace}$ is definite, and hence a norm.
  Since the function norm $\varrho_{\otimes}$ satisfies the Fatou property,
  Proposition~\ref{prop:FunctionSpaceCompleteness} shows that
  $(L_{\varrho_{\otimes}}, \|\bullet\|_\ast)$ is a Banach space.
  Hence so is $(\SumSpace, \|\bullet\|_{\SumSpace})$.

  \medskip{}

  Next, since $\varrho_{\otimes}$ satisfies the Fatou property
  (Lemma~\ref{lem:IteratedLebesgueSumNormIsFunctionNorm}),
  the Lorentz-Luxemburg representation theorem (Theorem~\ref{thm:LorentzLuxemburg})
  shows that $\varrho_{\otimes}'' = \varrho_{\otimes}$.
  Furthermore, we saw above that $\SumSpace = L_{\varrho_{\otimes}}$.
  Finally, Lemma~\ref{lem:LebesgueSumNormAssociateNorm} shows that $\varrho_{\otimes}'$
  is a function \emph{norm}.
  Therefore, Proposition~\ref{prop:SecondDualEasyCharacterization} shows
  for every measurable function $F : X_1 \times X_2 \to \CC$ that
  \begin{align*}
      F \! \in \! \lebesgue^1 \!
               + \! \lebesgue^\infty \!
               + \! \lebesgue^{1,\infty} \!
               + \! \lebesgue^{\infty,1}
    & \Longleftrightarrow
      F \in L_{\varrho_{\otimes}}
      \quad \overset{\varrho_{\otimes} = \varrho_{\otimes}''}{\Longleftrightarrow} \quad
      \varrho_{\otimes} '' (|F|) < \infty \\
    & \Longleftrightarrow
      F \!\cdot\! G \in \lebesgue^1 (\mu_1 \!\otimes\! \mu_2)
      \quad \forall \, G \!:\! X_1 \!\times\! X_2 \to \CC \text{ meas.~with } \varrho_{\otimes}' (|G|) < \infty \\
    ({\scriptstyle{\text{Lemma~}\ref{lem:LebesgueSumNormAssociateNorm}}})
    & \Longleftrightarrow
      F \cdot G \in \lebesgue^1 (\mu_1 \otimes \mu_2)
      \quad \forall \, G \in \IntersectionSpace.
  \end{align*}
  Here, we used in the last step that Lemma~\ref{lem:LebesgueSumNormAssociateNorm} shows
  that $\varrho_{\otimes}' (|G|) < \infty$ if and only if ${|G| \in \IntersectionSpace}$
  if and only if ${G \in \IntersectionSpace}$ (if $G$ is measurable).

  \medskip{}

  It remains to prove the norm equivalence.
  To this end, note as a consequence of Proposition~\ref{prop:IteratedFunctionNormEquivalentToSumNorm},
  Lemma~\ref{lem:IteratedLebesgueSumNormIsFunctionNorm}, Theorem~\ref{thm:LorentzLuxemburg},
  and Lemma~\ref{lem:LebesgueSumNormAssociateNorm} that
  \begin{align*}
    & \|F\|_{\SumSpace}
      \leq 16 \, \varrho_{\otimes} (|F|)
      =    16 \, \varrho_{\otimes} '' (|F|) \\
    & =    16 \, \sup \Big\{
                      \int_{X_1 \times X_2} \!\!\!
                        |F| \cdot G
                      \, d(\mu_1 \otimes \mu_2)
                      \, \Big| \,
                      G : X_1 \!\times\! X_2 \to [0,\infty] \text{ meas.~and }
                      \varrho_{\otimes}' (G) \leq 1
                    \Big\} \\
    ({\scriptstyle{\text{Lemma}~\ref{lem:LebesgueSumNormAssociateNorm}}})
    & \leq 16 \, \sup \Big\{
                      \int_{X_1 \times X_2}
                        |F \cdot G|
                      \, d(\mu_1 \otimes \mu_2)
                      \, \Big| \,
                      \begin{array}{l}
                        G \in \IntersectionSpace \\
                        \text{and } \|G\|_{\IntersectionSpace} \leq 1
                      \end{array}
                    \Big\},
  \end{align*}
  which is precisely the first estimate claimed in Theorem~\ref{thm:DualOfIntersection}.

  In a similar way, we get
  \begin{align*}
    & \qquad
      \sup \Big\{
             \int_{X_1 \times X_2}
               |F \cdot G|
             \, d(\mu_1 \otimes \mu_2)
             \, \Big| \,
             \begin{array}{l}
               G \in \IntersectionSpace \\
               \text{and } \|G\|_{\IntersectionSpace} \leq 1
             \end{array}
           \Big\} \\
    ({\scriptstyle{\text{Lemma } \ref{lem:LebesgueSumNormAssociateNorm}}})
    & \leq 16 \, \sup \Big\{
                        \int_{X_1 \times X_2} \!\!\!
                          |F| \cdot G
                        \, d(\mu_1 \!\otimes\! \mu_2)
                        \, \Big| \,
                        G : X_1 \!\times\! X_2 \to [0,\infty] \text{ meas.~and }
                        \varrho_{\otimes}' (G) \leq 1
                      \Big\} \\
    & = 16 \cdot \varrho_{\otimes} '' (|F|)
      = 16 \cdot \varrho_{\otimes} (|F|)
      \leq 16 \cdot \|F\|_{\SumSpace} \,\,\,.
    \qedhere
  \end{align*}
\end{proof}

\section{Proof of Lemma~\ref{lem:CountableDualityCharacterization}}
\label{sec:CountableDualityCharacterization}

For the proof of Lemma~\ref{lem:CountableDualityCharacterization}, we need two
auxiliary results from measure theory that we first collect.

\begin{lemma}\label{lem:LebesgueSpaceSeparability}(see \cite[Proposition~3.4.5]{CohnMeasureTheory})

  Let $(X,\CalF,\mu)$ be a $\sigma$-finite measure space,
  and assume that $\CalF$ is \emph{countably generated}
  (meaning that there is a countable set $\CalF_0 \subset \CalF$
  such that $\CalF$ is generated by $\CalF_0$; that is, $\sigma_X (\CalF_0) = \CalF$).
  Then the space $\lebesgue^1 (\mu)$ is separable.
\end{lemma}

In connection with this criterion, the following result will also turn out to be helpful:

\begin{lemma}\label{lem:EverySetOnlyNeedsCountableGenerator}
  Let $X$ be a set, and let $\CalF_0 \subset 2^X$ be an arbitrary subset of the power set of $X$.
  Let $\CalF := \sigma_X (\CalF_0)$ be the $\sigma$-algebra generated by $\CalF_0$.

  For each $A \in \CalF$, there is a countable family $\CalF_A \subset \CalF_0$ such that
  $A \in \sigma_X (\CalF_A)$.
\end{lemma}

\begin{proof}(see \cite[Exercise~7 in Section~1.1]{CohnMeasureTheory})
  Define
  \[
    \CalG
    := \big\{
         A \in \CalF
         \quad\colon\quad
         \exists \, \CalF_A \subset \CalF_0 \text{ countable such that } A \in \sigma_X (\CalF_A)
       \big\}.
  \]
  It is straightforward to verify that $\CalG$ is a $\sigma$-algebra.
  Furthermore, $\CalF_0 \subset \CalG$, because one can choose $\CalF_A := \{A\}$
  for $A \in \CalF_0$.
  Therefore, $\CalF = \sigma_X (\CalF_0) \subset \CalG \subset \CalF$.
\end{proof}

We will heavily use the following consequence of Lemma~\ref{lem:EverySetOnlyNeedsCountableGenerator}.

\begin{lemma}\label{lem:MeasurableFunctionCountablyGeneratedSigmaAlgebras}
  Let $(\Theta,\CalA)$ and $(\Lambda,\CalB)$ be measurable spaces.
  If $F : \Theta \times \Lambda \to [0,\infty]$ is $\CalA \otimes \CalB$-measurable,
  then there are countably generated $\sigma$-algebras $\CalA_0 \subset \CalA$
  and $\CalB_0 \subset \CalB$ such that $F$ is $\CalA_0 \otimes \CalB_0$-measurable.

  Furthermore, if $\mu : \CalB \to [0,\infty]$ is a $\sigma$-finite measure,
  then $\CalB_0$ can be chosen in such a way that $\mu|_{\CalB_0}$ is still $\sigma$-finite.
\end{lemma}

\begin{proof}
  First, note that
  \[
    \sigma(F)
    := \big\{ F^{-1} (M) \colon M \subset [0,\infty] \text{ measurable} \big\}
    \subset \CalA \otimes \CalB
  \]
  is countably generated.
  One way to see this is that---since $F$ is $\sigma(F)$-measurable---there is a sequence
  $(F_n)_{n \in \N}$ of simple, non-negative, $\sigma(F)$-measurable functions such that
  $F_n \nearrow F$ pointwise; see \cite[Proposition~2.1.8]{CohnMeasureTheory}.
  Let us write ${F_n = \sum_{\ell=1}^{N_n} \alpha_\ell^{(n)} \Indicator_{L_\ell^{(n)}}}$
  with $\alpha_\ell^{(n)} \in [0,\infty)$ and $L_\ell^{(n)} \in \sigma(F)$.
  Then, each $F_n$ is $\Sigma$-measurable, where
  $\Sigma := \sigma(\{ L_\ell^{(n)} \colon n \in \N, 1 \leq \ell \leq N_n \}) \subset \sigma(F)$
  is a countably generated $\sigma$-algebra.
  As a pointwise limit of the $F_n$, also $F$ is $\Sigma$-measurable,
  and hence $\sigma(F) \subset \Sigma \subset \sigma(F)$.
  Therefore, $\sigma(F) = \Sigma$ is indeed countably generated.

  Next, note that
  \[
    L_\ell^{(n)}
    \in     \sigma(F)
    \subset \CalA \otimes \CalB
    =       \sigma ( \{ A \times B \colon A \in \CalA, B \in \CalB \})
    \quad \text{for each } n \in \N \text{ and } 1 \leq \ell \leq N_n .
  \]
  In combination with Lemma~\ref{lem:EverySetOnlyNeedsCountableGenerator},
  this implies that there are countable families $(A_m)_{m \in \N} \subset \CalA$
  and $(B_m)_{m \in \N} \subset \CalB$ such that
  $L_\ell^{(n)} \in \sigma (\{ A_m \times B_m \colon m \in \N \})$
  for all $n \in \N$ and $1 \leq \ell \leq N_n$.
  Therefore, if we define $\CalA_0 := \sigma(\{ A_m \colon m \in \N \})$
  and $\CalB_0 := \sigma(\{ B_m \colon m \in \N \})$,
  then both $\CalA_0 \subset \CalA$ and $\CalB_0 \subset \CalB$ are countably generated, and
  \(
    \sigma(F)
    = \sigma(\{ L_\ell^{(n)} \colon n \in \N, 1 \leq \ell \leq N_n \})
    \subset \CalA_0 \otimes \CalB_0 .
  \)
  Since $F$ is $\sigma(F)$-measurable, this implies that $F$ is $\CalA_0 \otimes \CalB_0$-measurable.

  Finally, if $\mu : \CalB \to [0,\infty]$ is $\sigma$-finite,
  then $\Lambda = \bigcup_{n \in \N} E_n$ for suitable $E_n \in \CalB$
  with $\mu(E_n) < \infty$.
  Instead of the definition of $\CalB_0$ from above, we then define
  $\CalB_0 := \sigma(\{ B_m \colon m \in \N \} \cup \{ E_n \colon n \in \N \})$,
  so that $\CalB_0 \subset \CalB$ is still countably generated, $\mu|_{\CalB_0}$
  is $\sigma$-finite, and one sees precisely as before that
  $F$ is $\CalA_0 \otimes \CalB_0$-measurable.
\end{proof}

Using Lemmas~\ref{lem:LebesgueSpaceSeparability} and
\ref{lem:MeasurableFunctionCountablyGeneratedSigmaAlgebras},
we prove the following final technical ingredient
that we need for the proof of Lemma~\ref{lem:CountableDualityCharacterization}.

\begin{lemma}\label{lem:CountableLInfinityCharacterization}
  Let $(\Theta, \CalA)$ be a measurable space
  and let $(\Lambda, \CalB, \mu)$ be a $\sigma$-finite measure space
  with $\mu(\Lambda) > 0$.

  If $F : \Theta \times \Lambda \to [0,\infty]$ is measurable with respect to the
  product $\sigma$-algebra $\CalA \otimes \CalB$, then there is a countable family
  $(M_n)_{n \in \N} \subset \CalB$ of sets of finite, positive measure such that
  if we set $f_n := \Indicator_{M_n} / \mu(M_n)$, then
  \[
    \esssup_{\lambda \in \Lambda}
      F(\theta, \lambda)
    = \sup_{n \in \N}
        \int_\Lambda
          F(\theta, \lambda) \cdot f_n (\lambda)
        \, d \mu(\lambda)
    \qquad \forall \, \theta \in \Theta .
  \]
  In particular, the map
  \(
    \theta \mapsto \esssup_{\lambda \in \Lambda}
                     F(\theta, \lambda)
  \)
  is measurable.
\end{lemma}

\begin{proof}
  First, Lemma~\ref{lem:MeasurableFunctionCountablyGeneratedSigmaAlgebras}
  yields countably generated sub-$\sigma$-algebras $\CalA_0 \subset \CalA$
  and $\CalB_0 \subset \CalB$ such that $F$ is $\CalA_0 \otimes \CalB_0$-measurable
  and such that $\mu|_{\CalB_0}$ is still $\sigma$-finite.
  The remainder of the proof proceeds in two steps.

  \medskip{}

  \noindent
  \textbf{Step 1} \emph{(Constructing the sets $M_n$):}
  Since $\CalB_0$ is countably generated and $\mu|_{\CalB_0}$ is $\sigma$-finite,
  Lemma~\ref{lem:LebesgueSpaceSeparability} shows that $\lebesgue^1 (\mu|_{\CalB_0})$
  is separable.
  Since non-empty subsets of separable metric spaces are again separable
  (see \cite[Corollary~3.5]{Hitchhiker}), this implies that
  \[
    \mathscr{F}
    := \big\{
         \Indicator_M / \mu(M)
         \quad\colon\quad
         M \in \CalB_0 \text{ and } 0 < \mu(M) < \infty
       \big\}
    \subset \lebesgue^1 (\mu|_{\CalB_0})
  \]
  is separable, so that there is a sequence of sets $(M_n)_{n \in \N} \subset \CalB_0$
  such that $0 < \mu(M_n) < \infty$ for all $n \in \N$, and such that
  $(f_n)_{n \in \N} := \big( \Indicator_{M_n} / \mu(M_n) \big)_{n \in \N} \subset \mathscr{F}$
  is dense.
  Here, we implicitly used that $\mu(\Lambda) > 0$, which implies that $\mathscr{F} \neq \emptyset$,
  by $\sigma$-finiteness of $\mu|_{\CalB_0}$.

  Now, for each $n \in \N$, let us define
  \[
    \Psi_n : \Theta \to [0,\infty],
             \theta \mapsto \int_\Lambda
                              F(\theta, \lambda) \cdot f_n (\lambda)
                            \, d \mu(\lambda)
    \qquad \text{and} \qquad
    \Psi : \Theta \to [0,\infty], \theta \mapsto \sup_{n \in \N} \Psi_n (\theta),
  \]
  as well as
  \(
    \widetilde{\Psi} :
    \Theta \to [0,\infty],
    \theta \mapsto \esssup_{\lambda \in \Lambda} F(\theta, \lambda)
  \).
  \smallskip{}

  Since
  \(
    \Theta \times \Lambda \to [0,\infty],
    (\theta, \lambda) \mapsto F(\theta, \lambda) \cdot f_n(\lambda)
  \)
  is measurable, it follows from Tonelli's theorem (see \cite[Proposition~5.2.1]{CohnMeasureTheory})
  that each $\Psi_n$ is measurable, so that $\Psi$ is measurable as well;
  see \cite[Proposition~2.1.5]{CohnMeasureTheory}.
  Note that strictly speaking, we need to have a $\sigma$-finite measure $\nu$ on $(\Theta,\CalA)$
  to apply Tonelli's theorem, but we can simply take $\nu \equiv 0$.

  \medskip{}

  \noindent
  \textbf{Step 2} \emph{(Proving the claim of the lemma, that is, $\Psi = \widetilde{\Psi}$):}
  Since $\| f_n \|_{\lebesgue^1 (\mu)} = 1$ for all $n \in \N$,
  it is clear that $\Psi_n \leq \widetilde{\Psi}$ for all $n \in \N$,
  and hence $\Psi \leq \widetilde{\Psi}$.
  Now, assume towards a contradiction that $\Psi (\theta) < \widetilde{\Psi} (\theta)$ for some
  $\theta \in \Theta$.
  Then we can choose $\alpha \in \R$ with $0 \leq \Psi(\theta) < \alpha < \widetilde{\Psi} (\theta)$.
  By definition of $\widetilde{\Psi}$, this implies that
  \[
    M := [F(\theta,\bullet)]^{-1} ( [\alpha,\infty])
       = \{ \lambda \in \Lambda \colon F(\theta,\lambda) \geq \alpha \}
  \]
  has positive measure.
  Furthermore, since $F$ is $\CalA_0 \otimes \CalB_0$-measurable,
  \cite[Lemma 5.1.2]{CohnMeasureTheory} shows that $F(\theta,\bullet)$ is $\CalB_0$-measurable,
  and hence $M \in \CalB_0$.

  Since $\mu|_{\CalB_0}$ is $\sigma$-finite, we see that there is some set $M' \in \CalB_0$
  satisfying $M' \subset M$ and furthermore $0 < \mu(M') < \infty$.
  Hence, $f := \Indicator_{M'} / \mu(M') \in \mathscr{F}$,
  so that there is a sequence $(n_k)_{k \in \N}$ such that
  $f_{n_k} \to f$, with convergence in $\lebesgue^1(\mu|_{\CalB_0})$.
  By \cite[Propositions 3.1.3 and 3.1.5]{CohnMeasureTheory}, it follows that there is a further
  subsequence $(n_{k_\ell})_{\ell \in \N}$ such that $f_{n_{k_\ell}} \to f$ $\mu$-almost everywhere.
  Therefore, recalling that $F(\theta,\bullet) \geq \alpha$ on $M' \subset M$,
  while $f \equiv 0$ on $\Lambda \setminus M'$,
  and recalling that $\int_{\Lambda} f_{n_k} (\lambda) \, d \mu(\lambda) = 1$,
  we see as a consequence of Fatou's lemma that
  \begin{align*}
    \alpha
    & = \alpha \cdot \int_\Lambda \, f d \mu
      \leq \int_\Lambda
             f(\lambda) \cdot F(\theta,\lambda)
           \, d \mu (\lambda)
      =    \int_\Lambda
             \liminf_{\ell \to \infty}
             \Big(
               f_{n_{k_\ell}} (\lambda)
               \cdot F(\theta, \lambda)
             \Big)
           \, d \mu (\lambda) \\
    & \leq \liminf_{\ell \to \infty}
             \int_\Lambda
               f_{n_{k_\ell}} (\lambda)
               \cdot F(\theta, \lambda)
           \, d \mu (\lambda)
      =    \liminf_{\ell \to \infty} \Psi_{n_{k_\ell}} (\theta)
      \leq \Psi (\theta) ,
  \end{align*}
  which is the desired contradiction, since $\Psi(\theta) < \alpha$.
\end{proof}

Finally, we prove Lemma~\ref{lem:CountableDualityCharacterization}.

\begin{proof}[Proof of Lemma~\ref{lem:CountableDualityCharacterization}]
  In case of $\nu_1(Y_1) = 0$ or $\nu_2(Y_2) = 0$, we have
  $\| H(\omega,\bullet) \|_{\lebesgue^{\infty,1}(\nu)} = 0$ for all $\omega \in \Omega$,
  so that we can simply take $h_n \equiv 0$ for all $n \in \N$.
  In the following, we can thus assume that $\nu_1(Y_1) > 0$ and $\nu_2(Y_2) > 0$.
  The proof is divided into three steps.

  \medskip{}

  \noindent
  \textbf{Step 1}
  \emph{(Writing $\| H \big(\omega, (\bullet,y_2) \big) \|_{\lebesgue^{\infty}(\nu_1)}$
  as a countable supremum of integrals):}
  We aim to apply Lemma~\ref{lem:CountableLInfinityCharacterization}.
  To this end, define $(\Theta, \CalA) := (\Omega \times Y_2, \CalC \otimes \CalG_2)$
  and $(\Lambda,\CalB,\mu) := (Y_1, \CalG_1,\nu_1)$.
  Finally, set
  \(
    F :
    \Theta \times \Lambda \to [0,\infty],
    \big( (\omega, y_2), y_1 \big) \mapsto H \big(\omega, (y_1,y_2) \big)
  \).
  By applying Lemma~\ref{lem:CountableLInfinityCharacterization} with these choices,
  we obtain a sequence $(f_n)_{n \in \N} \subset \lebesgue^1(\nu_1)$ of non-negative
  functions which satisfy $\| f_n \|_{\lebesgue^1 (\nu_1)} = 1$ and
  \begin{equation}
    \begin{split}
      \big\| H \big( \omega, (\bullet,y_2) \big) \big\|_{\lebesgue^\infty (\nu_1)}
      & = \big\|
            F \big( (\omega,y_2), \bullet \big)
          \big\|_{\lebesgue^\infty (\mu)}
        = \sup_{n \in \N}
          \int_\Lambda
            F \big( (\omega,y_2), \lambda \big) \cdot f_n (\lambda)
          \, d \mu (\lambda) \\
      & = \sup_{n \in \N}
          \int_{Y_1}
            H \big( \omega, (y_1,y_2) \big) \cdot f_n (y_1)
          \, d \nu_1 (y_1)
        =: \sup_{n \in \N}
             \Psi_n (\omega, y_2)
    \end{split}
    \label{eq:EssentialSupremumAsCountableSupremum}
  \end{equation}
  for all $(\omega,y_2) \in \Omega \times Y_2$.
  Here, it follows from Tonelli's theorem that
  $\Psi_n : (\Omega \times Y_2, \CalC \otimes \CalG_2) \to [0,\infty]$ is measurable.

  \medskip{}

  \noindent
  \textbf{Step 2} \emph{(Finding $g_k : \Omega \times Y \to [0,\infty)$ such that
  \(
    \int_Y H(\omega,y) g_k(\omega,y) \, d \nu(y)
    \xrightarrow[k\to\infty]{} \| H(\omega,\bullet) \|_{\lebesgue^{\infty,1}}
  \)
  and $\| g_k (\omega, \bullet) \|_{\lebesgue^{1,\infty}} \leq 1$,
as well as $g_k (\omega, \bullet) \in \lebesgue^1(\nu)$):}

  Since $\nu_2$ is $\sigma$-finite, we have $Y_2 = \bigcup_{n \in \N} E_n$ for certain
  $E_n \in \CalG_2$ with $\nu_2 (E_n) < \infty$ and $E_n \subset E_{n+1}$ for all $n \in \N$.
  Now, given $k \in \N$, define $g_k : \Omega \times Y \to [0,\infty)$ by
  \[
    g_k \big( \omega, (y_1,y_2) \big)
    := \Indicator_{E_k} (y_2)
       \cdot \sum_{\ell=1}^k
             \bigg(
               f_\ell (y_1)
               \cdot \prod_{m=1}^{\ell - 1}
                       \Indicator_{\Psi_\ell > \Psi_m} (\omega, y_2)
               \cdot \prod_{m = \ell+1}^k
                       \Indicator_{\Psi_\ell \geq \Psi_m} (\omega, y_2)
             \bigg) .
  \]
  The significance of this convoluted-seeming definition will be explained shortly.
  Before that, however, it should be noted that $g_k$ is $\CalC \otimes \CalG$-measurable.
  Now, given fixed $(\omega, y_2) \in \Omega \times Y_2$,
  let $\ell_0 = \ell_0 (\omega, y_2, k) \in \{ 1,\dots,k \}$ be minimal with
  $\Psi_{\ell_0} (\omega,y_2) = \max_{1 \leq \ell \leq k} \Psi_\ell (\omega, y_2)$.
  We claim that then $g_k (\omega, (\bullet, y_2)) = \Indicator_{E_k} (y_2) \cdot f_{\ell_0}$.
  Indeed, for $\ell \in \{ 1,\dots,k \}$, there are three cases:
  \begin{itemize}
    \item If $\ell < \ell_0$, then $\Psi_\ell (\omega,y_2) < \Psi_{\ell_0} (\omega,y_2)$,
          while $\ell_0 \in \{ \ell+1,\dots,k \}$.
          Hence, the $\ell$-th summand in the definition
          of $g_k \big( \omega, (\bullet, y_2) \big)$ vanishes.

    \item If $\ell > \ell_0$, then $\Psi_\ell (\omega,y_2) \leq \Psi_{\ell_0} (\omega,y_2)$
          and $\ell_0 \in \{1,\dots,\ell-1\}$.
          Hence, the $\ell$-th summand in the definition of $g_k \big (\omega, (\bullet,y_2) \big)$
          again vanishes.

    \item If $\ell = \ell_0$, then $\Psi_{\ell_0} (\omega,y_2) > \Psi_m (\omega,y_2)$
          for $m \in \{ 1,\dots,\ell_0 - 1 \}$ and
          $\Psi_{\ell_0} (\omega,y_2) \geq \Psi_m (\omega, y_2)$
          for $m \in \{ \ell_0+1,\dots,k \}$.
          Therefore, the $\ell_0$-th summand in the definition of
          $g_k \big( \omega, (\bullet,y_2) \big)$ is simply $f_{\ell_0}$.
  \end{itemize}
  This representation of $g_k(\omega, (\bullet,y_2))$ has two crucial implications:
  \begin{enumerate}
    \item With $\ell_0 = \ell_0(\omega, y_2, k)$ as above, we have
          $g_k (\omega, (\bullet,y_2)) = \Indicator_{E_k} (y_2) \cdot f_{\ell_0}$, which implies that
          $\| g_k (\omega, (\bullet,y_2)) \|_{\lebesgue^1(\nu_1)} \leq \Indicator_{E_k}(y_2)$,
          since $\| f_{\ell_0} \|_{\lebesgue^1 (\nu_1)} = 1$.
          Therefore,
          \(
            \| g_k (\omega, \bullet) \|_{\lebesgue^{1,\infty}(\nu)} \leq 1
          \)
          and
          \[
            \| g_k (\omega, \bullet) \|_{\lebesgue^1 (\nu)}
            = \int_{Y_2}
                \big\|
                  g_k \big( \omega, (\bullet,y_2) \big)
                \big\|_{\lebesgue^1(\nu_1)} \, d \nu_2(y_2)
            \leq \nu_2 (E_k) < \infty.
          \]

    \item For fixed $(\omega,y_2) \in \Omega \times Y_2$ and with $\ell_0 = \ell_0(\omega, y_2, k)$
          as above, we have
          \begin{align*}
            \qquad
            \int_{Y_1} \!\!
              H \big( \omega, (y_1,y_2) \big)
              \cdot g_k \big( \omega, (y_1,y_2) \big)
            \, d \nu_1 (y_1)
            & = \Indicator_{E_k} (y_2)
                \,\, \int_{Y_1}
                       H \big( \omega, (y_1,y_2) \big)
                       \cdot f_{\ell_0} (y_1)
                     \, d \nu_1(y_1) \\
            & = \Indicator_{E_k} (y_2)
                \cdot \Psi_{\ell_0} (\omega, y_2)
              = \Indicator_{E_k} (y_2) \cdot \max_{1 \leq \ell \leq k} \Psi_\ell (\omega,y_2) \\
            & \nearrow \sup_{n \in \N} \Psi_n (\omega,y_2)
              = \big\| H \big( \omega, (\bullet,y_2) \big) \big\|_{\lebesgue^\infty(\nu_1)}
          \end{align*}
          as $k \to \infty$, thanks to Equation~\eqref{eq:EssentialSupremumAsCountableSupremum}.
          By the monotone convergence theorem and Tonelli's theorem, this shows
          for arbitrary $\omega \in \Omega$ that
          \begin{align*}
            \qquad \qquad
            \int_Y
              H ( \omega, y ) \cdot g_k ( \omega, y )
            \, d \nu(y)
            & = \int_{Y_2}
                  \int_{Y_1}
                    H \big( \omega, (y_1,y_2) \big) \cdot g_k \big( \omega, (y_1, y_2) \big)
                  \, d \nu_1(y_1)
                \, d \nu_2 (y_2) \\
            & \xrightarrow[k\to\infty]{}
                \int_{Y_2}
                  \big\| H \big( \omega, (\bullet,y_2) \big) \big\|_{\lebesgue^\infty (\nu_1)}
                \, d \nu_2(y_2)
              = \| H(\omega,\bullet) \|_{\lebesgue^{\infty,1}(\nu)}.
          \end{align*}
  \end{enumerate}

  \medskip{}

  \noindent
  \textbf{Step 3} \emph{(Completing the proof):}
  Since $g_k$ is $\CalC \otimes \CalG$-measurable,
  Lemma~\ref{lem:MeasurableFunctionCountablyGeneratedSigmaAlgebras} yields for each $k \in \N$
  countably generated $\sigma$-algebras $\CalC^{(k)} \subset \CalC$ and $\CalG^{(k)} \subset \CalG$%
  ---say $\CalG^{(k)} = \sigma \big( \big\{ M_n^{(k)} \colon n \in \N \big\} \big)$---%
  such that $g_k$ is $\CalC^{(k)} \otimes \CalG^{(k)}$-measurable and such that $\nu|_{\CalG^{(k)}}$
  is $\sigma$-finite.
  Define
  \[
    \CalG_0 := \sigma(\{ M_n^{(k)} \colon n,k \in \N \}) \subset \CalG,
  \]
  and note that $\nu|_{\CalG_0}$ is $\sigma$-finite.
  Lemma~\ref{lem:LebesgueSpaceSeparability} shows that $\lebesgue^1(\nu|_{\CalG_0})$ is separable.
  Since any non-empty subset of a separable metric space is again separable
  (see \cite[Corollary~3.5]{Hitchhiker}), this implies that
  \[
    \mathscr{F}
    := \big\{
         h \in \lebesgue^1(\nu|_{\CalG_0})
         \quad\colon\quad
         h \geq 0 \text{ and } \| h \|_{\lebesgue^{1,\infty}} \leq 1 < \infty
       \big\}
  \]
  is separable as well.
  Thus, let $\{ h_n \colon n \in \N \} \subset \mathscr{F}$ be dense
  with respect to $\| \bullet \|_{\lebesgue^1(\nu)}$.
  Note that since $g_k$ is measurable with respect to
  $\CalC^{(k)} \otimes \CalG^{(k)} \subset \CalC \otimes \CalG_0$,
  and by the properties of $g_k$ derived in Step~2, we have $g_k (\omega, \bullet) \in \mathscr{F}$
  for all $\omega \in \Omega$.

  For brevity, write $\Psi (\omega) := \| H(\omega, \bullet) \|_{\lebesgue^{\infty,1}}$
  and $\widetilde{\Psi} (\omega) := \sup_{n \in \N} \int_Y H(\omega,y) \, h_n(y) \, d \nu(y)$.
  We want to prove that $\Psi = \widetilde{\Psi}$.
  First, by the Hölder inequality for the mixed-norm Lebesgue spaces
  (see \cite[Equation~(1) in Section~2]{MixedLpSpaces}), we see that
  \[
    0
    \leq \int_Y
           H(\omega, y) \, h_n (y)
         \, d \nu(y)
    \leq \| H(\omega, \bullet) \|_{\lebesgue^{\infty,1}}
         \| h_n \|_{\lebesgue^{1,\infty}}
    \leq \Psi(\omega)
  \]
  for all $n \in \N$, and hence $\widetilde{\Psi} \leq \Psi$.

  To prove the converse, let $\omega \in \Omega$ and $k \in \N$ be arbitrary.
  Since $g_k (\omega, \bullet) \in \mathscr{F}$, there is a sequence $(n_m)_{m \in \N}$
  satisfying $h_{n_m} \to g_k (\omega, \bullet)$ as $m \to \infty$,
  with convergence in $\lebesgue^1 (\nu)$.
  It is well-known (see for instance \mbox{\cite[Propositions 3.1.3 and 3.1.5]{CohnMeasureTheory}})
  that this implies that there is a subsequence
  $(n_{m_\ell})_{\ell \in \N}$ such that $h_{n_{m_\ell}} \to g_k (\omega, \bullet)$
  as $\ell \to \infty$, with convergence $\nu$-almost everywhere.
  By Fatou's lemma, this implies that
  \begin{align*}
    \int_Y
      H(\omega, y) \cdot g_k (\omega,y)
    \, d \nu(y)
    & = \int_Y
          \liminf_{\ell \to \infty}
          \Big(
            H(\omega, y) \cdot h_{n_{m_\ell}} (y)
          \Big)
        \, d \nu(y) \\
    & \leq \liminf_{\ell \to \infty}
             \int_Y
               H(\omega,y) \cdot h_{n_{m_\ell}} (y)
             \, d \nu(y)
      \leq \widetilde{\Psi} (\omega) .
  \end{align*}
  Since we saw in Step~2 that
  \(
    \int_Y
      H(\omega, y) \cdot g_k (\omega,y)
    \, d \nu(y)
    \smash{\xrightarrow[k\to\infty]{}} \Psi(\omega),
  \)
  we arrive at $\Psi(\omega) \leq \widetilde{\Psi}(\omega)$, as desired.
\end{proof}

\section{A technical result concerning the embedding
\texorpdfstring{$\Phi_{K_\Psi} (\BanachOne) \hookrightarrow \lebesgue_{1/v}^{\infty}$}
{into a weighted 𝑳∞ space}}%
\label{sec:BoundednessEmbeddingEquivalence}

In \cite{kempka2015general}, it is assumed that
$\Phi_{K_\Psi}(\BanachOne) \hookrightarrow \lebesgue_{1/v}^\infty (\mu)$, meaning that
Equation~\eqref{eq:CoorbitLInftyEmbeddingAssumption} holds.
For general integral kernels $K$ instead of $K_\Psi$, this would be a much stronger condition
than boundedness of $\Phi_K : \BanachOne \to \lebesgue_{1/v}^\infty (\mu)$.
Since $K_\Psi$ is a reproducing kernel, however, the two conditions are actually equivalent,
as we now show.

\begin{lemma}\label{lem:BoundednessEmbeddingEquivalence}
  Let $\XTuple$ be a $\sigma$-finite measure space, let $\Hil$ be a separable Hilbert space,
  and let $\Psi = (\psi_x)_{x \in X} \subset \Hil$ be a continuous Parseval frame for $\Hil$.
  Finally, let $\BanachOne$ be a solid Banach function space on $X$ and let $v : X \to (0,\infty)$
  be measurable.

  With $K_\Psi$ as defined in Equation~\eqref{eq:ReproducingKernelDefinition}, assume that
  $\Phi_{|K_\Psi|} : \BanachOne \to \BanachOne$
  and $\Phi_{K_\Psi} : \BanachOne \to \lebesgue_{1/v}^\infty$ are well-defined and bounded.
  Then $\Phi_{K_\Psi}(\BanachOne) \hookrightarrow \lebesgue_{1/v}^\infty$,
  meaning that Equation~\eqref{eq:CoorbitLInftyEmbeddingAssumption} holds.
\end{lemma}

\begin{proof}
  Note that if $(\varphi_n)_{n \in I}$ is a \emph{countable} orthonormal basis
  for $\Hil$ (which exists by separability), then
  \(
    K_\Psi (x,y)
    = \langle \psi_y, \psi_x \rangle_{\Hil}
    = \sum_{n \in I}
        \langle \psi_y, \varphi_n \rangle_{\Hil}
        \langle \varphi_n, \psi_x \rangle_{\Hil} ,
  \)
  where $x \mapsto \langle \varphi_n, \psi_x \rangle_{\Hil}$ and
  $y \mapsto \langle \psi_y, \varphi_n \rangle_{\Hil}$ are measurable
  by definition of a continuous frame.
  Hence, $K := K_\Psi : X \times X \to \CC$ is measurable.

  By definition of a continuous Parseval frame, the voice transform
  \(
    V_\Psi : \Hil \to \lebesgue^2(\mu), f \mapsto V_\Psi f
  \)
  with $V_\Psi f (x) = \langle f, \psi_x \rangle_{\Hil}$ is an isometry, so that
  $V_\Psi^\ast V_\Psi = \identity_{\Hil}$.
  Thus, ${P := V_\Psi V_\Psi^\ast : \lebesgue^2(\mu) \to \lebesgue^2(\mu)}$ satisfies $P P = P$.
  Now, note that $V_\Psi^\ast F = \int_{X} F(y) \, \psi_y \, d \mu (y)$
  (with the integral understood in the weak sense), so that
  \(
    (P F)(x)
    = \langle V_\Psi^\ast F, \psi_x \rangle
    = \int_X F(y) \langle \psi_y, \psi_x \rangle \, d \mu (y)
    = (\Phi_K F) (x),
  \)
  meaning that $P = \Phi_K$.
  Because of $P P = P$, this means that $\Phi_K \Phi_K F = \Phi_K F$ for all $F \in \lebesgue^2(\mu)$.

  We now claim that $\Phi_K \Phi_K F = \Phi_K F$ also holds for all $F \in \BanachOne$.
  Once we show this, we immediately get the claim of the lemma, since then
  \[
    \| \Phi_{K} F \|_{\lebesgue_{1/v}^\infty}
    = \| \Phi_K \Phi_{K} F \|_{\lebesgue_{1/v}^\infty}
    \leq C \cdot \| \Phi_K F \|_{\BanachOne}
    \qquad \forall \, F \in \BanachOne ,
  \]
  since $\Phi_K : \BanachOne \to \lebesgue_{1/v}^\infty(\mu)$ is bounded by assumption of the lemma.

  Let $F \in \BanachOne$ be arbitrary.
  Since $X$ is $\sigma$-finite, we have $X = \bigcup_{n=1}^\infty X_n$ where $X_n \subset X_{n+1}$
  and $\mu(X_n) < \infty$.
  Define $Y_n := \{ x \in X_n \colon |F(x)| \leq n \}$, and note that $Y_n \subset Y_{n+1}$,
  $\mu(Y_n) < \infty$, and $X = \bigcup_{n =1}^\infty Y_n$, so that
  $F_n := \Indicator_{Y_n} \cdot F \in \BanachOne \cap \lebesgue^2(\mu)$ satisfies $F_n \to F$ pointwise.
  Note that $\Phi_K \Phi_K F_n = \Phi_K F_n$ and $|F_n| \leq |F|$.
  Since $\Phi_{|K|} : \BanachOne \to \BanachOne$ is well-defined, we have
  ${G := \Phi_{|K|} |F| \in \BanachOne}$, and in particular $(\Phi_{|K|} |F|) (x) < \infty$
  for $\mu$-almost all $x \in X$.
  For each such $x$, we see by the dominated convergence theorem that
  \[
    G_n (x)
    := \Phi_K F_n (x)
    = \int_{X} K(x,y) \, F_n (y) \, d \mu(y)
    \xrightarrow[n\to\infty]{} \int_X K(x,y) \, F(y) \, d \mu(y)
    = \Phi_K F (x) .
  \]
  Next, we also have $\int_X |K(x,y)| \, G(y) \, d \mu(y) = (\Phi_{|K|} G)(x) < \infty$
  for $\mu$-almost all $x \in X$.
  Since $|G_n (y)| \leq G (y)$, the dominated convergence theorem thus shows for $x \in X$
  with $(\Phi_{|K|} G) (x) < \infty$ (and hence for $\mu$-almost all $x \in X$) that
  \[
    \Phi_K G_n(x)
    = \int_X K(x,y) \, G_n (y) \, d \mu(y)
    \xrightarrow[n\to\infty]{} \int_X K(x,y) \, \Phi_K F(y) \, d \mu(y)
    = \Phi_K [\Phi_K F] (x) .
  \]
  Since we also have $\Phi_K G_n = \Phi_K \Phi_K F_n = \Phi_K F_n \xrightarrow[n\to\infty]{} \Phi_K F$
  almost everywhere, we thus see ${\Phi_K [\Phi_K F] = \Phi_K F}$ almost everywhere, as desired.
\end{proof}

\section{Sharpness of Schur's test for complex-valued kernels}%
\label{sec:SharpnessComplexValued}

The classical form of Schur's test gives a \emph{complete characterization} of the boundedness of
$\Phi_K : \lebesgue^p \to \lebesgue^p$ for $p \in [1,\infty]$, even for \emph{complex-valued}
integral kernels $K : X \times Y \to \CC$.
This seems to be a folklore result, but we could not locate an appropriate reference.
Therefore, we provide a proof below.
Somewhat surprisingly, it turns out that our generalized form of Schur's test
for mixed Lebesgue-spaces does not provide such a complete characterization
of the boundedness of $\Phi_K$ for complex-valued kernels $K$, as we will show using an example.

\begin{proposition}\label{prop:ClassicalSchurSharpness}
  Let $\XTuple$ and $\YTuple$ be $\sigma$-finite measure spaces, and let $K : X \times Y \to \CC$
  be measurable.
  Let $(Y_n)_{n \in \N} \subset \CalG$ with $Y = \bigcup_{n=1}^\infty Y_n$ and $\nu(Y_n) < \infty$,
  as well as $Y_n \subset Y_{n+1}$, and such that $\Phi_K f : X \to \CC$ is well-defined
  ($\mu$-almost everywhere) for every $f \in \mathscr{G}$, for the vector space
  \[
    \mathscr{G}
    := \bigl\{
         f \in \lebesgue^\infty(\nu)
         \quad \colon \quad
         \exists \, n \in \N:
           f = 0 \text{ a.e.~on } Y \setminus Y_n
       \bigr\}
    \subset \lebesgue^1(\nu) \cap \lebesgue^\infty (\nu)
    .
  \]
  Finally, assume for each $p \in \{ 1, \infty \}$ that
  $\Phi_K : (\mathscr{G}, \| \bullet \|_{\lebesgue^p (\nu)}) \to \lebesgue^p(\mu)$
  is bounded, with operator norm $\theta_p$.
  Then we have $C_1 (K) \leq \theta_\infty < \infty$ and $C_2(K) \leq \theta_1 < \infty$.
\end{proposition}

\begin{rem*}
  The space $\mathscr{G}$ is introduced to avoid having to assume a priori that $\Phi_K f$
  is a well-defined function for $f \in \lebesgue^1 (\nu) + \lebesgue^\infty (\nu)$,
  which would be unnecessarily restrictive.
\end{rem*}

\begin{proof}
  \textbf{Step~1:} (\emph{Showing $C_1(K) \leq \theta_\infty$}):
  The claim is clear in case of $\nu(Y) = 0$; therefore, let us assume $\nu(Y) > 0$.
  Since $[\Phi_K \Indicator_{Y_\ell}] (x)$ is well-defined almost everywhere,
  for each $\ell \in \N$ there is a $\mu$-null-set $M_\ell \subset X$
  such that $\int_{Y_\ell} |K (x,y)| \, d \nu (y) < \infty$ for all $x \in X \setminus M_\ell$.

  Next, by splitting $K$ into the positive- and negative part of its real- and imaginary parts,
  Lemma~\ref{lem:MeasurableFunctionCountablyGeneratedSigmaAlgebras} yields a countably generated
  $\sigma$-algebra $\CalG_0 \subset \CalG$ such that $K$ is $\CalF \otimes \CalG_0$-measurable.
  By enlarging $\CalG_0$, we can assume that $Y_n \in \CalG_0$ for all $n \in \N$.
  Now, Lemma~\ref{lem:LebesgueSpaceSeparability} shows that $\lebesgue^1(\nu|_{\CalG_0})$
  is separable.

  Since nonempty subsets of separable metric spaces are again separable
  (see \cite[Corollary~3.5]{Hitchhiker}), setting
  \(
    \mathscr{G}_0
    := \{
         f \in \lebesgue^1(\nu|_{\CalB_0}) \cap \mathscr{G}
         \colon
         \| f \|_{\lebesgue^\infty} \leq 1
       \} ,
  \)
  we can find a countable family $(g_n)_{n \in \N} \subset \mathscr{G}_0$ which is dense
  (with respect to $\| \bullet \|_{\lebesgue^1(\nu)}$) in $\mathscr{G}_0$.

  Define $h_{\ell,n} := \Indicator_{Y_\ell} \cdot g_n \in \mathscr{G}_0 \subset \mathscr{G}$,
  and note by our assumptions that for each $\ell,n \in \N$, there is a $\mu$-null-set
  $N_{\ell,n} \subset X$ satisfying $|(\Phi_K h_{\ell,n}) (x)| \leq \theta_\infty < \infty$
  for all $x \in X \setminus N_{\ell,n}$.
  Define ${N := \bigcup_{\ell=1}^\infty M_\ell \cup \bigcup_{\ell,n=1}^\infty N_{\ell,n}}$
  and note $\mu(N) = 0$.
  Fix $x \in X \setminus N$ and $\ell \in \N$ for the moment, and define
  ${f_\ell := \Indicator_{Y_\ell} \cdot \overline{\mathrm{sign} \bigl( K(x,\bullet) \bigr)}}$,
  noting that $f_\ell \in \mathscr{G}_0$.
  Since convergence in $\lebesgue^1(\nu)$ implies existence of a subsequence converging almost everywhere,
  we thus obtain a sequence $(n_k)_{k \in \N}$ such that $g_{n_k} \to f_\ell$ $\nu$-almost everywhere,
  and hence also $h_{\ell, n_k} = \Indicator_{Y_\ell} \cdot g_{n_k} \to f_\ell$ $\nu$-almost everywhere
  as $k \to \infty$.
  Since also $|K(x,y) \, h_{\ell,n_k}(y)| \leq \Indicator_{Y_\ell}(y) \, |K(x,y)|$
  and $\int_{Y_\ell} |K(x,y)| \, d \nu(y) < \infty$, we can apply the dominated convergence theorem
  to conclude
  \begin{align*}
    \int_{Y_\ell}
      |K(x,y)|
    \, d \nu(y)
    & = \bigg|
          \int_Y
            f_\ell(y) \, K(x,y)
          \, d \nu(y)
        \bigg| \\
    & = \lim_{k \to \infty}
        \bigg|
          \int_Y
          K(x,y) \, h_{\ell,n_k}(y)
          \, d \nu(y)
        \bigg|
      = \lim_{k \to \infty}
          \bigl| [\Phi_K \, h_{\ell,n_k}](x) \bigr|
      \leq \theta_\infty .
  \end{align*}
  Since this holds for all $\ell \in \N$, while $Y = \bigcup_{\ell = 1}^\infty Y_\ell$
  and $Y_\ell \subset Y_{\ell+1}$, we see $\int_Y |K(x,y)| \, d \nu(y) \leq \theta_\infty$
  for all $x \in X \setminus N$, and hence $C_1 (K) \leq \theta_\infty < \infty$.

  It should be noted that this step did not use that $\theta_1 < \infty$.

  \medskip{}

  \noindent
  \textbf{Step~2:} (\emph{Showing $C_2(K) \leq \theta_1$}):
  This is essentially a duality argument.
  Since $\mu$ is $\sigma$-finite, we can find $(X_n)_{n \in \N} \subset \CalF$
  with $\mu(X_n) < \infty$, $X_n \subset X_{n+1}$ and $X = \bigcup_{n=1}^\infty X_n$.
  Let us define
  \({
    \widetilde{\mathscr{G}}
    := \bigl\{
         g \in \lebesgue^\infty (\mu)
         \quad \colon \quad
         \exists \, n \in \N : g = 0 \text{ a.e.~on } X \setminus X_n
       \bigr\} .
  }\)
  Since $C_1 (K) < \infty$ by the previous step, we have
  \[
    \int_Y \int_{X_\ell} |K^T (y,x)| \, d \mu(x) \, d \nu(y)
    = \int_{X_\ell} \int_Y |K(x,y)| \, d \nu(y) \, d \mu(x)
    \leq C_1(K) \cdot \mu(X_\ell)
    <    \infty ,
  \]
  which easily implies that $(\Phi_{K^T} g) (y)$ is well-defined for each $g \in \widetilde{\mathscr{G}}$
  and $\nu$-almost all $y \in Y$.

  Finally, if $\ell \in \N$ and if $f : Y \to \CC$ is a simple function
  (i.e., a finite linear combination of indicator functions of measurable sets)
  with $f = 0$ on $Y \setminus Y_\ell$ (so that in particular $f \in \mathscr{G}$),
  and if $g \in \widetilde{\mathscr{G}}$, then
  \[
    \int_X \int_Y |K(x,y)| \, |f(y)|  \, d \nu(y) \, |g(x)| \, d \mu(x)
    \leq C_1(K) \| f \|_{\lebesgue^\infty} \| g \|_{\lebesgue^1}
    <    \infty ,
  \]
  so that Fubini's theorem is applicable in the following calculation:
  \begin{align*}
    \bigg|
      \int_Y
        f(y) \cdot (\Phi_{K^T} g) (y)
      \, d \nu(y)
    \bigg|
    & = \bigg|
          \int_X
            g(x)
            \int_Y
              K(x,y) \, f (y)
            \, d \nu(y)
          \, d \mu(x)
        \bigg| \\
    & \leq \| g \|_{\lebesgue^\infty(\mu)}
           \cdot \| \Phi_K f \|_{\lebesgue^1(\mu)}
      \leq \| g \|_{\lebesgue^\infty(\mu)}
           \cdot \theta_1
           \cdot \| f \|_{\lebesgue^1(\nu)} .
  \end{align*}
  By the usual characterization of the $\lebesgue^\infty$-norm by duality
  (see \cite[Theorem~6.14]{FollandRA}) applied on $Y_\ell$, this implies
  $\| (\Phi_{K^T} g)|_{Y_\ell} \|_{\lebesgue^\infty(\nu)} \leq \theta_1 \cdot \| g \|_{\lebesgue^\infty(\mu)}$.
  Since this holds for every $\ell \in \N$, we get
  $\| \Phi_{K^T} g \|_{\lebesgue^\infty(\nu)} \leq \theta_1 \cdot \| g \|_{\lebesgue^\infty(\mu)}$
  for all $g \in \widetilde{\mathscr{G}}$.
  Therefore, applying Step~1 to $K^T$ instead of $K$, we see that $C_1(K^T) \leq \theta_1 < \infty$.
  Since $C_2(K) = C_1(K^T)$ by (an obvious variation of) Lemma~\ref{lem:SchurConstantsForAdjointKernel},
  we are done.
\end{proof}

Now, we provide an example showing that for mixed Lebesgue spaces and \emph{complex-valued}
kernels, our generalized form of Schur's test is in general only sufficient---but not necessary---%
for the boundedness of $\Phi_K : \lebesgue^{p,q} \to \lebesgue^{p,q}$ for all $p,q \in [1,\infty]$.

\begin{example}\label{exa:ComplexKernelNoCharacterization}
  Define $X_1 := [0,1]$ with the Borel $\sigma$-algebra, and $X_2 := Y_1 := Y_2 := \Z$,
  all equipped with the power set $2^\Z$ as the $\sigma$-algebra.
  Next, let $\mu_1 := \lambda$ be the Lebesgue measure on $[0,1]$ and $\nu_2$ the counting measure on $\Z$.
  Furthermore, define $\mu_2 (A) := \sum_{n \in A} e^{-|n|}$ for $A \subset \Z$, noting that this is a
  finite measure on $\Z$ satisfying $\mu_2(\{ k \}) > 0$ for all $k \in \Z$.
  Finally, choose a positive sequence $c = (c_m)_{m \in \Z} \in \ell^2(\Z) \setminus \ell^1(\Z)$,
  say $c_m = (1 + |m|)^{-2/3}$, and define $\nu_1 (A) := \sum_{n \in A} \beta_n$ for $A \subset \Z$,
  where $\beta_n := \bigl( (1 + n^2) \cdot \sum_{|m| \leq |n|} c_m \bigr)^{-1}$.
  Note that ${\nu_1(\Z) \leq c_0^{-1} \sum_{n \in \Z} (1 + n^2)^{-1} < \infty}$.

  Now, with $X = X_1 \times X_2$ and $Y = Y_1 \times Y_2$ and with $\mu = \mu_1 \otimes \mu_2$
  and $\nu = \nu_1 \otimes \nu_2$, define
  \[
    K : X \times Y \to \CC,
        \big( (x,k), (n,m) \big) \mapsto c_m
                                         \cdot e^{-2 \pi i m x}
                                         \cdot \Indicator_{|m| \leq |n|}
                                         \cdot \Indicator_{|m| \leq |k|} .
  \]
  We claim that $\Phi_K : \lebesgue^{p,q}(\nu) \to \lebesgue^{p,q}(\mu)$ is well-defined
  and bounded for all $p,q \in [1,\infty]$, but that $C_3 (K) = \infty$,
  while $C_i (K) < \infty$ for $i \in \{ 1,2,4 \}$.
  This will show that for complex-valued kernels, our generalized form of Schur's test does not
  yield a complete characterization, in contrast to the case of non-negative kernels.

  First of all, note for fixed $p,q \in [1,\infty]$, $k \in \Z$, and $f \in \lebesgue^{p,q}(\nu)$ that
  \begin{equation}
    \begin{split}
      \bigl| (\Phi_K f) (x,k) \bigr|
      & \leq (\Phi_{|K|} |f|) (x,k)
        \leq \sum_{|m| \leq |k|}
               c_m
               \sum_{n \in \Z}
                 \beta_n \, |f(n,m)| \\
      & \leq \sum_{|m| \leq |k|}
               c_m \, \| f(\bullet, m) \|_{\lebesgue^1(\nu_1)}
        \lesssim \sum_{|m| \leq |k|}
                   c_m \, \| f(\bullet,m) \|_{\lebesgue^p(\nu_1)}
        < \infty .
    \end{split}
    \label{eq:ComplexKernelOperatorWellDefined}
  \end{equation}
  Here, we used that $\| \bullet \|_{\lebesgue^1(\nu_1)} \lesssim \| \bullet \|_{\lebesgue^p(\nu_1)}$
  since $\nu_1$ is a finite measure.
  The estimate \eqref{eq:ComplexKernelOperatorWellDefined} shows that
  $\Phi_K f : [0,1] \times \Z \to \CC$ is well-defined and also that
  ${\Phi_{|K|} |f| (\bullet, k) \in \lebesgue^\infty([0,1])}$.
  Finally, it shows that $\lebesgue^{p,q}(\nu) \to \CC, f \mapsto (\Phi_K f) (x,k)$
  is a bounded linear functional for arbitrary $(x,k) \in X$.

  Next, note that
  \[
    C_1 (K)
    = \esssup_{(x,k) \in [0,1] \times \Z}
        \int_{\Z^2}
          \bigl| K \bigl( (x,k), (n,m) \bigr) \bigr|
        \, d \nu(n,m)
    \leq \sum_{n \in \Z}
           \beta_n
           \sum_{|m| \leq |n|} c_m
    = \sum_{n \in \Z}
        (1 + n^2)^{-1}
    < \infty ,
  \]
  as well as
  \begin{align*}
    C_2 (K)
    & = \esssup_{(n,m) \in \Z^2}
          \int_{[0,1] \times \Z}
            \big| K \big( (x,k), (n,m) \big) \big|
          \, d \mu(x,k) \\
    & \leq \sup_{m \in \Z}
             \int_{\Z}
               \int_{[0,1]}
                 c_m
               \, d x
             \, d \mu_2(k)
      \leq \| c \|_{\ell^\infty} \cdot \mu_2 (\Z)
      <    \infty ,
  \end{align*}
  and
  \[
    C_4 (K)
    = \esssup_{m \in \Z}
        \int_{\Z}
          \esssup_{x \in [0,1]}
            \int_{\Z}
              \big|
              K \big( (x,k), (n,m) \big)
              \big|
            \, d \nu_1 (n)
        \, d \mu_2(k)
    \leq \| c \|_{\ell^\infty} \cdot \nu_1(\Z) \cdot \mu_2 (\Z)
    <    \infty .
  \]
  In view of Theorem~\ref{thm:SchurTestSufficientUnweighted}, this implies that
  $\Phi_K : \lebesgue^{p,q}(\nu) \to \lebesgue^{p,q}(\mu)$ is well-defined and bounded
  for all $p,q \in [1,\infty]$ with $p \geq q$, and in particular for $p = q = 1$, $p = q = \infty$
  and for $(p,q) = (\infty, 1)$.

  Next, for $f \in \lebesgue^{1,\infty}(\nu)$, $k \in \Z$ and $g \in \lebesgue^2 ([0,1])$, we have
  \begin{equation}
    \begin{split}
      \bigg|
        \int_0^1
          \Phi_K f (x,k) \, g (x)
        \, d x
      \bigg|
      & = \bigg|
            \sum_{n \in \Z} \,\,
              \sum_{|m| \leq \min \{ |k|, |n| \}} \!\!\!
                c_m \,
                \beta_n \,
                f (n,m) \,
                \int_0^1
                  g(x) \, e^{- 2 \pi i m x}
                \, d x
          \bigg| \\[0.2cm]
      & \leq \sum_{m \in \Z}
               c_m \, |\widehat{g}(m)|
               \sum_{n \in \Z}
                 \beta_n \, |f(n,m)|
        \leq \| f \|_{\lebesgue^{1,\infty}(\nu)}
             \cdot \| c \|_{\ell^2 (\Z)}
             \cdot \| \widehat{g} \|_{\ell^2(\Z)} \\
      & =    \| f \|_{\lebesgue^{1,\infty}(\nu)}
             \cdot \| c \|_{\ell^2 (\Z)}
             \cdot \| g \|_{\lebesgue^2 ([0,1])} ,
    \end{split}
    \label{eq:ComplexKernelOperatorBoundedness}
  \end{equation}
  where we used Plancherel's theorem $\| \widehat{g} \|_{\ell^2(\Z)} = \| g \|_{\lebesgue^2}$
  in the last step.
  The interchange of the series with the integral above can be justified using the dominated
  convergence theorem and the estimate in \eqref{eq:ComplexKernelOperatorWellDefined}.
  Since we know from above that
  $\Phi_K f (\bullet,k) \in \lebesgue^{\infty}([0,1]) \subset \lebesgue^2([0,1])$,
  the estimate \eqref{eq:ComplexKernelOperatorBoundedness} easily implies
  \(
    \| \Phi_K f (\bullet,k) \|_{\lebesgue^1}
    \leq \| \Phi_K f (\bullet,k) \|_{\lebesgue^2}
    \leq \| f \|_{\lebesgue^{1,\infty}(\nu)} \cdot \| c \|_{\ell^2 (\Z)},
  \)
  which shows that ${\Phi_K : \lebesgue^{1,\infty} (\nu) \to \lebesgue^{1,\infty}(\mu)}$
  is bounded, as claimed.
  Now, a version of the Riesz-Thorin interpolation theorem for mixed Lebesgue-spaces
  (see \cite[Section~7, Theorem~2]{MixedLpSpaces}) shows that
  $\Phi_K : \lebesgue^{p,q}(\nu) \to \lebesgue^{p,q}(\mu)$
  is in fact well-defined and bounded for all $p,q \in [1,\infty]$.

  Finally, to see $C_3 (K) = \infty$, recall that $\nu_2$ is the counting measure on $\Z$,
  so that
  \[
    C_3 (K)
    = \esssup_{k \in \Z}
        \int_{\Z}
          \esssup_{n \in \Z}
            \int_0^1
              \big| K \big( (x,k), (n,m) \big) \big|
            \, d x
        \, d \nu_2 (m)
    = \esssup_{k \in \Z}
        \sum_{|m| \leq |k|}
          c_m
    = \| c \|_{\ell^1}
    = \infty.
  \]
\end{example}

\let\section\origsection
\markleft{References}
\markright{}

\bibliographystyle{abbrv}
\bibliography{addbib,felixbib}

\end{document}